\documentclass[12pt]{article}
\usepackage[utf8]{inputenc}
\date{\today}
\usepackage{amsthm}
\usepackage{tabularx}
\usepackage{ upgreek }
\usepackage[normalem]{ulem}
\usepackage{amsmath,amssymb} 
\usepackage{caption}   
\usepackage{xfrac}
\usepackage{caption}
\usepackage{graphicx} 
\usepackage{mathrsfs}
\usepackage{array}
\usepackage{booktabs}
\usepackage{enumitem}
\usepackage{xcolor}
\usepackage[normalem]{ulem}
\usepackage{mdframed} 
\usepackage{amsmath}

\usepackage{algorithm}
\usepackage{algcompatible}
\usepackage{tablefootnote}
\usepackage{algpseudocode}
\usepackage[utf8]{inputenc}
\usepackage[UKenglish]{babel}
\usepackage[nottoc,numbib]{tocbibind}
\usepackage{mathtools,amssymb,amsthm}
\usepackage{adjustbox,graphicx}
\usepackage{mathtools}
\usepackage{hyperref}
\usepackage{bbm}
\usepackage{lineno}
\usepackage{float}
\usepackage{tikz}
\usetikzlibrary{shapes.geometric, arrows}
\usepackage{csquotes}
\usepackage[sorting=none]{biblatex}
\AtBeginBibliography{\small}
\usepackage{amsthm}
\newtheorem{prop}{Proposition} 
\theoremstyle{definition}
\newtheorem{lem}[prop]{Lemma}

\newtheorem{thm}[prop]{Theorem}
\newtheorem{cor}[prop]{Corollary}
\newtheorem{defn}[prop]{Definition}
\newtheorem{rmk}[prop]{Remark}
\newtheorem{exam}[prop]{Example}

\usepackage[normalem]{ulem} 



\newcommand{\IE}{\mathbb{E}}

\newcommand{\indep}{\perp\!\!\!\perp}

\newmdenv[
  backgroundcolor=yellow!20,
  linecolor=yellow!50!black,
  linewidth=2pt,
  roundcorner=5pt
]{highlightblock}

\title{The Unseen Species Problem Revisited}
\date{}
\author{Edward Eriksson\thanks{Max Planck Institute for Mathematics in the Sciences, Leipzig. \texttt{Edward.Eriksson@mis.mpg.de}}}
\begin{document}
\maketitle

\begin{abstract}
Given $n$ i.i.d. samples from an unknown discrete distribution over an unknown set, the unseen species problem is to predict how many new outcomes would be observed in $m$ additional samples. For small $m$ we show that the Good--Toulmin estimator is the unique estimator which both respects the symmetries of the problem and has non-trivial rate. We resolve the open problem of constructing principled prediction intervals for it. For intermediate $m$ we propose a new estimator which has vastly improved worst case MSE guarantees compared to competing methods and good empirical performance. For large $m$ we follow previous authors in assuming a power law tail and show that a simple estimator achieves the same rate as, and better empirical performance than, a recent sophisticated method. Moreover, we give pre-asymptotic guarantees and asymptotically calibrated prediction intervals.
  
Many of our results extend to incidence data, without further independence assumptions, provided that the sets are of bounded size. Using Stein's method we obtain concentration inequalities for some natural functionals of sequences of i.i.d. discrete-set-valued random variables which are of independent interest.
\end{abstract}
\section{Introduction}

In the early 1940s the British naturalist Alexander Corbet spent two years trapping butterflies in British Malaya \cite{orlitsky_2016_optimal}. Upon his return to England he consulted with R.A. Fisher as to how many new species he may expect to discover on a return expedition \cite{orlitsky_2016_optimal}. Making such a prediction has become known as the unseen species problem. With access to a good estimate one could, for example, make a more informed decision about whether a return trip is worthwhile.

The return expedition need not be of the same length as the first expedition. In the limit of an infinitely long return expedition the unseen species problem becomes the classical support size estimation problem. At the other extreme, namely a return expedition consisting of capturing only a single specimen, it becomes the classical missing mass problem in which one tries to estimate how much probability mass is associated to as of yet unobserved outcomes. In this way the unseen species problem \enquote{interpolates} between these classical problems, each arising as an extremal version of the unseen species problem. As we feel the two are sometimes confused, we spend a few more sentences on distinguishing the unseen species problem from support size estimation. Support size estimation is asking \enquote{how much \emph{could} one discover?} and the unseen species problem is asking \enquote{how much \emph{will} we discover?}. For example, in the context of software code \cite{bhme_2018_stads} support size estimation is estimating the number of bugs that \emph{exist}, while the unseen species problem is the  question of determining how many bugs will be \emph{encountered} in some finite time interval. As another example, there were a large number of body fragments collected in the aftermath of the 9/11 attacks, so many in fact that as of the writing of \cite{adams_2022_victim} the identification effort was still ongoing. If one imagines that each fragment is sampled i.i.d. from some distribution over the victims, then the support size is the number of victims, which is known by other means \cite{adams_2022_victim}. The task carried out in that paper, namely to predict the number of victims who are represented in the finite number of recovered fragments, is the unseen species problem.

\subsection{The Generalized Unseen Species Problem} \label{subsec: problem statement}

Let $\mathcal{S}$ be the unknown, but presumed at most countable, set of species. Following \cite{efron_1976_estimating} we let $n,m$ be Poisson random variables with means $t,t'$, respectively\footnote{This so-called \enquote{Poissonization} is mostly done for technical reasons, but at least in the classical setting one may view it as fixing the duration of the trips ($t,t'$) rather than the number of specimens caught ($n,m$), which is quite natural. In any case one expects it not to make a significant difference since Poisson random variables are well-concentrated.}. In the classical unseen species problem one takes the observations $X_1,\ldots,X_n$ to be i.i.d. samples from an unknown $\mu \in \mathcal{P}(\mathcal{S})$, where $\mathcal{P}(\cdot)$ denotes the set of probability measures over a set, and attempts to predict how many new outcomes will be observed in the next $m$ samples. We consider the following generalization: Let the data consist of $X_1,\ldots,X_n$, i.i.d. samples from an unknown $\mu\in\mathcal{P}\left(\mathcal{T}(\mathcal{S})\right)$, where $\mathcal{T}(\mathcal{S})$ denotes the set of finite multisets over $\mathcal{S}$. For each species $s$, let $N_s$ be the number of times it was observed, counted with multiplicity. Let $\mathcal{S}_t := \{s \in \mathcal{S}: N_s > 0\}$ denote the set of species observed until time $t$, $S_t:=|\mathcal{S}_t|$ and define $T:=t+t'$. We then call the task of predicting
$$S_{t,T}:=\left|\mathcal{S}_T\setminus\mathcal{S}_t\right|=S_T-S_t,$$
the generalized unseen species problem.

The classical unseen species problem is recovered by insisting that $\mu$ is supported on singletons. All the results which we state for the generalized problem with sets of size at most one also hold for the classical problem, except for the value of the constant in Theorem \ref{thm: opt powerhouse}, but for convenience we do not state this each time. We believe our model better captures the dependence structure in the problem, which we illustrate with some examples.
\begin{exam}
    When collecting butterflies (or maybe even more so when observing fish as in \cite{chao_2017_seen}) one expects clustering of members of the same species and of species which enjoy the same micro-environments. It is therefore much more plausible to believe that every day a set of butterflies is caught and that independence arises between \emph{days} rather than \emph{specimens}.
\end{exam}
\begin{exam}

To sequence a genome from a single cell one prepares a \enquote{library} of DNA fragments which haphazardly cover different parts of the target genome. Then a machine reads those fragments. One would like to predict from a small number of fragments how much of the genome could be recovered by reading many fragments from that library as then one could compare protocols for preparing libraries without having to do (expensive) deep sequencing experiments \cite{daley_2014_modeling}. Rather than \emph{genomic locations} being independent, it is much more plausible that \emph{fragments} are independent. 
\end{exam}
\begin{exam}
    Suppose that an intelligence agency is spying on a terrorist organization by means of monitoring phone calls and is facing a decision: either arrest the known terrorists at risk of sending the rest into hiding, or wait to discover more terrorists. Clearly  being able to forecast the number of new terrorists that will be discovered by waiting is key to making an informed decision. The usual model would have one believe that who is on one end of a phone call is independent of who is at the other end. It is much more plausible to believe that \emph{calls} rather than \emph{callers} are independent. 
\end{exam}
\begin{exam} \label{exam: mtg}
    In many trading card games customers can buy small packages (\enquote{packs}) of randomized cards. The customer may wish to forecast the growth of their collection as they buy such packs, which is closely related to the famous coupon collector's problem. The number of cards in a pack is known and fixed but the exact contents are random and unknown. The cards are not independent as various constraints are imposed. For example, the presence of precisely one rare card, roughly equal presence of cards of various types etc. Hence it is much more plausible to believe that \emph{packs} rather than \emph{cards} are independent.
\end{exam}
That data is often structured in ways resembling the above examples is well-known to biologists who refer to it as incidence data \cite{chao_2014_rarefaction} and to statisticians who speak of feature allocation models \cite{ayed_2019_a} or edge exchangeable random graphs \cite{crane_2018_edge}. Edge exchangeable random graphs are a way to encode the data as a multi-hypergraph by setting the observed species as the vertices and the observed sets as the edges. For estimation problems resembling those discussed here the only analyses we are aware of assume that membership in each set is independent for each species \cite{chao_2014_rarefaction,ayed_2019_a}, essentially assuming away precisely the dependence we are trying to capture. We will often assume that the sets are of bounded size, but make no such within-set independence assumption. Owing to the de Finetti-type result of \cite{crane_2018_edge} one could assume that the data is an exchangeable sequence of sets rather than an independent one, with suitable modifications to our results. Concretely, the de Finetti-type result says, essentially, that weakening  independence corresponds to allowing $\mu$ itself to be random. However, many of our bounds are worst case bounds which automatically hold also for random $\mu$.

\subsubsection{Reduction to Multiplicity-Free Case}
When we observe multisets we may forget about the duplicates without changing when we discover what species, in particular, without changing $S_{t,T}$. Thus in the remainder it suffices to consider $\mu \in \mathcal{P}(2^\mathcal{S})$ (still supported on finite sets). Perhaps it appears that we have robbed ourselves of some information that our estimators\footnote{Following previous authors we use the term \enquote{estimator} even though \enquote{predictor} is perhaps more accurate.} could depend on, but note that starting with a $\mu \in \mathcal{P}(2^\mathcal{S})$ we can add multiplicities arbitrarily without changing when what is discovered so the duplicates cannot, in general, be informative.

\begin{rmk} As observed in \cite{chao_2017_seen}, while such a reduction is always possible it is sometimes essentially unavoidable. Therein a data set of tropical fish is considered. The data is collected by groups of divers submerging and then, upon their return to the surface, reporting what they observed. As noted in \cite{chao_2017_seen} there is a substantial risk that the same specimen was seen by multiple divers causing essentially \enquote{illusory} duplicate observations and that counting the number of fish in a school may in any case be hopeless.
\end{rmk}

\subsubsection{Bounded Arity}
One readily sees that the generalized unseen species problem is not feasible in full generality. 
\begin{exam} \label{cave exam}
    Suppose all animals live in a single cave such that either none or all of them are discovered. This corresponds to a $\mu$ which assigns positive probability only to a single large set. By making this set arbitrarily large and tuning the probability of discovery we can clearly make the prediction problem arbitrarily hopeless.
\end{exam}
Thus some extra assumption is needed. The following will be of central importance to us.
\begin{defn}
    We say that $\mu$ is of \textbf{$B$-bounded arity} and write $\mu \prec B$ if $\mu(A) = 0$ for all $A \in 2^\mathcal{S}$ such that $|A| >B$. We say that $\mu$ is of \textbf{bounded arity} if it is of $B$-bounded arity for some $B \in \mathbb{Z}_{\geq 1}$.
\end{defn}
In practice $B$ may or may not be known. For example, the genome fragments consist of $B=100$ base pairs. Phone calls are between $B=2$ people. The card game Magic: The Gathering (MTG) has packs of $B=15$ cards. On the other hand, when collecting butterflies $B$ is not known, insofar as it may be said to exist at all. When we wish to disallow assigning probability to the empty set we write $1 \prec \mu$.

\subsection{Previous Work and Our Contribution} 
Different methodologies are suitable for different ranges of $r := \frac{t'}{t}$. To reflect this we speak of the near, intermediate and distant future, corresponding roughly to $r \leq 1$, $1<r \lesssim \log(t)$ and $\log(t) \lesssim r \lesssim \exp(t^\frac{\alpha}{2})$ for a constant $\alpha \in (0,1)$, respectively. The rest of the subsection will be split by regime, detailing pre-existing work and our contribution. Before describing the literature and our contribution more accurately, we give a crude and necessarily oversimplified overview in Table~\ref{tab:results}. 

\begin{table}[H]
\centering
\begin{tabular}{c|>{\raggedright\arraybackslash}p{2cm}|>{\raggedright\arraybackslash}p{2.3cm}|>{\raggedright\arraybackslash}p{2cm}}
   & Near & Intermediate & Distant \\ \hline
  Classical & 
      EB $\star$
      
      RO $\checkmark^\prime$
      
      ED $\star$
  &  
  EB $\checkmark^\prime$
  
  RO $\checkmark^\prime$
  
  ED $\star$
  & 
  EB $\star$
  
  RO $\checkmark^{{(\prime)}}$
  
  ED $\star$
  \\ \hline
  Generalized &   EB $\star$
  
  RO $\star$
  
  ED $\times$ &   EB $\star$
  
  RO $\star$
  
  ED $\times$ &   EB $\star$
  
  RO $\star$
  
  ED $\times$ \\ \hline
\end{tabular}

\caption[Overview of relevant results]{%
  \footnotesize
  Overview of relevant results, categorized in three prediction ranges
  (near, intermediate, distant) and two set-ups (classical, generalized).

  EB: Explicit, non-asymptotic error bounds, RO: Rate and optimality,
  ED: Characterizing the limiting distribution of the error.
  $\checkmark$: Pre-existing result.
  $\checkmark^\prime$: We add to or improve on a pre-existing result.
  $\checkmark^{(\prime)}$: We improve on an existing result in a very weak sense. 
  $\star$: Our contribution.
  $\times$: Open problem.\protect\footnotemark
}
\label{tab:results}
\end{table}
\footnotetext{Even showing that $S_t$ is asymptotically Gaussian, never mind the error in
predicting it, is an open problem, even when the sets are of size two. See problem 6.9 of \cite{janson_2017_on} and progress in \cite{eriksson_2025_edge}.}

\subsubsection{Near Future}

For near future prediction the classical Good--Toulmin estimator \cite{good_1956_the} (GTE),

$$\hat{S}_{t,T}^{(\text{GT})}:= -\sum_{s \in \mathcal{S}}1_{N_s > 0}(-r)^{N_s},$$
has long been the dominant approach. We obtain explicit constants in the first part of Theorem 11 of \cite{polyanskiy_2019_dualizing}, thereby showing that the GTE is optimal to within a factor $1+r \leq 2$ from the worst case MSE achievable by any eligible estimator. Actually, we show something more delicate: that when moving to the generalized unseen species problem not only does the GTE not suffer a rate degradation in $t$, it even retains the expected rate in $\IE[S_t]$. We show that no estimator which respects some natural problem symmetries other than precisely the GTE achieves a rate which is better than that of the estimator which always predicts zero, henceforth referred to as the null estimator. For $\mu \prec 1$ we define a variance proxy and show that a central limit theorem holds with this variance proxy, settling the open problem \cite{contardi_2025_gaussian} of constructing principled prediction intervals. Practitioners \cite{adams_2022_victim} have previously estimated $\mu$ and then simulated the unseen species problem on $\hat{\mu}$ \cite{chao_2014_rarefaction}, which lacks theoretical guarantees. 

\subsubsection{Intermediate Future}
In the intermediate future the variance of the GTE is so large it is unusable and one must seek other estimators. The most well-known alternatives are the Smoothed Good--Toulmin estimators (SGTE) proposed by \cite{orlitsky_2016_optimal}, which encompass the Efron--Thisted estimator of \cite{efron_1976_estimating}. They are linear estimators (see Definition \ref{defn: linear}) which work by averaging over stochastic truncations of the GTE. Specifically, if $L$ is a non-negative-integer-valued random variable then
$$\hat{S}_{t,T}^{(\text{SGT})} := -\sum_{s \in \mathcal{S}}1_{N_s > 0}P(L \geq N_s)(-r)^{N_s},$$
is the SGTE associated to (the distribution of) $L$. Define the normalized risk,
\begin{align*}
    \mathcal{R}_B(\hat{S}_{t,T},S_{t,T}) := \sup_{\mu \prec B} \frac{\IE[|\hat{S}_{t,T}-S_{t,T}|^2]}{r^2t^2}.
\end{align*}

The best known choice of the distribution of $L$ is a binomial distribution with $\left\lfloor \frac{1}{2} \log_3 \frac{tr^2}{r-1} \right\rfloor$ attempts and success probability $\frac{2}{2+r}$ \cite{orlitsky_2016_optimal} and henceforth when discussing the SGTE, it is always with this particular choice. The SGTE achieves a rate, in $\mathcal{R}_1$, of $t^{-\log_3\left(1+\frac{2}{r}\right)}$ \cite{orlitsky_2016_optimal}. The same authors also showed that no estimator can achieve a rate better than $t^{-c/r}$ for some constant $c$, thereby characterizing the prediction horizon as $r(t)$ growing logarithmically in $t$ \cite{orlitsky_2016_optimal}. Polyanskiy and Wu then showed that for fixed $r$, the optimal rate is between $t^{-\frac{2}{r+1}}\log^{-4}(t)$\footnote{The pre-print currently says $\log^{-2}(t)t^{-\frac{2}{r+1}}$, but we believe this is a typo.} and $t^{-\frac{2}{r+1}}\log^4(t)$ \cite{polyanskiy_2019_dualizing} and in so doing provided an estimator with the latter rate, which we denote by $\hat{S}_{t,T}^{(\text{PW})}$.
The construction of $\hat{S}_{t,T}^{(\text{PW})}$ involves, in some sense, a filtering step on top of a candidate linear estimator which is defined as the solution to a certain linear programming problem. Their estimator is a valuable theoretical contribution but has some practical problems. It suffers from large pre-factors which render any pre-asymptotic guarantees useless. In particular, it suffers from  the $\log^4(t)$ factor and quadratic dependence on a \enquote{large enough} constant $C_0$. For their proof to go through verbatim one needs $C_0 > \frac{11}{\log(2)-\frac{1}{2}} \approx 57$, but this is easily optimized to $C_0 > \frac{1}{2\log(2)-1} \approx 2.59$, which is the value we use. It is non-linear, which we argue is undesirable in light of the symmetries of the problem, and in particular fails to be invariant under permutations of the (exchangeable) data. In addition to this being concerning on theoretical grounds, this means that it cannot be applied when only the final counts and not the entire history of the discovery process is available. This is the case, for example, for Corbet's data \cite{fisher_1943_the}.

 We propose a linear estimator of our own, defined as the minimizer of a certain convex functional which sandwiches the MSE. Our estimator has much better pre-asymptotic guarantees than even the SGTE, better asymptotic guarantees than $\hat{S}_{t,T}^{(\text{PW})}$ and has good empirical performance. We show our estimator to be optimal, pre-asymptotically, to within an explicit multiplicative constant, in the class of linear estimators. We argue that the class of linear estimators is very natural in light of the symmetries of the problem. Independently, we show a rate improvement of $\log^{6-\frac{1}{1+r}}(t)$ compared to \cite{polyanskiy_2019_dualizing}. The principal difference between our approach to constructing an estimator as compared to that of \cite{polyanskiy_2019_dualizing} is that we reduce the construction to an optimization problem which cannot be solved by linear programming. In that it obtains an estimator by minimizing a convex functional our approach resembles \cite{convex_stats}, but they work with \enquote{$\epsilon$-risk} guarantees rather than MSE ones, our problem is not of the type they prove guarantees for, and should one insist on using their approach anyway this would involve minimizing a different, although not unrelated, functional.  

We give a central limit theorem with a bias term and a variance proxy for linear estimators when $\mu \prec 1$. Since the estimators are in general biased this does not immediately allow the construction of principled prediction intervals. However, we find that for our estimator applied to natural distributions, good coverage can be achieved by simply dropping the bias term.

\subsubsection{Distant Future}
In light of Theorem 2 of \cite{orlitsky_2016_optimal} some extra assumption is needed to make predictions beyond logarithmically growing $r$. Following previous authors, we consider a regular variation (power law) assumption. The consequences of making such an assumption in the unseen species problem are briefly mentioned in \cite{benhamou_2017_concentration} and subsequently explored in detail in \cite{favaro_2023_nearoptimal}. From these works it is known that if one obtains an accurate estimator $\hat{\alpha} \in (0,1)$ for the regular variation index then \begin{align}\label{induced estimator} 
    \hat{S}_{t,T} := S_t\big((1+r)^{\hat{\alpha}}-1\big),
\end{align} is an accurate estimator in the unseen species problem. The prediction problem is thus reduced to a parameter estimation problem. The authors of \cite{favaro_2023_nearoptimal} gave an estimator for $\alpha$, motivated by certain maximum likelihood considerations. They showed that the resulting estimator for the unseen species problem achieves the minimax-rate up to polylog factors. With $\phi_1$ being the number of species seen precisely once, we suggest using $\hat{\alpha}:= \frac{\phi_1}{S_t}$ instead. The consideration of this ratio is certainly not new. In fact Corbet's writings already mention that $\frac{\phi_1}{S_t}$ appears to be constant in time \cite{corbet_1956_the}. The most detailed treatment of it which we are aware of is \cite{chebunin_2018_asymptotically}, which builds on a long line of research on the infinite urn scheme, see \cite{Gnedin_notes} for a (somewhat outdated) review. Even the use in the unseen species problem is not new as this was mentioned and shown to give rise to a multiplicatively consistent unseen species estimator in \cite{benhamou_2017_concentration}. However,
 the rate for the resulting unseen species estimator was only known for the more modern estimator of \cite{favaro_2023_nearoptimal}. Trying to work directly with the estimator of \cite{favaro_2023_nearoptimal} in the generalized problem would have been quite difficult since it depends on the data in a complicated way. Our insight is that one can apply the work of \cite{bartroff_2018_bounded} on Stein's method with bounded size biased couplings for configuration models to  get concentration inequalities for, among other quantities, $\phi_1$ and $S_t$ and from this give pre-asymptotic tail bounds on the probability of the prediction error exceeding a given amount. Our pre-asymptotic result implies a rate guarantee which is at least as good as that of \cite{favaro_2023_nearoptimal} but covers the generalized version of the problem, with $\mu$ of bounded arity.

Equivalently the concentration inequalities give concentration for the number of vertices in an edge exchangeable graph and for the number of vertices of a given degree. Such quantities were of interest in \cite{janson_2017_on}. We also expect them to be useful for studying other species sampling problems with incidence data. For example, $\phi_1$ plays a central role in the missing mass problem \cite{Good_1953}.

For $\mu \prec 1$ we give a CLT with a variance proxy, allowing the construction of principled prediction intervals. On the Bayesian side it is known that in the classical problem, under the regular variation assumption, the Pitman--Yor process serves as a reasonable prior and that Gaussian credible intervals can be constructed for this approach \cite{contardi_2025_gaussian}.  

\subsection{Structure} Section \ref{sec: symmetries} discusses some symmetries it would be desirable for the estimators to respect.  Then the next three sections, \ref{sec: near}, \ref{sec: inter}, \ref{sec: distant} deal with the near future, intermediate future and distant future, respectively. Finally we compare the estimators empirically in Section \ref{sec: comparison}.

Proofs are relegated to the appendix but selected proof sketches, most applying only to $\mu \prec 1$, are provided in the main text. 

\section{Linear Estimators and Problem Symmetries} \label{sec: symmetries}
\begin{defn} \label{defn: linear}
    We say that an estimator is \textbf{linear} if there exists a function $H: \mathbb{Z}_{\geq 0} \to \mathbb{R}$, in general depending on $r$ and $t$ such that,
    \begin{align*}
    \hat{S}_{t,T} = \hat{S}^{(H)}_{t,T} := \sum_{s \in \mathcal{S}}H(N_s).
    \end{align*}
\end{defn}
For a linear estimator to be eligible we need $H(0)=0,$ as we don't know how many unseen species there are and so we cannot have them contribute. For $i \in \mathbb{Z}_{>0}$, let $\phi_i := \sum_{s \in \mathcal{S}}1_{N_s = i}$ be the number of species observed exactly $i$ times and encode the sequence as a vector $\boldsymbol{\phi}:=(\phi_1,\phi_2,\ldots)$. Notice that $\sum_{s \in \mathcal{S}}H(N_s) = \sum_{i=1}^\infty H(i)\phi_i$, so the estimator is linear in the $\phi_i$'s, justifying the name. We will often think of $H$ as a sequence and consequently write $H_i$ in place of $H(i)$. 
\begin{defn} We say that an estimator has the \textbf{synchronous expeditions property} (SEP) if it is a function of $r,t$ and $\boldsymbol{\phi}$ and \begin{align*}
        \hat{S}_{t,T}\left(\boldsymbol{\phi}\right)+\hat{S}_{t,T}(\tilde{\boldsymbol{\phi}}) = \hat{S}_{t,T}(\boldsymbol{\phi}+\tilde{\boldsymbol{\phi}}),
    \end{align*}
for all $\boldsymbol{\phi},\tilde{\boldsymbol{\phi}}$.
\end{defn}
To motivate the SEP, imagine that two expeditions set out to two locations with disjoint sets of species. The expeditions spend an amount of time $t$ at their destinations and then, on their way back, each expedition makes a prediction as to how many new species they would discover on a return voyage. Imagine also that once they rendezvous they jointly make a prediction using the combined data as to how many new species will be discovered by the return trips. Since the sets of species were presumed disjoint, the number of such discoveries will be equal to the sum of the number of discovered species by each expedition and the SEP is essentially that the sum of the predictions should satisfy the same equality. Although it is not directly applicable, to see why it is natural to consider estimators that depend on the data only through $\boldsymbol{\phi}$, refer to Theorem 2 of \cite{hao2020}. In a mild abuse of notation, we write SEP for the class of estimators with the SEP.

\begin{defn}
    We say that an estimator is \textbf{representation invariant} (RI) if 
    it depends on the data only through $\boldsymbol{\phi}$ and 
    $$\hat{S}_{t,T}(B\boldsymbol{\phi}) = B\hat{S}_{t,T}(\boldsymbol{\phi}),$$
    for all $\boldsymbol{\phi}$ and $B \in \mathbb{Z}_{\geq 1}$.
\end{defn}
To motivate representation invariance, imagine that we are sampling matching pairs of shoes. We can either see this as the classical unseen species problem with an individual being a pair of shoes or as the generalized unseen species problem with shoes as individuals and all sets of size two (one left and one right shoe, which we consider distinct). Representation invariance is then an enforcement of consistency between these two views. We frequently use the following construction to move lower bounds from the classical to the generalized problem. Given a $\mu \prec 1$ on $\mathcal{S}$ one may associate a $\mu \prec B$ on $\mathcal{S} \times \{1,\ldots,B\}$ by $\mu(\{(s,1),(s,2),\ldots,(s,B)\}) := \mu(\{s\})$ for all $s \in \mathcal{S}$. We call this speciating $\mu$. Note that this scales $\boldsymbol{\phi}$ and $S_{t,T}$ by a factor $B$. Thus, if using a representation invariant estimator the MSE is scaled by $B^2$.

Linear estimators play very well with the symmetries of the problem.
\begin{prop}\label{prop:symmetries} An estimator has the SEP if and only if it is linear. Linear estimators are representation invariant.
\end{prop}
It is natural for estimators to not depend (explicitly) on $t$.
\begin{defn}
    We say that an estimator is \textbf{atemporal} if it depends only on the data and $r$.
\end{defn}
To motivate atemporality, pick some $c \in \mathbb{R}_{>1}$. Consider two worlds. One in which we are faced with the generalized unseen species problem with $r,t,\mu$ and an alternative world where we are faced with $r,ct,\mu'$ where for all $A \in 2^\mathcal{S}\backslash \{\emptyset\}$, $\mu'(A) = \frac{\mu(A)}{c}$ and $\mu'(\emptyset) = 1-\frac{1}{c}(1-\mu(\emptyset))$. Couple these problems in the natural way, obtaining a coupling such that at all times the numbers of observations of all species agree, as does $S_{t,T}$. Then using an estimator which is not atemporal would mean making different predictions in these worlds. We write AT for the class of atemporal estimators.

\section{Near Future}\label{sec: near}

\subsection{Optimality of the Good--Toulmin Estimator}
For $x,y \in \mathcal{S}$ let
\begin{align*}
    M_{x \cap y} &:= \sum_{A \in 2^\mathcal{S}: A \supset \{x,y\} }\mu(A)
     &M_{x \cup y}:=  \sum_{A \in 2^\mathcal{S}: A \cap \{x,y\} \neq \emptyset }\mu(A) \\
    M_{x} &:= \sum_{A \in 2^\mathcal{S}: x \in A}\mu(A),
     &M_{x \backslash y}:= \sum_{A \in 2^\mathcal{S}: x \in A,y \notin A} \mu(A).
\end{align*}
We adopt the convention that $\sum_{x,y}$ means the sum over ordered pairs of distinct elements of $\mathcal{S}$. 
Let 
\begin{align*}
    \delta(t) :=& \IE[S_{t,T}+\sum_{s \in \mathcal{S}}1_{N_s >0}r^{2N_s}] \\
\epsilon(t) :=& \sum_{x,y}e^{-(1+r)M_{x \cup y}t}(e^{r(r+1)M_{x \cap y}t}-1).
\end{align*}
We may then decompose the error in using the Good--Toulmin estimator in the generalized unseen species problem as one term which depends only on the marginals and one term which measures the dependence between the species as follows.

\begin{prop} \label{prop: error decomposition}
    \begin{align*}
        \IE[|S_{t,T}-\hat{S}_{t,T}^{(\text{GT})}|^2] = \delta(t)+\epsilon(t).
    \end{align*}
\end{prop}

\begin{thm}\label{thm: GT analysis}If $\mu \prec B, r \leq 1$ then,
    \begin{align*}
        \IE[|S_{t,T}-\hat{S}_{t,T}^{(\text{GT})}|^2] &\leq (Br(r+1)-r)\IE[S_t]+\IE[S_{t,T}] \leq B^2r(r+1)t, \\
    \inf_{\hat{S}_{t,T}} \sup_{\mu \prec B}\IE[|S_{t,T}-\hat{S}_{t,T}|^2] &\geq B^2rt, \\
    \inf_{\hat{S}_{t,T} \in \text{SEP}}\sup_{\mu \prec B}\IE[|S_{t,T}-\hat{S}_{t,T}|^2] &\geq B^2\left(tr+\frac{r^2t^2}{t+1}\right), \\
    \sup_{\mu \prec B}\IE[|S_{t,T}-\hat{S}_{t,T}|^2] & \gtrsim B^2t^2 \text{ for any $\hat{S}_{t,T}$} \in \text{SEP} \cap \text{AT} \setminus \{\hat{S}_{t,T}^{(\text{GT})}\},
\end{align*}
where the implicit constant in the last result depends on the (fixed) value of $r$ and on $\hat{S}_{t,T}$ but not on $B$. 
\end{thm} That is to say, for $r \leq 1$ the Good--Toulmin estimator is minimax optimal without class restriction to within a factor $1+r \leq 2$. Within the class of linear estimators it is optimal to within an additive offset of $B^2r^2\frac{t}{t+1}$ and within the class of estimators that enjoy the SEP and are atemporal, it is the unique estimator with rate better than that of the null estimator. For $B > 1$ we read the intermediate upper bound as saying that not only does the Good--Toulmin estimator not suffer a rate degradation in $t$ as we move to $B>1$, it doesn't even suffer a rate degradation in $\IE[S_t]$.
\begin{proof}[Proof sketch]
For the first lower bound we compare to using the conditional expectation as an estimator. For the second lower bound we compute how much error a linear estimator must incur on a distribution with many rare species. For the third lower bound, take some $\mu \prec 1$ and time $t$. Consider moving to $2t$ and building $\mu'$ by replacing each species by two new species, splitting the probability mass in half. The bias, which is a sum over functions depending on the product of $t$ and the observation probability, has each term duplicated under this transformation. Since this makes the bias grow linearly in $t$, the result follows. 
\end{proof}

\subsection{Uncertainty Quantification}

When $\mu \prec 1$, Theorem \ref{thm: CLT Intermediate} allows the construction of asymptotically calibrated prediction intervals. Let $\hat{V}_t := \sum_{s \in \mathcal{S}} 1_{N_{s}>0}r^{2N_s}+\hat{S}^{(\text{GT})}_{t,T}$. 
\begin{cor} \label{cor: GT Gaussianity} If $r \leq 1$, $\mu \prec 1$ and $\mathbb{V}[S_{t,T}-\hat{S}_{t,T}^{(\text{GT})}] \rightarrow \infty$ as $t\to\infty$ then  
     $$\frac{S_{t,T}-\hat{S}_{t,T}^{(\text{GT})}}{\sqrt{\hat{V}_t}} \overset{d}{\rightarrow} \mathcal{N}(0,1).$$
\end{cor}

Formally, as the estimated variance could be zero (or negative) some convention is needed for the ratio in the theorem to make sense. Similar concerns reoccur elsewhere in the paper. When the choice of convention is immaterial we make no mention of it.

It would be very desirable to be able to estimate $\epsilon(t)$, as this would form the basis of uncertainty quantification also in the generalized unseen species problem. We are not able to offer a genuinely satisfactory estimator, but the next three propositions give some suggestions. Let $N_{x \cup y}$ be the number of observations that include $x$ or $y$, $N_{x \cap y}$ the number of observations that include both $x$ and $y$, $N_{x \backslash y}$ the number of times $x$ has been observed without $y$ and $N_{x \otimes y}:= N_{x \backslash y}+ N_{y \backslash x}$.
\begin{prop}\label{prop: epsilon-estimator} Let $\hat{\epsilon}(t) := \sum_{x,y}(r^{2N_{x \cap y}}-(-r)^{N_{x \cap y}})(-r)^{N_{x \otimes y}}$. Then,
    $$\epsilon(t)=\IE[\hat{\epsilon}(t)].$$
\end{prop}
Unfortunately some numerics suggest that $\hat{\epsilon}(t)$ is not well-concentrated.

As expressed in the following proposition, if we assume a small multiplicative difference between the unions and the intersections we can take the number of pairs co-observed precisely once as a conservative estimator of $\epsilon(t)$.

\begin{prop}\label{Prop: epsilon-perfect-pairs}If $\left(1+\frac{r^2}{1+r}\right)M_{x \cap y} \leq  M_{x \cup y}$ for all distinct $x,y \in \mathcal{S}$ then $$\epsilon(t) \leq r(r+1)\IE\left[\sum_{x,y}1_{N_{x \cap y}=1}\right].$$
\end{prop}

\begin{prop}\label{prop: epsilon-connectedness} Let $D_{x,y}$ be the event that $x,y$ are co-discovered, namely the event that the first set which contains either $x$ or $y$ contains both and let $r \leq 1$. Then,
    $$\epsilon(t) \leq r \IE\Big[\sum_{x,y}1_{D_{x,y}}\Big].$$
\end{prop}
Note that $\sum_{x,y}1_{D_{x,y}}$ is not measurable with respect to the sigma algebra generated by the data, but that since likely most co-discoveries occur fairly early the number of co-discoveries up to time $t$ can serve as a reasonable approximation. Proposition \ref{prop: epsilon-connectedness} justifies the following heuristic: \emph{If the number of co-discovered pairs is small compared to the number of discovered species, then the dependence is negligible.}

\section{Intermediate Future} \label{sec: inter}

\subsection{Constructing Linear Estimators}

Following previous authors we aim to construct (linear) estimators suited to $r>1$. Naturally these will not be atemporal so as to avoid the rate rigidity result of Theorem \ref{thm: GT analysis}. The restriction to linear estimators is principally motivated by Proposition \ref{prop:symmetries} but, as we will see, the class is large enough to contain performant estimators. 

Our linear estimator will be defined as the minimizer of a certain functional which we now introduce. Let $g_H(x):=\sum_{i=1}^{\infty}\frac{H_i}{i!}x^i$ denote the exponential generating function of $H$ and let $H^2$ denote the element-wise square of $H$. 
\begin{align*}
    Y_b &:= \left(\max_{p \in [0,1]} \frac{e^{-pt}}{p}|1-e^{-rpt}-g_{H}(pt)|\right)^2, \\
    Y_v &:= \max_{q \in [0,1]} \frac{e^{-qt}}{q}(1-e^{-rqt}+g_{H^2}(qt)), \\
    G(H) &:=Y_b+Y_v,
\end{align*}
where both terms are extended continuously at zero. $G(H)$ characterizes the worst case MSE in using a linear estimator. 

\begin{thm}\label{thm: opt powerhouse} For $\hat{S}_{t,T}^{(H)}$ a linear estimator,
\begin{align*}
    \frac{B^2}{6}G(H) &\leq \sup_{\mu \prec B} \IE[|S_{t,T}-\hat{S}^{(H)}_{t,T}|^2] \leq B^2G(H). \\
    \frac{B^2}{34}G(H) &\leq \sup_{1 \prec\mu \prec B} \IE[|S_{t,T}-\hat{S}^{(H)}_{t,T}|^2] \leq B^2G(H).
\end{align*}
\end{thm}
\begin{proof}[Proof sketch]
\renewcommand{\qedsymbol}{}%
A species with probability mass $p$ contributes $e^{-pt}(1-e^{-rpt}-g_H(pt))$ to the bias so that $\frac{e^{-pt}|1-e^{-rpt}-g_H(pt)|}{p}$ can be viewed as the cost-effectiveness of creating species with probability $p$ when trying to maximize the absolute bias. Spending a total unit mass, the maximal cost-effectiveness bounds the maximal achievable bias. The second term is analogous, but comes from the variance. If the maximizers $p^*$ and $q^*$ happened to be reciprocals of integers then one could realize the bias bound by taking a uniform distribution on $\frac{1}{p^*}$ symbols, or the variance bound by taking a uniform distribution on $\frac{1}{q^*}$ symbols. Combining these would give a lower bound of $\frac{1}{2}G(H)$. As $p^*$ and $q^*$ need not be reciprocals of integers technical complications arise.
\end{proof}

Critically, $G(H)$ is convex and even strictly convex in the following sense.

\begin{prop}\label{prop: convexity} Let $\theta \in (0,1)$. For any sequences $H^{(1)},H^{(2)}$,
    \begin{align} \label{eq: convexity of GH}
        G(\theta H^{(1)}+(1-\theta)H^{(2)}) \leq \theta G(H^{(1)})+(1-\theta)G(H^{(2)}).
    \end{align}
Moreover, if $H^{(1)} \neq H^{(2)}$ and $\theta$ are such that $q^*$, an argmax in $Y_v$, associated to $\theta H^{(1)}+(1-\theta)H^{(2)}$ is non-zero and $G(\theta H^{(1)}+(1-\theta)H^{(2)}) < \infty$, then strict inequality holds in (\ref{eq: convexity of GH}).
\end{prop}

Convexity suggests feasibility of seeking a minimizer of $G(H)$ and the following result establishes that there is something to converge to.

\begin{thm} \label{thm: existence and uniqueness}
The functional $G(H)$ admits a minimizer. That is, for all $r,t>0$ there exists an $H^*$ such that $G(H^*) \leq G(H)$ for all other sequences $H$. Moreover, for all $r > 2$ there exists some $t_0(r)$ such that for all $t > t_0$ the minimizer is unique.
\end{thm}
\begin{proof}[Proof sketch]
The proof of existence is by the direct method in the calculus of variations. For uniqueness a supplementary argument is needed to account for the failures of strict convexity. This amounts to showing that if  $q^*=0$ then the $H$-values are so small that a large bias is unavoidable.
\end{proof}

In light of Theorem \ref{thm: existence and uniqueness} we may define an estimator $\hat{S}_{t,T}^{(H^*)}$, where $H^*$ is a minimizer of $G(H)$. None of our analysis requires uniqueness, although we expect it to hold more generally than established by Theorem \ref{thm: existence and uniqueness}.

\subsection{Minimax Optimality}
Pre-asymptotic optimality, up to multiplicative constants, in the class of linear estimators, is immediate from Theorem \ref{thm: opt powerhouse}.
\begin{cor}\label{cor: optimality of H^*} For any linear estimator $\hat{S}_{t,T}^{(H)}$,
\begin{align*}
\sup_{\mu \prec B}\IE[|S_{t,T}-\hat{S}_{t,T}^{(H^*)}|^2] &\leq 6\sup_{\mu \prec B}\IE[|S_{t,T}-\hat{S}_{t,T}^{(H)}|^2].
\\
    \sup_{1 \prec \mu \prec B}\IE[|S_{t,T}-\hat{S}_{t,T}^{(H^*)}|^2] &\leq 34\sup_{1 \prec \mu \prec B}\IE[|S_{t,T}-\hat{S}_{t,T}^{(H)}|^2].
\end{align*}
\end{cor}

The rate cannot be read from Corollary \ref{cor: optimality of H^*}.
Essentially by averaging $\hat{S}_{t,T}^{(\text{PW})}$ over permutations of the data one can use the results of \cite{polyanskiy_2019_dualizing} to obtain a linear estimator with rate $B^2t^{-\frac{2}{1+r}}\log^4(t)$, for fixed $r$ (see Proposition \ref{prop: distilling} in the appendix) thus obtaining the same guarantee for $\hat{S}_{t,T}^{(H^*)}$ without much work. However, we supply an independent sharper rate guarantee, proven with substantial aid of Anthropic's Claude Fable 5. 
\begin{thm}\label{thm: main rate} With $r(t) > 1$ such that for all $t$ sufficiently large,
\begin{align*}
    \frac{r-1}{r+1} \log(t) \geq 4\log\log(t),
\end{align*} we have
\begin{align*}
    \mathcal{R}_B(\hat{S}_{t,T}^{(H^*)},S_{t,T}) \lesssim B^2\min\left\{1,\frac{(1+r)^2}{r\log(t)}\right\}^2K^{\frac{2}{1+r}}t^{-\frac{2}{1+r}},
\end{align*}
where $K:=(r+1)(r+3)\sqrt{\log(t)}\min\left\{\frac{1}{r-1},\log(t)\right\}$ and the implicit constant is independent of $r$.
\end{thm}
\begin{proof}[Proof sketch]
\renewcommand{\qedsymbol}{}%
By a minimax theorem we can study the dual problem, which involves certain signed measures on the positive half-line. The dual objective can be expressed in terms of a Laplace-type transform of these measures. The constraints on the signed measures allow the control of the transforms, which are analytic functions, by a two-constants type argument.
\end{proof}

We specialize Theorem \ref{thm: main rate} in two ways. Once far from the prediction horizon and once near it.

\begin{cor}
    For fixed $r > 1$,
    \begin{align*}
        \mathcal{R}_B(\hat{S}_{t,T}^{(H^*)},S_{t,T}) \lesssim B^2\log^{-2+\frac{1}{1+r}}(t)t^{-\frac{2}{1+r}},
    \end{align*}where the implicit constant depends on $r$.
\end{cor}

This saves a factor of $\log^{6-\frac{1}{1+r}}(t)$ compared to Theorem 11 of \cite{polyanskiy_2019_dualizing}.

\begin{cor}
    Let $r(t) = c\log(t)$, then $$\mathcal{R}_B(\hat{S}_{t,T}^{(H^*)},S_{t,T}) \lesssim B^2\min\left\{1,c^2\right\} e^{-2/c}.$$ 
\end{cor}
For small $c$ this is a meaningful improvement compared to the SGTE which in this regime obtains $\mathcal{R}_B(\hat{S}_{t,T}^{(\text{SGT})},S_{t,T}) \lesssim  B^2e^{-\frac{2}{c\log(3)}}$.

In the generalized unseen species problem there is in principle much more an estimator could make use of than $\boldsymbol{\phi}$, for example, the number of pairs observed a given number of times, so that one may hope for estimators which have a sub-quadratic dependence on $B$. However, without much difficulty the lower bound of Theorem 11 of \cite{polyanskiy_2019_dualizing} can be generalized to the $B>1$ setting, ruling out this possibility.

\begin{prop}\label{prop: rate lower} In the generalized unseen species problem, for fixed $r > 1$,
    \begin{align*}
        \inf_{\hat{S}_{t,T}}\mathcal{R}_B(\hat{S}_{t,T},S_{t,T}) \gtrsim B^2 t^{-\frac{2}{1+r}} \log^{-4}(t),
    \end{align*}
    the implicit constant depending on $r$.
\end{prop}
\subsection{Finite Sample Guarantees}

A solver (which uses the cutting-plane method) to compute an approximation of $H^*$, as well as the code to generate all figures in the paper is available on our GitHub\footnote{github.com/EdwardEriksson/Unseen-Species. Anthropic's LLMs were used extensively in producing the code.}. 

Once our solver finds an $H$ we evaluate $G(H)$ to get a bound on the MSE. In particular, one need not believe that our optimization procedure found the optimal $H$ to trust the bound on the MSE. It suffices that we can evaluate $G(H)$ accurately for the found $H$. We show the worst case guarantees we can give by this method against those for the SGTE and those for $\hat{S}_{t,T}^{(\text{PW})}$ in Figure \ref{fig:UsVSOrlitskyRMSE}. 
\begin{figure}
    \centering
    \includegraphics[width=1.0\linewidth]{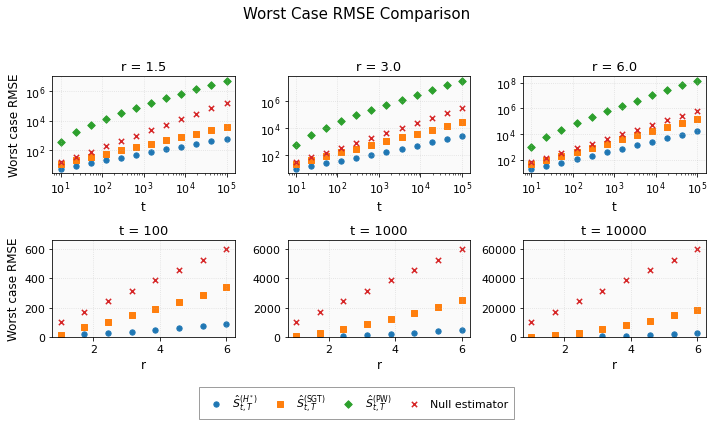}
    \caption{Comparison of the worst case RMSE guarantee of the SGTE, $\hat{S}_{t,T}^{(\text{PW})}$ and $\hat{S}_{t,T}^{(H^*)}$. Note that the top plots use log-log scales and that the guarantees for $\hat{S}_{t,T}^{(\text{PW})}$ are too poor to be seen in the bottom row of plots.}
    \label{fig:UsVSOrlitskyRMSE}
\end{figure}

It is clear that for $r,t$ as large as we consider the worst case guarantees for the estimators of \cite{polyanskiy_2019_dualizing} and \cite{orlitsky_2016_optimal} are much worse than ours. One may then ask whether this reflects a genuine difference in the worst case performance of the estimators or is an artifact of the proof techniques used to get the bounds. It is also interesting to ask on what distributions the estimators perform well and poorly.

To partially answer such questions, and in light of the proof of Theorem \ref{thm: opt powerhouse} which suggests that studying performance on uniform distributions is enough to characterize the performance of a linear estimator, we experimentally test the performance of the estimators on near-uniform distributions. Namely, with $p \in (0,1)$ we consider the uniform distribution on $\left\lfloor \frac{1}{p} \right\rfloor$ symbols with an extra species to absorb the leftover probability mass. We query values of $p$ adaptively and use a Gaussian process prior. The results are displayed in Figure \ref{fig: performance uniforms}.

\begin{figure}
    \centering
    \includegraphics[width=0.85\linewidth]{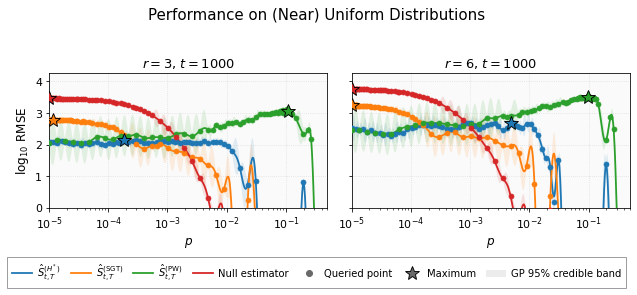}
    \caption{Performance of estimators on distributions which are essentially the uniform distribution on $\frac{1}{p}$ symbols.}
    \label{fig: performance uniforms}
\end{figure}
Clearly, $\hat{S}_{t,T}^{(H^*)}$ is dominant in worst case error,
although for distributions with very little to discover the SGTE outperforms it, and in the absence of common species $\hat{S}_{t,T}^{(\text{PW})}$ is quite competitive with it.

We note that Theorem \ref{thm: opt powerhouse} can easily be adapted to the assumption of lower bounded probabilities which is a common assumption in, for example, support size estimation \cite{wu_2019_chebyshev}. Specifically, if one restricts the supremum to be over $\mu$ with no species less probable than $p_0$ for some $p_0>0$, then the upper bound in Theorem \ref{thm: opt powerhouse} holds with the suprema restricted to $[p_0,1]$. Thus one can construct, using essentially the same methodology, estimators which are better suited to the regime of few rare species.

\subsection{Uncertainty Quantification}

For a linear estimator, for $\mu \prec 1$, we suggest estimating the variance by
$$\hat{V}_t := \hat{S}_{t,T}^{(H)}+\sum_{s \in \mathcal{S}}H(N_s)^2.$$ 
Let $N_s'$ denote the number of observations of species $s$ on the second expedition and $b_s := \IE[1_{N_s =0}1_{N_s'>0}-H(N_s)]$  the bias associated with species $s$.

\begin{thm} \label{thm: CLT Intermediate}
Let $\hat{S}_{t,T}^{(H)}$ be a linear estimator, not necessarily atemporal, and let $r(t)$ possibly vary in time. Suppose that $\IE[\hat{S}^{(H)}_{t,T}] \geq 0$, $\mu \prec 1$ and that

\begin{align*}
    \frac{|\sum_{s \in \mathcal{S}}{b_s}|+\sum_{s \in \mathcal{S}}b_s^2+||H||_\infty^2+1}{\mathbb{V}[S_{t,T}-\hat{S}_{t,T}]} \rightarrow 0.
\end{align*}
Then,
    \begin{align*}
        \frac{S_{t,T}-\hat{S}_{t,T}^{(H)}-\IE[S_{t,T}-\hat{S}_{t,T}^{(H)}]}{\sqrt{\hat{V}_t}} \overset{d}{\rightarrow} \mathcal{N}(0,1).
    \end{align*}
\end{thm}
\begin{proof}[Proof sketch]
\renewcommand{\qedsymbol}{}%
Because of linearity of the estimator the error is a sum over contributions from the different species, so that we may use the Feller--Lindeberg CLT for triangular arrays to get a CLT with the true variance. Then we ensure that the estimator is multiplicatively consistent for the variance and apply Slutsky's Theorem. 
\end{proof}
\begin{rmk}
Theorem \ref{thm: CLT Intermediate} implies Corollary  \ref{cor: GT Gaussianity} by noticing that the GTE has $b_s \equiv 0$ and that for $r\leq 1$, $||H||_\infty \leq 1$.
\end{rmk}
Since Theorem \ref{thm: CLT Intermediate} includes a bias (and has in practice unverifiable, although reasonable, assumptions) it cannot directly be used to construct principled prediction intervals. However, we find empirically that on natural distributions the variance is the dominant contributor to the error of $\hat{S}_{t,T}^{(H^*)}$, especially when $r$ is not very large, so that dropping the bias produces good prediction intervals and that the same is true of the linearized version of $\hat{S}_{t,T}^{(\text{PW})}$. See Table \ref{tab:coverage}. We emphasize that this is for $\mu \prec 1$. Satisfactory uncertainty quantification in the generalized unseen species problem, when the dependence between species is too large to be assumed away, is still open.

\begin{table}[ht]
\centering
\scriptsize
\setlength{\tabcolsep}{3pt}
\resizebox{\textwidth}{!}{%
\begin{tabular}{lcccccc}
\toprule
 & \multicolumn{3}{c}{$t=1000$, $r=2$} & \multicolumn{3}{c}{$t=1000$, $r=5$} \\
\cmidrule(lr){2-4}\cmidrule(lr){5-7}
Distribution & $H^*$ & SGT & PW & $H^*$ & SGT & PW \\
\midrule
Zipf$(1)$, $M=10^4$ & $95.5 \pm 0.3$ & $76.7 \pm 0.6$ & $95.6 \pm 0.3$ & $94.9 \pm 0.3$ & $0.7 \pm 0.1$ & $96.8 \pm 0.2$ \\
Zipf$(1.25)$, $M=10^4$ & $95.0 \pm 0.3$ & $85.9 \pm 0.5$ & $96.7 \pm 0.3$ & $95.7 \pm 0.3$ & $14.9 \pm 0.5$ & $96.0 \pm 0.3$ \\
Zipf$(1.5)$, $M=10^4$ & $94.5 \pm 0.3$ & $91.4 \pm 0.4$ & $96.5 \pm 0.3$ & $94.8 \pm 0.3$ & $56.0 \pm 0.7$ & $96.5 \pm 0.3$ \\
Zipf$(2)$, $M=10^4$ & $95.1 \pm 0.3$ & $93.4 \pm 0.4$ & $97.2 \pm 0.2$ & $95.2 \pm 0.3$ & $89.7 \pm 0.4$ & $97.0 \pm 0.2$ \\
Zipf$(1.5)$, $M=10^5$ & $94.5 \pm 0.3$ & $90.7 \pm 0.4$ & $96.8 \pm 0.3$ & $95.1 \pm 0.3$ & $49.1 \pm 0.7$ & $96.6 \pm 0.3$ \\
Hamlet (i.i.d.) & $94.7 \pm 0.3$ & $86.4 \pm 0.5$ & $93.4 \pm 0.4$ & $95.5 \pm 0.3$ & $14.7 \pm 0.5$ & $94.7 \pm 0.3$ \\
Surnames (2010 Census) & $95.3 \pm 0.3$ & $48.1 \pm 0.7$ & $83.2 \pm 0.5$ & $94.9 \pm 0.3$ & $0.0 \pm 0.0$ & $95.0 \pm 0.3$ \\
Uniform, $M=5000$ & $94.6 \pm 0.3$ & $78.6 \pm 0.6$ & $75.9 \pm 0.6$ & $78.3 \pm 0.6$ & $9.5 \pm 0.4$ & $67.2 \pm 0.7$ \\
\bottomrule
\end{tabular}
}
\caption{Empirical coverage (\%), $\pm$ SEM, of nominal (ignoring bias) 95\% prediction intervals. }
\label{tab:coverage}
\end{table}

\section{Distant Future} \label{sec: distant}

\subsection{Concentration Inequalities for Set Functionals}
\label{subsec: Distant Future, Generalized Setting}
The key technical tool of the present section is the following concentration inequality. Let $S_t^{(i)} := \sum_{j=i}^\infty \phi_j$ be the number of species seen at least $i$ times. 
\begin{lem}\label{lem: new concentration} Suppose $\mu \prec B$. Then, for $z \geq 0$,
\begin{align*}
    P\left(S_t^{(i)}-\IE\big[S_t^{(i)}\big] \leq -z\right) &\leq \exp\left(-\frac{z^2}{2((B-1)i+1)\IE\big[S_t^{(i)}\big]}\right), \\
    P\left(S_t^{(i)}-\IE\big[S_t^{(i)}\big] \geq z\right) &\leq \exp\left(-\frac{z^2}{2((B-1)i+1)\big(\IE\big[S_t^{(i)}\big]+\frac{z}{3}\big)}\right).
\end{align*}
\end{lem}
\begin{proof}[Proof sketch]
\renewcommand{\qedsymbol}{}%
An application of the work of \cite{bartroff_2018_bounded} on concentration via size biased couplings for configuration models. A bounded size biased coupling can be constructed since $\mu$ is of bounded arity, which implies concentration. This is essentially  the hypergraph version of the graph example there.
\end{proof}
\begin{rmk} In \cite{benhamou_2017_concentration} concentration inequalities are proven for the classical urn scheme. Two incomparable results are obtained, one by a negative association argument and one by the entropy method. They note that using Stein's method one obtains the same result as using negative association and indeed in the classical setting this is what Lemma \ref{lem: new concentration} amounts to.
\end{rmk}
Note that although we will use it under a power law approximation, Lemma \ref{lem: new concentration} itself does not rely on this. We generalize the power law hypothesis in the natural way, namely by setting $\nu(x) := \sum_{s \in \mathcal{S}}1_{M_s > x}$ and requiring that $\nu(x) \approx cx^{-\alpha}$ for some $c > 0, \alpha \in (0,1)$. We then obtain the following bound on the probability of making a large error. Let $\mathcal{L}[f]$ denote the Laplace transform of a function $f$.
\begin{thm} \label{thm: far future main}
    Let $\hat{S}_{t,T}:=S_t\big((1+r)^{\phi_1/S_t}-1\big)$, $p \in (0,1)$ and $q:= 1-p$. Suppose that $\mu$ is such that for some constants $c>0$ and $\alpha \in (0,1)$ we have $|\nu(x)-cx^{-\alpha}| \leq f(x)$ for all $x \in \mathbb{R}_{>0}$, where $f \geq 0$ is such that $\mathcal{L}[f](t) < \infty$ for all $t > 0$, and that $\mu \prec B$. Let
\begin{align*}
    d(z) &:= \frac{pz c\Gamma(1-\alpha)t^\alpha(1+r)^{-\alpha}}{(1+2\log(1+r))c\Gamma(1-\alpha)t^\alpha+pz(2-\alpha)\log(1+r)(1+r)^{-\alpha}},
\end{align*}
and suppose that 
\begin{align}\label{eq: three hypothesis estimates}
    T\mathcal{L}[f](T) \leq qz, \qquad t\mathcal{L}[f](t) \leq d(z), \qquad -t^2 \mathcal{L}[f]'(t) \leq d(z). 
\end{align}
Then,
\begin{align*}
    P(|\hat{S}_{t,T}-S_{t,T}|>z) &\leq \exp \left(-\frac{(qz-T\mathcal{L}[f](T))^2}{2B\IE[S_T]}\right)\\
    &+ \exp\left(-\frac{(qz-T\mathcal{L}[f](T))^2}{2B\IE[S_T]+\frac{2}{3}B(qz-T\mathcal{L}[f](T))}\right)\\
    &+\exp\left(-\frac{(d(z)-t\mathcal{L}[f](t))^2}{2B\IE[S_t]}\right) \\
    &+\exp\left(-\frac{(d(z)-t\mathcal{L}[f](t))^2}{2B\IE[S_t]+\frac{2}{3}B(d(z)-t\mathcal{L}[f](t))}\right) \\
    &+\exp\left(-\frac{(d(z)+t^2\mathcal{L}[f]'(t))^2}{(4B-2)\IE[S_t^{(2)}]}\right) \\
    &+\exp\left(-\frac{(d(z)+t^2\mathcal{L}[f]'(t))^2}{(4B-2)\IE[S_t^{(2)}]+\frac{4B-2}{3}(d(z)+t^2\mathcal{L}[f]'(t))}\right).
\end{align*}
\end{thm}
\begin{proof}[Proof sketch]
\renewcommand{\qedsymbol}{}%
$S_{t,T}-\hat{S}_{t,T} = S_T-S_t(1+r)^{\phi_1/S_t}$. Heuristically, Proposition 17 of \cite{Gnedin_notes} suggests that if (i) the power law approximation was perfect (which it cannot be) and (ii) $S_t,S_T$ and $\phi_1$ were perfectly concentrated (which they are not), then the prediction error would be zero. It is thus natural to control the error by studying (a) how much deviation from a perfect power law can impact $\IE[S_t],\IE[S_T]$ and $\IE[\phi_1]$, (b) controlling how far $S_t,S_T$ and $\phi_1$ can be from their means and (c) controlling sensitivity to deviation in these quantities. 

Part (a) is done through some elementary Laplace transform  analysis, very much in the spirit of \cite{Gnedin_notes} but using a pre-asymptotic hypothesis to get pre-asymptotic bounds. Part (b) amounts to noticing that $S_t = S_t^{(1)}$, $\phi_1 = S_t^{(1)}-S_t^{(2)}$ and appealing to Lemma \ref{lem: new concentration} and (c) uses an analytical lemma due to Anthropic's Claude Opus 4.7.
\end{proof}

Note that using $\hat{\alpha} = \frac{\phi_1}{S_t}$ gives rise to a representation invariant estimator, while the estimator of \cite{favaro_2023_nearoptimal} is not. 

Define the loss $\ell_{\alpha}(S_{t,T},\hat{S}_{t,T}) := \frac{|S_{t,T}-\hat{S}_{t,T}|^2}{(rt)^{2\alpha}}.$

\begin{cor}\label{cor: far future rate} If $\mu \prec B$, $\log(r) = o(t^{\alpha/2})$, $r(t) \geq r_0 > 1$ for all $t$ for some fixed $r_0$, $|\nu(x)-cx^{-\alpha}| \leq Kx^{-\frac{\alpha}{2}}$ for all $x \in (0,1)$ and $\hat{S}_{t,T} = S_t((1+r)^{\phi_1/S_t}-1)$ then for every $\epsilon > 0$ there exists a constant $C(\epsilon,\alpha,c,K,r_0)$ such that
\begin{align*}
    \limsup_{t \rightarrow \infty}P\left(\ell_{\alpha}(S_{t,T},\hat{S}_{t,T}) > \frac{CB\log(r)^2}{t^\alpha}\right) < \epsilon.
\end{align*}
\end{cor}
Note that one should not expect the dependence to be on $B^2$ as it was in the previous sections. This is because speciating a $\mu$ with $|\nu(x)-cx^{-\alpha}| \leq Kx^{-\alpha/2}$ does not yield a $\mu$ satisfying the same estimate. Instead one obtains a $\mu$ with an associated $\nu$ satisfying $|\nu(x)-Bcx^{-\alpha}| \leq BKx^{-\alpha/2}$. Nonetheless it seems likely that speciation could give a generalization of Theorem 6 of \cite{favaro_2023_nearoptimal} with the aim of pinning down the optimal dependence on $B$. We have not pursued this. Instead, specializing Theorem \ref{thm: far future main} to the problem class of \cite{favaro_2023_nearoptimal}, we obtain the same near-optimal rate in $r,t$ as they do.

If the power law approximation is very good the $B=1$ case of Corollary \ref{cor: far future rate} improves the prediction horizon from $\log(r)=O(t^{\alpha/2}/\sqrt{\log(t)})$ to $\log(r(t)) = o(t^{\alpha/2})$ as compared to Theorem 3 of \cite{favaro_2023_nearoptimal}. However, as the authors explained in personal communication, they chose their class for reasons related to the lower bound construction. On our problem class they can strengthen their results to match ours.

For us the following is really a lemma used in the proof of Theorem \ref{thm: CLT distant}, but as it may be of interest beyond the unseen species problem we state it as a theorem.

\begin{thm}\label{thm: alpha-rate} Let $\hat{\alpha} = \frac{\phi_1}{S_t}$. Fix $\epsilon > 0$. If $\mu \prec B$ and there exist constants $c,\alpha,K$ such that $|\nu(x)-cx^{-\alpha}| \leq Kx^{-\frac{\alpha}{2}}$ for all $x \in (0,1)$ then there exists a constant $C(\epsilon,c,\alpha,K,B)$ so that
\begin{align*}
    \limsup_{t \rightarrow \infty}P\left(|\alpha-\hat{\alpha}| > \frac{C}{t^{\frac{\alpha}{2}}}\right) < \epsilon.
\end{align*}
\end{thm}

\subsection{Uncertainty Quantification}

The following central limit theorem for the prediction error was proven with substantial aid of Anthropic's Claude Opus 4.8. It allows the construction of asymptotically calibrated prediction intervals when $B=1$.

\begin{thm}\label{thm: CLT distant} Let $\hat{S}_{t,T} := S_t((1+r)^{\phi_1/S_t}-1)$ and suppose that $\mu \prec 1$ is such that $\nu(x)=cx^{-\alpha}+o(x^{-\alpha/2})$ as $x \downarrow 0$ for some $c > 0, \alpha \in (0,1)$, $\frac{\log(r(t))}{t^{\alpha/2}} \rightarrow 0$ and $\liminf_{t \rightarrow \infty}r(t) >0$. Then,
\begin{align*}
    \frac{S_{t,T}-\hat{S}_{t,T}}{\sqrt{V_t}} \overset{d}{\rightarrow} \mathcal{N}(0,1),
\end{align*} where $V_t$ may be replaced by $\hat{V}_t$ and these quantities are defined by
\begin{align*}
    \rho :=& r+1, \qquad \beta := \log(\rho), \qquad \hat{\alpha} := \phi_1/S_t \\
    \Xi(r,\alpha) :=& \alpha\beta^2-\beta^2\alpha(1-\alpha)2^{\alpha-2}+(1-\alpha\beta)^2(2^\alpha-1)+\rho^{-\alpha}(2^\alpha-1) \\
    +&\beta(1-\alpha\beta)\alpha 2^\alpha -2\beta \rho^{-\alpha}\alpha (1+\rho)^{\alpha-1}-2(1-\alpha\beta)\rho^{-\alpha}((\rho+1)^\alpha-\rho^\alpha) \\
    V_t :=& \rho^{2\alpha}c\Gamma(1-\alpha)\Xi(r,\alpha)t^\alpha, \qquad
    \hat{V}_t := \rho^{2\hat{\alpha}}  \Xi(r,\hat{\alpha}) S_t
\end{align*}
\end{thm}
\begin{proof}[Proof sketch]
\renewcommand{\qedsymbol}{}%
It transpires that the prediction error can be written as a sum over species-wise contributions plus a correction term. The sum is shown asymptotically normal by the Feller--Lindeberg CLT for triangular arrays much like in Theorem \ref{thm: CLT Intermediate} and then the correction term is shown to be vanishingly small. Slutsky's theorem is used to move to an estimated variance.
\end{proof}

\section{Experiments} \label{sec: comparison}

\subsection{Task}
For $r > 1$ we compare the performance of a variety of estimators. Our prescribed task is modeled after the use case in \cite{adams_2022_victim}. Following \cite{efron_1976_estimating},  \cite{wu_2019_chebyshev} and \cite{orlitsky_2016_optimal} we consider a person reading Hamlet. They know the total number of words in the play and aim to predict the total number of distinct words from reading an initial portion of the piece. Setting $\hat{S}_T := S_t + \hat{S}_{t,T}$ we see that this is an instantiation of the unseen species problem. In order to make the plots as comparable as possible we study absolute relative error, namely $\frac{|\hat{S}_T-S_T|}{S_T}$. When it makes sense to do so we give both the performance when the data is in the true temporal order and an average over $100$ shuffles to reduce noise. The shuffling is justified insofar as we believe exchangeability holds. We prefer this to sampling i.i.d. from the empirical distribution as that underestimates the number of rare species and artificially introduces exchangeability.

\subsection{Data}

In addition to Hamlet we consider data sets of butterfly collecting \cite{yu_2025_spatial}, messages on a social media network \cite{panzarasa_2009_patterns} (sub-sampled), genome coverage (SRA code SRX151616, restricted to the first $10^5$ genome locations, sub-sampled), and simulated packs of MTG cards \cite{joshbirnholz_2025_github}. The messages and genome fragments are sub-sampled to put us into a regime where there is still a reasonable number of species left to discover. We think this is the interesting regime and hence warranted, but note that Figure \ref{fig: performance uniforms} suggests that the SGTE is better at detecting when there is essentially nothing to discover and so this choice is disfavorable for the SGTE. For all but the Hamlet data we use the generalized unseen species model. The data is described in more detail in Appendix \ref{app: data}.

\subsection{Benchmarks}
We compare $\hat{S}_{t,T}^{(H^*)}$ and our suggested distant future estimator (\emph{Ratio-$\alpha$}) to several competitors. A naive benchmark is obtained from the null estimator (\emph{Trivial}). We include the binomial smoothing SGTE with the suggested parameters, the power law estimator of \cite{favaro_2023_nearoptimal} (\emph{MLE-$\alpha$}), and $\hat{S}_{t,T}^{(\text{PW})}$. We also compare against the method of \cite{daley_2014_modeling} (\emph{Padé-GT$[2,3]$}) developed for applications in genome coverage and the support size estimator of \cite{wu_2019_chebyshev} (\emph{Chebyshev}), treating the estimated support size as an estimate for $S_T$. We do this to investigate the difference between the unseen species problem and support size estimation. In their notation we set $k$ to be the corpus length and $c_0=0.45,c_1 = 0.5$ as their experimentally chosen values.
\begin{figure}[ht]
    \centering
    \includegraphics[width=1\linewidth]{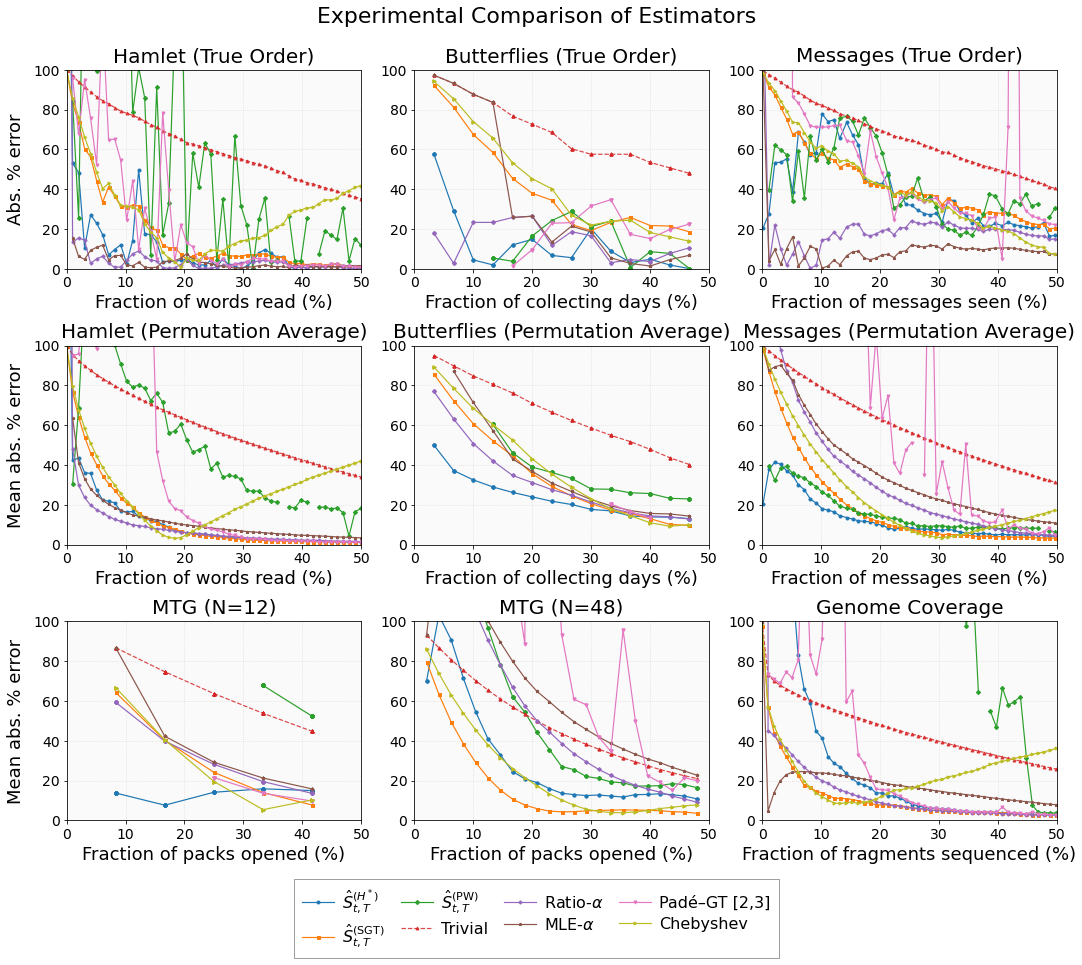}
    \caption{Empirical Comparison of Performance of Estimators}
    \label{fig:Key}
\end{figure}

\subsection{Results}

See the result in Figure \ref{fig:Key}. We see that for the messages data the plots for the true order and the permutation average look totally different. This strongly suggests that exchangeability was severely violated and that the model is therefore not a good fit. One sees that good performance on the permutation average does not translate into good predictions under the temporal ordering. This serves as an important reminder that exchangeability is a non-trivial modeling assumption and that applications may require specialized models. One way to test the suitability of the model might be to artificially use an initial portion of the data as the first expedition and the second part of the data as the return expedition.

We note that as expected the Chebyshev-polynomial-based estimator of \cite{wu_2019_chebyshev} performs poorly for small $r$ as for such $r$ the unseen species problem is very different from support size estimation.

We find that our suggested distant future estimator fairly consistently outperforms that of \cite{favaro_2023_nearoptimal} and that $\hat{S}_{t,T}^{(H^*)}$ performs well but is not universally superior. One expects, from Zipf's law \cite{zipf_1949_human}, that the Hamlet data is prototypical power law data and therefore the poor performance of $\hat{S}_{t,T}^{(\text{PW})}$ there is particularly concerning. There are some missing data points for $\hat{S}_{t,T}^{(\text{PW})}$ because our linear programming solver was sometimes unable to find a solution to the LP. Neither the linear nor the power law estimators are consistently better than the others across data sets. Also, despite theoretically having drastically different prediction horizons it does not appear that the power law estimators are dominant for large $r$. No doubt the extent to which the power law approximation is accurate is an important factor in prediction performance of those estimators. See appendix F of \cite{favaro_2023_nearoptimal} for a diagnostic tool to assess such fit. Raw data for Figure \ref{fig:Key} can be found in Appendix \ref{app: data tables}. 

\section*{Acknowledgments}
We thank Alexander Fuchs-Kreiss, who carefully read previous drafts and provided elaborations on a number of technical details. We also thank Torben Swoboda, who thought of Example \ref{exam: mtg}.

\printbibliography[heading=bibintoc]

\appendix
\section{Proofs}
\subsection{General Facts}
We collect some facts which are not associated with any particular regime but see frequent use in the subsequent proofs.
\begin{itemize}
    \item $\mathcal{F}_t \indep \mathcal{F}_{t,T}$, where $\mathcal{F}_{t,T}$ is the sigma algebra generated by the second expedition.
    \item By the thinning property of Poisson processes, $N_{x \backslash y}, N_{x \cap y}$ and $N_{y \backslash x}$ are independent.
    \item From the definitions it is immediate that $M_{x \cap y} + M_{x \backslash y} = M_{x}$ and $M_{x \backslash y} + M_{x \cap y} + M_{y \backslash x} = M_{x \cup y}.$
\end{itemize}

\begin{lem} \label{lem: linear bounds} If $\mu \prec B$, then
    \begin{align*}
        \IE[S_t] \leq Bt, \qquad
        \IE[S_{t,T}] \leq Brt.
    \end{align*}
\end{lem}
\begin{proof}[Proof of Lemma \ref{lem: linear bounds}]
  As before, let $n \sim \text{Po}(t)$ be the number of sets observed by time $t$ and let $m \sim \text{Po}(rt)$ be the number of sets observed in the second time period. 
  \begin{align*}
      \IE[S_t] &\leq \IE[Bn] = Bt, \\
      \IE[S_{t,T}] &\leq \IE[Bm] = Brt.
  \end{align*}
  \end{proof}
To realize these bounds, take $K$ disjoint sets of $B$ elements each. Let $\mu$ be the uniform distribution on these $K$ sets and then let $K \rightarrow \infty$.

\subsection{Symmetries}
\begin{proof}[Proof of Proposition~\ref{prop:symmetries}] We show that the SEP implies linearity. The other implications are trivial. For $i \in \mathbb{Z}_{>0}$, let $e_i \in \mathbb{Z}_{\geq 0}^\infty$ be the sequence with $i$th coordinate $1$ and all other coordinates zero. It holds that
    \begin{align*}
        \hat{S}_{t,T}(r,t,\boldsymbol{\phi}) &= \hat{S}_{t,T}(r,t,\sum_{i=1}^\infty \phi_i e_i), \\
        &= \sum_{i=1}^\infty \hat{S}_{t,T}(r,t,\phi_ie_i), \\
        &= \sum_{i=1}^\infty\sum_{j=1}^{\phi_i} \hat{S}_{t,T}(r,t,e_i), \\
        &= \sum_{i=1}^\infty \hat{S}_{t,T}(r,t,e_i) \phi_i.
    \end{align*}
Now defining $H(i) := \hat{S}_{t,T}(r,t,e_i)$ shows the estimator to be linear.
\end{proof}
\subsection{Near Future}

\begin{proof}[Proof of Proposition \ref{prop: error decomposition}]
We expand
\begin{align*}
    \IE[|S_{t,T}-\hat{S}_{t,T}|^2] &=  \IE[|\sum_{s \in \mathcal{S}}1_{N_s=0}1_{N_s'>0}+1_{N_s>0}(-r)^{N_s}|^2], \\
    &= \IE[\sum_{s \in \mathcal{S}}1_{N_s=0}1_{N_s'>0}+1_{N_s > 0}r^{2N_s}] \\
    &+ \sum_{x,y} \IE[(1_{\substack{N_x=0\\ N_x'>0}}+1_{N_{x}>0}(-r)^{N_x})(1_{\substack{N_y=0\\ N_y'>0}}+1_{N_y > 0}(-r)^{N_y})].
\end{align*}
The first line above is just $\delta(t)$. It remains to prove that the second line equals $\epsilon(t)$. Expanding the brackets we see that we have three types of terms to deal with. We deal with each in turn.
By independence of past and future,
\begin{align*}
    \IE[1_{N_x=0}1_{N_x'>0}1_{N_y=0}1_{N_y'>0}] &= \IE[1_{N_x=0}1_{N_y=0}]\IE[1_{N_x'>0}1_{N_y'>0}]
\end{align*}
It is immediate that $\IE[1_{N_x=0}1_{N_y=0}] = e^{-M_{x \cup y}t}$. Moreover,
\begin{align*}
    \IE[1_{N_x'>0}1_{N_y'>0}] &= \IE[1_{N_x'>0}1_{N_y'>0}|N_{x \cap y}'=0]P(N_{x \cap y}'=0) \\
    &+\IE[1_{N_x'>0}1_{N_y'>0}|N_{x \cap y}'>0]P(N_{x \cap y}'>0), \\
    &= (1-e^{-rM_{x \backslash y}t})(1-e^{-rM_{y \backslash x}t})e^{-rM_{x \cap y}t}+(1-e^{-rM_{x \cap y}t}), \\
    &= 1-e^{-rM_xt}-e^{-rM_yt}+e^{-rM_{x \cup y}t}.
\end{align*}
Therefore,
\begin{align} \label{eq: type 1}
    \IE[1_{N_x=0}1_{N_x'>0}1_{N_y=0}1_{N_y'>0}] = e^{-M_{x \cup y}t}(1-e^{-rM_xt}-e^{-rM_yt}+e^{-rM_{x \cup y}t}).
\end{align}
Next we compute one of the cross terms. The other is the same up to swapping the roles of $x$ and $y$.
If $L$ is a Poisson random variable with mean $\lambda$ then $\IE[1_{L>0}(-r)^L] = \IE[(-r)^L]-P(L = 0) = e^{-\lambda(1+r)}-e^{-\lambda} = e^{-\lambda}(e^{-r\lambda}-1)$. From this one sees that
\begin{align}
    \IE[1_{N_x=0}1_{N_x'>0}1_{N_y>0}(-r)^{N_y}] &= P(N_x=0)\IE[1_{N_x'>0}1_{N_y>0}(-r)^{N_y}|N_x=0]\nonumber, \\
    &= P(N_x=0)P(N_x'>0)\IE[1_{N_y>0}(-r)^{N_y}|N_x=0]\nonumber, \\    
    &=(1-e^{-rM_xt})e^{-M_xt}\IE[1_{N_{y \backslash x} > 0}(-r)^{N_{y \backslash x}}] \nonumber, \\
    &= (1-e^{-rM_xt})e^{-M_xt}e^{-M_{y \backslash x}t}(e^{-rM_{y \backslash x}t}-1) \nonumber, \\
    &= -e^{-M_{x \cup y}t}(1-e^{-rM_xt})(1-e^{-rM_{y \backslash x}t}). \label{eq: type 2}
\end{align}
We turn our attention to the third type of term.
\begin{align*}
    \IE[1_{N_x >0}1_{N_y > 0}(-r)^{N_x}(-r)^{N_y}] &= \IE[(-r)^{N_x}(-r)^{N_y}]-P(N_x=0)\IE[(-r)^{N_{y \backslash x}}] \\
    &-P(N_y=0)\IE[(-r)^{N_{x \backslash y}}]+P(N_x=0,N_y=0), \\
    &= \IE[(-r)^{N_x}(-r)^{N_y}]-e^{-M_xt}e^{-(1+r)M_{y \backslash x}t} \\
    &- e^{-M_yt}e^{-(1+r)M_{x \backslash y}t}+e^{-M_{x \cup y}t}.
\end{align*}
Next notice that $\IE[(-r)^{N_x}(-r)^{N_y}]$ is exactly the probability generating function of a suitable bivariate Poisson distribution. From the first equation of \cite{loukas_1986_the}, we have
\begin{align*}
    \IE[(-r)^{N_x}(-r)^{N_y}] = \exp(-(1+r)M_{x \cup y}t+r(r+1)M_{x \cap y}t).
\end{align*}
Combining these terms and tidying up slightly we have

\begin{align} \label{eq: type 3}
    \IE[1_{N_x >0}1_{N_y > 0}(-r)^{N_x}(-r)^{N_y}] &= e^{-(1+r)M_{x \cup y}t+r(r+1)M_{x \cap y}t} \\
    &+e^{-M_{x \cup y}t}(1-e^{-rM_{y \backslash x}t}-e^{-rM_{x \backslash y}t}). \notag
\end{align}

Combining (\ref{eq: type 1}), (\ref{eq: type 2}),(\ref{eq: type 3}) gives  that each term is
\begin{align*}
    e^{-(1+r)M_{x \cup y}t+r(r+1)M_{x \cap y}t}+e^{-M_{x \cup y}t}L_{x,y},
\end{align*}
where
\begin{align*}
    L_{x,y} :&=1-e^{-rM_{x \backslash y}t}-e^{-rM_{y \backslash x}t}  \\
    &-1+e^{-rM_xt}+e^{-rM_{y \backslash x}t} - e^{-rM_{x \cup y}t} \\
    &-1+e^{-rM_yt}+e^{-rM_{x \backslash y}t} - e^{-rM_{x \cup y}t} \\
    &+1-e^{-rM_xt}-e^{-rM_yt}+e^{-rM_{x \cup y}t},\\
    &= -e^{-rM_{x \cup y}t}.
\end{align*}
Thus the decomposition holds with $$\epsilon(t) = \sum_{x,y}e^{-(1+r)M_{x \cup y}t}(e^{r(r+1)M_{x \cap y}t}-1),$$
as claimed.
\end{proof}

\begin{lem}\label{lem: epsilon bound} If $\mu \prec B$ and $r \leq 1$ then
\begin{align*}
    \epsilon(t) \leq r(r+1)(B-1)\IE[S_t].
\end{align*}
\end{lem}
\begin{proof}
Since there are only countably many subsets of $\mathcal{S}$ with at most $B$ elements we may enumerate them in some fixed but arbitrary way. Let $\mu_n$ agree with $\mu$ on the first $n$ sets but assign zero intensity to all subsequent sets (since we are dealing with the Poissonized version of the problem $\mu_n$ not being normalized is not a problem). We aim to establish the sought inequality for all the measures $\mu_n$ and then use a limit argument to obtain it for $\mu$.

Notice that the inequality trivially holds for $\mu_0 \equiv 0$. It remains to show that the inequality never ceases to hold. Let $\lambda$ be the amount of mass we add at the $(n+1)$th step and $A_{n+1}$ the set we add it to. We will show that the RHS is increasing faster in $\lambda$ than the LHS so that when we add the mass (producing $\mu_{n+1}$) the inequality remains true. The $\epsilon(t)$ associated to $\mu_{n+1}$ may be written as,
$$\sum_{x,y}\epsilon_{x,y}^{(n+1)}(t) := \sum_{x,y}e^{-(r+1)(M_{x \cup y}^{(n)}+1_{\{x,y\} \cap A_{n+1} \neq \emptyset}\lambda)t}(e^{r(r+1)(M_{x \cap y}^{(n)}+1_{\{x,y\} \subset A_{n+1}}\lambda)t}-1),$$
where $M_{x \cap y}^{(n)} := \sum_{k \leq n: A_k \supset \{x,y\}} \mu(A_k)$ is the $M_{x \cap y}$ associated to $\mu_{n}$ and we define ${M}_{x \cup y}^{(n)}$ analogously. This notation separates that which depends on $\lambda$ from that which does not. We decompose as,
\begin{align*}
    \partial_{\lambda}\epsilon^{(n+1)}(t) &= \partial_{\lambda} \sum_{x,y}\epsilon_{x,y}^{(n+1)}(t), \\
    &= \partial_{\lambda}\sum_{x,y: |\{x,y\} \cap A_{n+1}| =0}\epsilon_{x,y}^{(n+1)}(t)+\partial_{\lambda}\sum_{x,y: |\{x,y\} \cap A_{n+1}| =1}\epsilon_{x,y}^{(n+1)}(t) \\
    &+\partial_{\lambda}\sum_{x,y: |\{x,y\} \cap A_{n+1}| =2}\epsilon_{x,y}^{(n+1)}(t).
\end{align*}
Increasing $\lambda$ does not change the first term at all since for such $x,y$ neither $M_{x \cap y}$ nor $M_{x \cup y}$ are changed. For the second collection of terms, only $M_{x \cup y}$ is changed, and it is increasing which can only decrease $\epsilon_{x,y}(t)$. Thus we may take only the third term (which has finitely many summands) as an upper bound on the derivative. A short computation shows that for such terms,
\begin{align*}
    \partial_\lambda \epsilon_{x,y}(t)&= t(r+1)e^{-(r+1)M_{x \cup y}^{(n+1)}t}(1+(r-1)e^{r(r+1)M_{x \cap y}^{(n+1)}t}), \\
    &\leq t(r+1)e^{-(r+1)M_{x \cup y}^{(n+1)}t}r,
\end{align*}
Note that $\lambda$ does implicitly appear in the above equations but that we have absorbed it into the $M$'s. Let $\IE_n[S_t] := \sum_{s \in \mathcal{S}}\left(1-e^{-M^{(n)}_st}\right)$, where $M^{(n)}_s := \sum_{k \leq n: s \in A_k}\mu(A_k)$. Then,
\begin{align}
    \partial_{\lambda}r(r+1)(B-1)\IE_{n+1}[S_t] &=  r(r+1)(B-1)\partial_\lambda \sum_{s \in A_{n+1}}\left(1-e^{-M_s^{(n)}t}e^{-\lambda t}\right)  \nonumber, \\
    &=r(r+1)(B-1)\sum_{s \in A_{n+1}}t e^{-M_s^{(n)}t}e^{-\lambda t} \nonumber \\
    &= r(r+1)(B-1)\sum_{s \in A_{n+1}}t e^{-M_s^{(n+1)}t}\label{eps RHS2}.
\end{align}
There are exactly $|A_{n+1}|(|A_{n+1}|-1)$ pairs contained in $A_{n+1}$ and exactly $|A_{n+1}|$ terms in (\ref{eps RHS2}). Use up the $B-1$ pre-factor in (\ref{eps RHS2}) to duplicate each term in the sum $B-1$ times. After doing so we end up with $|A_{n+1}|(B-1) \geq |A_{n+1}|(|A_{n+1}|-1)$ terms. Pick a pair of distinct elements $x,y \in A_{n+1}$. There are two ordered pairs consisting of these elements. Map one ordered pair to $x$ and the other to $y$. This assigns to each element of $A_{n+1}$ at most $B-1$ ordered pairs in such a way that each element of $A_{n+1}$ is assigned only pairs which contain it. Observing that $M_{x \cup y} \geq M_x$ for any $y$ the result now follows for all finite $\mu_n$ from a term-wise comparison. To finish, note that $\IE_n[S_t]$ is increasing in the masses and so $\lim_{n \rightarrow \infty}\IE_{n}[S_t] = \IE[S_t]$ by the monotone convergence theorem. Now, by continuity $\epsilon^{(n)}_{x,y}(t) \rightarrow \epsilon_{x,y}(t)$ and so we have, by Fatou's lemma,
\begin{align*}
    \epsilon(t) &= \sum_{x,y} \epsilon_{x,y}(t), \\
    &=  \sum_{x,y} \liminf_{n \rightarrow \infty}\epsilon_{x,y}^{(n)}(t), \\
    &\leq \liminf_{n \rightarrow \infty} \epsilon^{(n)}(t), \\
    &\leq r(r+1)(B-1)\liminf_{n \rightarrow \infty} \IE_{n}[S_t], \\
    &= r(r+1)(B-1)\IE[S_t].
\end{align*}
\end{proof}

\begin{proof}[Proof of Theorem \ref{thm: GT analysis}]

The upper bounds follow from combining Proposition \ref{prop: error decomposition}, Lemma \ref{lem: epsilon bound} and Lemma \ref{lem: linear bounds}. 

Next we show the lower bound without symmetry constraints on the estimator. We consider first the case $B=1$.
Recall that the conditional expectation is the best $\mathcal{F}_t$-measurable prediction of a random variable in mean-squared error, that is, no $\mathcal{F}_t$-measurable estimator can have a better mean squared error than that of $\IE[S_{t,T}|\mathcal{F}_t]$. 
\begin{align*}
    \IE[S_{t,T}| \mathcal{F}_t] &= \sum_{s \in \mathcal{S}}\IE[1_{N_s=0}1_{N_s'>0}|\mathcal{F}_t], \\
    &= \sum_{s \in \mathcal{S}}1_{N_s = 0}(1-e^{-rM_st}).
\end{align*}
Next we compute the mean squared error in using the conditional expectation as an estimator for $S_{t,T}$ for $\mu \prec 1$. By the tower law of conditional expectation, $\IE[\IE[S_{t,T}|\mathcal{F}_t]] = \IE[S_{t,T}]$ so that the estimator is unbiased. Notice also that the unbiasedness also holds species-wise and recall the independence of the future and the past.
\begin{align*}
    \IE[(S_{t,T}-\IE[S_{t,T}|\mathcal{F}_t])^2] &= \mathbb{V}\left[\sum_{s \in \mathcal{S}}1_{N_s =0}(1_{N_s' > 0}-1+e^{-rM_st})\right], \\
    &=\sum_{s \in \mathcal{S}}\mathbb{V}[1_{N_s=0}(1_{N_s' > 0}-1+e^{-rM_st})], \\
    &= \sum_{s \in \mathcal{S}}\IE[1_{N_s=0}]\IE[(1_{N_s'>0}-1+e^{-rM_st})^2], \\
    &= \sum_{s \in \mathcal{S}}\IE[1_{N_s=0}]\mathbb{V}[1_{N_s'>0}], \\
    &= \sum_{s \in \mathcal{S}}e^{-M_st}(1-e^{-rM_st})e^{-rM_st}.
\end{align*}
Again as in the comment following Lemma \ref{lem: linear bounds} we may make the right hand side as close as we wish to $rt$ by taking a distribution on many rare species. For $B > 1$, use speciation.  

We next show the lower bound under the SEP assumption for $\mu \prec 1$. By Proposition \ref{prop:symmetries} the estimator is linear with some choice of $H$. Computing the MSE for a linear estimator gives
        \begin{align*}
        \IE[|S_{t,T}-\hat{S}^{(H)}_{t,T}|^2] &= \mathbb{V}[S_{t,T}-\hat{S}^{(H)}_{t,T}]+\IE[S_{t,T}-\hat{S}_{t,T}^{(H)}]^2 \\
        &= \sum_{s \in \mathcal{S}}\IE[1_{N_s=0}1_{N_s'>0}+H(N_s)^2]\\&-\sum_{s \in \mathcal{S}} \IE[1_{N_s=0}1_{N_s' >0}-H(N_s)]^2\\
        &+\left(\sum_{s \in \mathcal{S}}\IE[1_{N_s=0}1_{N_s' >0}-H(N_s)]\right)^2.
        \end{align*}

Let $p$ be the reciprocal of an integer and consider the uniform distribution on $\frac{1}{p}$ species. We intend on taking $p \rightarrow 0$ through reciprocals of integers. Studying the representation of the MSE above one sees that in the second term one has $\frac{1}{p}$ terms, which, owing to the square, each depend on $p$ only from second order, meaning that we can make the second term arbitrarily small. In taking $p \rightarrow 0$ we can make the event that there are no duplicate samples up to time $T$ arbitrarily likely. Bringing the sums inside the expectations we see that this implies the lower bound
    \begin{align*}
        \sup_\mu \IE[|S_{t,T}-\hat{S}^{(H)}_{t,T}|^2] &\geq rt+tH_1^2+(rt-tH_1)^2 =(r-H_1)^2t^2+(r+H_1^2)t, \\
        &= (t^2+t)H_1^2-2rt^2H_1+r^2t^2+rt.
    \end{align*}
    This is a quadratic in $H_1$ with positive leading coefficient so its minimum is achieved at $H_1 = -\frac{-2rt^2}{2(t^2+t)} = \frac{rt^2}{t^2+t}$. Evaluating the quadratic at this value and tidying up the resulting expression completes the proof. For $B >1$ use speciation.

Next we show the lower bound under the assumption that the estimator both have the SEP and be atemporal. Let $g_H(x)$ be the exponential generating function of the sequence $H$. Then
\begin{align*}
    \IE[S_{t,T}-\hat{S}_{t,T}^{(H)}] = \sum_{s \in \mathcal{S}}e^{-M_st}(1-e^{-rM_st}-g_H(M_st)).
\end{align*} 

Take a uniform distribution on $K$ species and send $K \rightarrow \infty$ as before. We see that this gives $\sup_{\mu}|\IE[S_{t,T}-\hat{S}^{(H)}_{t,T}]| \geq t|r-H_1|$. Thus if $H_1 \neq r$ the result is immediate and we may assume $H_1 = r$. Now pick some $x$ such that $g_H(x) \neq 1-e^{-rx}$. If this is not possible then by the identity theorem for analytic functions and the fact that convergent power series are absolutely convergent the estimator is the GTE or has infinite MSE.

At time $t$ take $\left\lfloor \frac{t}{x} \right\rfloor$ species with probability $\frac{x}{t}$ then assign the remaining mass by, for some $p > 0$, creating more species in such a way that no species is more probable than $p$. As we take $p$ to zero, we can make the impact on the bias of the leftover mass as small as we like since $H_1 = r$. Thus, \begin{align*}
\sup_{\mu}|\IE[S_{t,T}-\hat{S}_{t,T}^{(H)}]| &\geq \left\lfloor \frac{t}{x}\right\rfloor e^{-x} |g_H\left(t\frac{x}{t}\right)-1+e^{-r\frac{x}{t}t}| \\ &\geq \left(\frac{t}{x}-1\right)e^{-x}|g_H(x)-1+e^{-rx}|,\end{align*}
which grows linearly in $t$. This implies the result since the MSE is lower bounded by the squared absolute bias. For $B>1$, use speciation.
\end{proof}

\begin{proof}[Proof of Proposition \ref{prop: epsilon-connectedness}]
    Let $f(t) := e^{-bt}(e^{at}-1)$ for positive reals with $a < b$. Differentiating and evaluating $f$ at the maximizer one gets that the maximum is
    \begin{align*}
        \left(\frac{b-a}{b}\right)^\frac{b}{a}\frac{a}{b-a} = \left(\frac{b-a}{b}\right)^{\frac{b}{a}-1}\frac{a}{b} \leq \frac{a}{b}.
    \end{align*}
    It is a standard property of independent exponential random variables (the waiting times of a Poisson process) that the probability that the minimum is realized by one is the ratio of its intensity to the sum of the intensities. Thus $P(D_{x,y}) = \frac{M_{x \cap y}}{M_{x \cup y}}$. Combining these facts now completes the proof.
\end{proof}
\begin{proof}[Proof of Proposition \ref{prop: epsilon-estimator}]
    For each $x,y$ we have
    \begin{align*}&\IE\left[(r^{2N_{x \cap y}}-(-r)^{N_{x \cap y}})(-r)^{N_{x \backslash y}+N_{y \backslash x}}\right]\\
    &= \IE[r^{2N_{x \cap y}}-(-r)^{N_{x \cap y}}]\IE[(-r)^{N_{x \backslash y}+N_{y \backslash x}}], \\
    &= (e^{(r^2-1)M_{x \cap y}t}-e^{-(r+1)M_{x \cap y}t})e^{-(r+1)(M_{x \backslash y}+M_{y \backslash x})t}, \\
    &= e^{-(r+1)M_{x \cup y}t}(e^{r(1+r)M_{x \cap y}t}-1).
    \end{align*}
where we used independence of $N_{x \cap y}, N_{x \backslash y}, N_{y \backslash x}$ and the fact that a Poisson random variable with intensity $\lambda$ has probability generating function $e^{\lambda(z-1)}$. Summing over $x,y$ gives the result.
\end{proof}
\begin{proof}[Proof of Proposition \ref{Prop: epsilon-perfect-pairs}]
For each $x,y$
    \begin{align*}
        e^{-(1+r)M_{x \cup y} t}(e^{r(r+1)M_{x \cap y}t}-1) &\leq e^{-(1+r)M_{x \cup y} t}r(r+1)M_{x \cap y}te^{r(r+1)M_{x \cap y}t}, \\
        &\leq r(r+1)M_{x \cap y}te^{-M_{x \cap y}t} \\
        &=r(r+1)P(N_{x \cap y}=1),
    \end{align*} where we applied the mean value theorem in the first line, the assumption in the second, and recalling the probability that a Poisson random variable takes the value one in the third. Summing over all pairs $x,y$ now completes the proof.
\end{proof}

\subsection{Intermediate Future}\label{sec: Intermediate Future}

For the lower bound in Theorem \ref{thm: opt powerhouse} we will need the following two lemmas.

\begin{lem} \label{lem: lower bias build}
    \begin{align*}
    &\sup_{1 \prec \mu \prec 1}\IE[S_{t,T}-\hat{S}_{t,T}]^2 \geq \frac{1}{16}Y_b.\\
    &\sup_{\mu \prec 1}\IE[S_{t,T}-\hat{S}_{t,T}]^2 \geq \frac{1}{4}Y_b.
\end{align*}
\end{lem}
\begin{proof}
We will lower bound the supremum by finding a $\mu$ with large squared bias. Because the supremum over $p$ is over a compact set, the supremum is attained at some $p^*$ by continuity. If $p^*=0$ then taking a uniform distribution on $\frac{1}{p}$ symbols and then sending $p$ to zero through reciprocals of integers gives the result. We may therefore focus on $p^* > 0$. 

For the $\mu \prec 1$ result, create $\lfloor \frac{1}{p^*}\rfloor$ species with probability $p^*$, with total probability mass $m_b := \lfloor \frac{1}{p^*}\rfloor p^*$. Now, for the $\mu \prec 1$ result assign the leftover probability mass $1-m_b$ to the empty set. This probability mass does not affect the bias, in particular there is no cancellation due to it. Therefore, in this case,

\begin{align*}
    \sup_\mu |\IE[S_{t,T}-\hat{S}^{(H)}_{t,T}]| &\geq \left \lfloor \frac{1}{p^*} \right\rfloor e^{-p^*t}|1-e^{-rp^*t}-g_H(p^*t)|, \\
    & \geq p^*\left\lfloor \frac{1}{p^*} \right\rfloor \frac{e^{-p^*t}}{p^*}|1-e^{-rp^*t}-g_H(p^*t)|, \\
    & \geq \frac{1}{2}\frac{e^{-p^*t}}{p^*}|1-e^{-rp^*t}-g_H(p^*t)|.
\end{align*}
and the result follows from squaring.

Turning now to the $1 \prec \mu \prec 1$ result, suppose that there exists a $p' \leq p^*$ such that the following hold.
\begin{itemize}
    \item $
    \frac{e^{-p't}}{p'}|1-e^{-rp't}-g_H(p't)| \geq \frac{1}{2}\frac{e^{-p^*t}}{p^*}|1-e^{-rp^*t}-g_H(p^*t)|$.
    \item $1-e^{-rpt}-g_H(pt)$ has no root in $(0,p')$.
\end{itemize}

 Then take $\lfloor \frac{1}{p'} \rfloor$ species each with probability $p'$, allocating a total mass of $p'\lfloor \frac{1}{p'} \rfloor \geq \frac{1}{2}$ this way. For the remaining mass, allocate it however. Because there is no root there is no sign change so it will only contribute to the bias. Thus we have that in this case, the constructed $\mu$ witnesses that

\begin{align*}
    \sup_\mu |\IE[S_{t,T}-\hat{S}^{(H)}_{t,T}]| &\geq p'\left \lfloor \frac{1}{p'} \right\rfloor\frac{e^{-p't}}{p'}|1-e^{-rp't}-g_H(p't)|, \\
    & \geq p'\left\lfloor \frac{1}{p'} \right\rfloor \frac{1}{2}\frac{e^{-p^*t}}{p^*}|1-e^{-rp^*t}-g_H(p^*t)|, \\
    & \geq \frac{1}{4}\frac{e^{-p^*t}}{p^*}|1-e^{-rp^*t}-g_H(p^*t)|.
\end{align*}
Now for this case the lemma follows by squaring and we may thus suppose that no such $p'$ exists. In this case, begin constructing a $\mu$ by creating $\left\lfloor \frac{1}{p^*}\right\rfloor$ species with probability $p^*$, assigning a total mass of $p^*\left\lfloor \frac{1}{p^*}\right\rfloor \geq \frac{1}{2}$ this way. If $1-e^{-rpt}-g_H(pt)$ has no root in $(0,p^*)$, then $p^*$ would have the properties of $p'$, which was presumed not to exist. Thus $1-e^{-rpt}-g_H(pt)$ has a root in $(0,p^*)$. Let $p_0$ be the smallest root (if there are roots arbitrarily close to but not equal to zero an analyticity argument gives the result trivially). Repeatedly create species with probability $p_0$. This has no impact on the bias and ensures that the leftover mass $p_L$ is now smaller than $p_0$. Now create a single species with probability $p_L$. By the assumed non-existence of $p'$ we have that for all $p \in (0,p_0)$,
\begin{align*}
    \frac{e^{-pt}}{p}|1-e^{-rpt}-g_H(pt)| \leq \frac{1}{2}\frac{e^{-p^*t}}{p^*}|1-e^{-rp^*t}-g_H(p^*t)|,
\end{align*}
in particular, this holds for $p_L$. Thus this $\mu$ witnesses that if $p'$ does not exist,
\begin{align*}
\sup_\mu |\IE[S_{t,T}-\hat{S}^{(H)}_{t,T}]| &\geq p^*\left\lfloor \frac{1}{p^*}\right\rfloor  \frac{e^{-p^*t}}{p^*}|1-e^{-rp^*t}-g_H(p^*t)|\\
&-p_L\frac{e^{-p_Lt}}{p_L}|1-e^{-rp_Lt}-g_H(p_Lt)|, \\
&\geq p^*\left\lfloor \frac{1}{p^*}\right\rfloor \frac{e^{-p^*t}}{p^*}|1-e^{-rp^*t}-g_H(p^*t)|\\
&-p^*\left\lfloor \frac{1}{p^*}\right\rfloor\frac{1}{2}\frac{e^{-p^*t}}{p^*}|1-e^{-rp^*t}-g_H(p^*t)|, \\
&=\frac{1}{2}p^*\left\lfloor \frac{1}{p^*}\right\rfloor \frac{e^{-p^*t}}{p^*}|1-e^{-rp^*t}-g_H(p^*t)|, \\
&\geq \frac{1}{4}\frac{e^{-p^*t}}{p^*}|1-e^{-rp^*t}-g_H(p^*t)|.
\end{align*}
Squaring we see that the lemma holds also in this case and the proof is complete.
\end{proof}

Let $q^*$ be the point where the supremum in $Y_v$ is attained, which exists by continuity and compactness and let $m_v = q^* \lfloor \frac{1}{q^*}\rfloor$.
\begin{lem} \label{lem: lower variance build}
If $q^* \leq \frac{1}{2}$, then
\begin{align*}
    \sup_{1 \prec \mu \prec 1}\IE[|S_{t,T}-\hat{S}_{t,T}|^2]  \geq \frac{2}{3}Y_v-\frac{2}{3}Y_b.
\end{align*}
Further, if $q^* > \frac{1}{2}$ then
\begin{align*}
    \sup_{1 \prec \mu \prec 1}\IE[|S_{t,T}-\hat{S}_{t,T}|^2] &\geq  \frac{1}{2}Y_v-\frac{1}{2}Y_b.
\end{align*}
In either case,
\begin{align*}
    \sup_{\mu \prec 1}\IE[|S_{t,T}-\hat{S}_{t,T}|^2] \geq \frac{1}{2}Y_v.
\end{align*}
\end{lem}
\begin{proof} Our aim is to construct a distribution with a substantial mean square error. If $q^*=0$ then one sees that we may take a uniform distribution over $\frac{1}{q}$ symbols and send $q$ to zero through reciprocals of integers. Thus we may assume $q^*>0$. By independence,
    \begin{align} \label{eq: decomp stars}
        \sup_{\mu}\IE[|S_{t,T}-\hat{S}_{t,T}^{(H)}|^2] &= \sup_\mu (\IE[\sum_{s \in \mathcal{S}}1_{N_s=0}1_{N_s'>0}+H(N_s)^2]  +\sum_{x,y}b_xb_y).
    \end{align}
Create $\lfloor \frac{1}{q^*}\rfloor$ species with probability $q^*$. For $1 \prec \mu \prec 1$ use the remaining probability mass to create a single species with probability $q' := 1-\lfloor \frac{1}{q^*}\rfloor q^* = 1-m_v$. Let $b(p) := e^{-pt}(1-e^{-rpt}-g_H(pt))$. We begin by treating the off-diagonal term. For $\mu$ of the type discussed, one sees that \begin{align} \label{eq: quad spot}
\sum_{x,y}&\IE[1_{N_x=0}1_{N_x'>0}-H(N_x)]\IE[1_{N_y=0}1_{N_y'>0}-H(N_y)] = 2b(q')b(q^*)\left\lfloor \frac{1}{q^*}\right\rfloor\nonumber\\
&+\left\lfloor \frac{1}{q^*}\right\rfloor\left(\left\lfloor \frac{1}{q^*}\right\rfloor-1\right)b(q^*)^2.
\end{align}
In the case $\mu \prec 1$ do not create the extra species, instead putting the remaining probability mass on the empty set. The cross terms sum is then a sum of squares and so trivially non-negative.
Setting aside for the moment the case $q^* > \frac{1}{2}$ we see that this is a quadratic in $b(q^*)$ with positive leading coefficient. Thus its minimum is achieved at

\begin{align*}
    b(q^*) = \frac{-2b(q')\left\lfloor \frac{1}{q^*}\right\rfloor}{2\left\lfloor \frac{1}{q^*}\right\rfloor\left(\left\lfloor \frac{1}{q^*}\right\rfloor-1\right)} = -\frac{b(q')}{\left\lfloor \frac{1}{q^*}\right\rfloor-1}.
\end{align*}
 Substituting this in place of $b(q^*)$ in (\ref{eq: quad spot}) we obtain \begin{align*}
    \sum_{x,y}&\IE[1_{N_x=0}1_{N_x'>0}-H(N_x)]\IE[1_{N_y=0}1_{N_y'>0}-H(N_y)] \geq -b(q')^2 \frac{\left\lfloor \frac{1}{q^*}\right\rfloor}{\left\lfloor \frac{1}{q^*}\right\rfloor-1} \\
    &\geq -Y_b (1-m_v)\frac{\left\lfloor \frac{1}{q^*}\right\rfloor}{\left\lfloor \frac{1}{q^*}\right\rfloor-1}.\\
\end{align*}
where we used that for any $p$, $(\frac{b(p)}{p})^2 \leq Y_b$ so that $|b(p)| \leq p \sqrt{Y_b}$. Now the result follows by reasoning about which intervals the floor functions are constant on.

For $q^* > \frac{1}{2}$ one sees that $q' = 1-q^*$ and now using the same bound on $|b(p)|$ gives a lower bound of $-2q^*(1-q^*)Y_b$ on the RHS of (\ref{eq: quad spot}). Again one finishes by some floor-function analysis.

Turning our attention to the diagonal term in (\ref{eq: decomp stars}) we get \begin{align*}
    \IE[\sum_{s \in \mathcal{S}}1_{N_s=0}1_{N_s'>0}+H(N_s)^2] &\geq \left\lfloor \frac{1}{q^*}\right\rfloor e^{-q^*t}(1-e^{-rq^*t}+g_{H^2}(q^*t)), \\
    &= m_v Y_v.
\end{align*} from the uniform part of the constructed measure and that the lone species with probability $q'$ can only contribute positively.
\end{proof}

\begin{proof}[Proof of Theorem \ref{thm: opt powerhouse}] 
We first prove the upper bound. By direct calculation, with $\mathbb{C}$ here denoting covariance,

\begin{align*}
    \IE[|S_{t,T}-\hat{S}_{t,T}^{(H)}|^2] &= \IE[\sum_{s \in \mathcal{S}}1_{N_s=0}1_{N_s'>0}-H(N_s)]^2 \\
    &+\sum_{s }\mathbb{V}[1_{N_s=0}1_{N_s'>0}-H(N_s)] \\
    &+ \sum_{x,y}\mathbb{C}[1_{N_x=0}1_{N_x'>0}-H(N_x),1_{N_y=0}1_{N_y'>0}-H(N_y)]
\end{align*}
We bound each of these terms in order.
\begin{align*}
    \sum_{s \in \mathcal{S}}b_s &= \sum_{s \in \mathcal{S}}e^{-M_st}(1-e^{-rM_st}-g_H(M_st)),\\ &= \frac{\sum_{s \in \mathcal{S}}M_s}{\sum_{s \in \mathcal{S}}M_s}\sum_{s \in \mathcal{S}}e^{-M_st}(1-e^{-rM_st}-g_H(M_st)), \\
    &\leq \sum_{s \in \mathcal{S}}M_s \times\sup_{s \in \mathcal{S}}\frac{e^{-M_st}(1-e^{-rM_st}-g_H(M_st))}{M_s},\\
    &\leq B \sup_{p \in [0,1]} \frac{e^{-pt}(1-e^{-rpt}-g_H(pt))}{p}.
\end{align*}
The negative bias is treated similarly. Next we focus on the variance. \begin{align*}
\sum_{s \in \mathcal{S}}\mathbb{V}[1_{N_s=0}1_{N_s'>0}-H(N_s)] &\leq \sum_{s \in \mathcal{S}}\IE[1_{N_s=0}1_{N_s'>0}+H(N_s)^2], \\
&= \sum_{s \in \mathcal{S}}e^{-M_st}(1-e^{-rM_st}+g_{H^2}(M_st)).
\end{align*}
Extending by $\sum_{s \in \mathcal{S}}M_s$ as before one sees that the variance is bounded by $B\sup_{q \in [0,1]}\frac{e^{-qt}}{q}(1-e^{-rqt}+g_{H^2}(qt))$.

We move on to bounding the sum over correlations. To make the connection to the theory of maximal correlation clear, let $$h((N_x,N_x')):=1_{N_x=0}1_{N_x'>0}-H(N_x)$$ and let $R(X,Y)$ denote the maximal correlation between a function of $X$ and a function of $Y$ as in \cite{Maximal_Correlation}.
\begin{align*}
    &\mathbb{C}[1_{N_x=0}1_{N_x'>0}-H(N_x),1_{N_y=0}1_{N_y'>0}-H(N_y)] \\
    &= \mathbb{C}[h(N_x,N_x'),h(N_y,N_y')] \\
    &\leq R((N_x,N_x'),(N_y,N_y')) \sqrt{\mathbb{V}[h(N_x,N_x')]\mathbb{V}[h(N_y,N_y')]} \\
    &= \max\{R(N_x,N_y),R(N_x',N_y')\} \sqrt{\mathbb{V}[h(N_x,N_x')]\mathbb{V}[h(N_y,N_y')]} \\
    &= \max\{\rho(N_x,N_y),\rho(N_x',N_y')\}\sqrt{\mathbb{V}[h(N_x,N_x')]\mathbb{V}[h(N_y,N_y')]} \\
    &= \max\{\frac{M_{x \cap y}}{\sqrt{M_x}\sqrt{M_y}},\frac{M_{x \cap y}}{\sqrt{M_x}\sqrt{M_y}}\}\sqrt{\mathbb{V}[h(N_x,N_x')]\mathbb{V}[h(N_y,N_y')]}, \\
    &\leq \frac{M_{x \cap y}}{\sqrt{M_x}\sqrt{M_y}}\sqrt{M_x Y_v M_yY_v}, \\
    &= M_{x \cap y}Y_v.
\end{align*}
where the max arises from the application of the Csáki--Fischer identity, Theorem 6.2 of \cite{Correlation_Identity}. That for bivariate-Poisson the maximal correlation is the correlation of the underlying random variables is the last sentence of section 1.1 of \cite{Maximal_Correlation}. 
Evaluating the correlation is a short (omitted) computation. The last estimate merely employs the fact that $\mathbb{V}[1_{N_x=0}1_{N_x'>0}-H(N_x)] \leq e^{-M_xt}(1-e^{-rM_xt}+g_{H^2}(M_xt)) \leq M_x Y_v$.

Because a set of size $|A|$ can contribute to at most $2 \binom{|A|}{2} = |A|(|A|-1)$ pairs we have $\sum_{x,y}M_{x \cap y} \leq B(B-1) \sum_{A \in 2^\mathcal{S}}\mu(A) = B(B-1)$ and consequently

\begin{align*}
\sum_{x,y}\mathbb{C}[1_{N_x=0}1_{N_x'>0}-H(N_x),1_{N_y=0}1_{N_y'>0}-H(N_y)] \leq B(B-1)Y_v
\end{align*}

Now combining the estimates on the bias, sum of variances and sum of covariances gives the upper bound in the theorem.

We move on to the lower bound. Since the squared bias lower bounds the MSE, Lemmas \ref{lem: lower bias build} and \ref{lem: lower variance build} imply, respectively that
\begin{align*}
    \frac{2}{3}\sup_{\mu \prec 1}\IE[|S_{t,T}-\hat{S}_{t,T}|^2] &\geq \frac{2}{3}\frac{1}{4} Y_b, \\
    \frac{1}{3}\sup_{\mu \prec 1}\IE[|S_{t,T}-\hat{S}_{t,T}|^2] &\geq \frac{1}{3}\frac{1}{2}Y_v.
\end{align*}
Adding these inequalities now gives the $\mu \prec 1$ result.

Next we treat $1\prec \mu \prec 1$. Similarly to the above, if $q^* > \frac{1}{2}$,

\begin{align*}
    \frac{16}{17}\sup_{1 \prec \mu \prec 1}\IE[|S_{t,T}-\hat{S}_{t,T}|^2] &\geq \frac{16}{17}\frac{1}{16} Y_b, \\
    \frac{1}{17}\sup_{1 \prec \mu \prec 1}\IE[|S_{t,T}-\hat{S}_{t,T}|^2] &\geq \frac{1}{17}\frac{1}{2}Y_v-\frac{1}{17}\frac{1}{2}Y_b.
\end{align*}
and the result follows by adding the inequalities. If $q^* \leq \frac{1}{2}$,

\begin{align*}
    \frac{65}{68}\sup_{1 \prec \mu \prec 1}\IE[|S_{t,T}-\hat{S}_{t,T}|^2] &\geq \frac{65}{68}\frac{1}{16} Y_b \geq \frac{1}{17}Y_b, \\
    \frac{3}{68}\sup_{1 \prec \mu \prec 1}\IE[|S_{t,T}-\hat{S}_{t,T}|^2] &\geq \frac{3}{68}\frac{2}{3}Y_v-\frac{3}{68}\frac{2}{3}Y_b.
\end{align*}
and addition yields the result. For $B > 1$, speciate.
\end{proof}

\begin{proof}[Proof of Proposition \ref{prop: convexity}]
By monotonicity of $x \mapsto x^2$ we may take the square inside the first supremum. Because sums of convex functions are convex and suprema preserve convexity in order to get convexity of $G(H)$ it suffices to establish convexity for the expressions inside the suprema for given $p,q$. We begin with the first term. Because $g_H(x)$ is linear in $H$ we see that for $p \neq 0$ the first part is the square of an affine function of $H$, which is convex. For $p=0$ the expression simplifies to $t^2(r-H_1)^2$ which is again the square of an affine function of $H$. We turn to the second term. Note that $g_{H^2}(x)$ is increasing in $H$ in the sense that if $H'$ is a sequence such that $H_i^2 \leq {H_i'}^2$ for all $i$ then for any non-zero $x$, $g_{H^2}(x) \leq g_{{H'}^2}(x)$. From this and the  convexity of $x \mapsto x^2$ we get that for $q \neq 0$
\begin{align*}
    g_{(\theta H^{(1)}+(1-\theta)H^{(2)})^2}(qt) \leq g_{\theta {H^{(1)}}^2+(1-\theta){H^{(2)}}^2}(qt) = \theta g_{{H^{(1)}}^2}(qt)+(1-\theta)g_{{H^{(2)}}^2}(qt).
\end{align*}
Adding and multiplying by a constant preserves this convexity, so we have convexity of the second term for $q \neq 0$. For $q=0$ the expression simplifies to $t(r+H_1^2)$ which is convex in $H$.

For strictness, let $H^{(1)},H^{(2)}$ be as hypothesized. Then, for some $q^* \neq 0$
\begin{align*}
    &\sup_{q \in [0,1]}\frac{e^{-qt}}{q}(1-e^{-rqt}+g_{(\theta H^{(1)}+(1-\theta)H^{(2)})^2}(qt)),\\
    &=\frac{e^{-q^*t}}{q^*}(1-e^{-rq^*t}+g_{(\theta H^{(1)}+(1-\theta)H^{(2)})^2}(q^*t)).
\end{align*}
Applying strict convexity at this $q^* \neq 0$ now propagates to give the result as sums of convex and strictly convex functions are strictly convex.
\end{proof}
\begin{proof}[Proof of Theorem \ref{thm: existence and uniqueness}]
Define the inner product $\langle a,b \rangle_t := \sum_{i=1}^\infty \frac{a_ib_it^i}{i!}$ and the Hilbert space $\mathcal{H}_t$ to be the space of sequences which have $H_0=0$ and $\langle H,H \rangle_t < \infty $. Let $\boldsymbol{p}=\{p^i\}_{i=1}^\infty$. Then $||\boldsymbol{p}||^2 = \sum_{i=1}^\infty p^{2i} \frac{t^i}{i!} = e^{p^2t}-1 < \infty$ so $\boldsymbol{p} \in \mathcal{H}_t$.
\begin{align*}
    g_H(pt) = \sum_{i=1}^\infty H_i\frac{(pt)^i}{i!} = \langle H,\boldsymbol{p} \rangle_t. 
\end{align*}
By the Cauchy--Schwarz inequality we have that \begin{align}\label{eq: cont of g_H}|g_H(pt)| \leq \sqrt{e^{p^2t}-1}||H||,
\end{align}and so $g_H(pt)$ is a bounded linear functional and hence continuous. Let $P_p: \mathcal{H}_t \to \mathcal{H}_t$ be defined by $(P_pH)_i= p^{i/2}H_i$. Then,
\begin{align*}
||P_pH||^2 = \sum_{i=1}^\infty p^iH_i^2\frac{t^i}{i!} \leq p||H||^2.
\end{align*}
Thus $P_p$ is a bounded linear operator and hence continuous. Next notice that
\begin{align*} 
    g_{H^2}(pt)=\sum_{i=1}^\infty H_i^2 \frac{(pt)^i}{i!} = ||P_pH||^2.
\end{align*}
Since the norm is always continuous this means we have realized $g_{H^2}(pt)$ as a composition of continuous functions and it is therefore continuous.

Using the second term in the definition of $G(H)$ and evaluating the supremum at $q=1$ one gets,
\begin{align}\label{eq: G_H lower}
    G(H) &\geq \sup_{q \in [0,1]}\frac{e^{-qt}}{q} g_{H^2}(qt) \geq e^{-t}\sum_{i=1}^\infty H_i^2\frac{t^i} {i!} = e^{-t}||H||^2 \geq 0.
\end{align}
Thus for all sequences not in the Hilbert space, $G(H) = \infty$\footnote{With just a little more work in fact the above estimates actually show this to be an if and only if condition.} so to claim optimality among all sequences it suffices to consider those in $\mathcal{H}_t$.

We are now ready to show existence of a minimizer. We proceed by the direct method in the calculus of variations. Hilbert spaces are reflexive. We minimize over the full space, which is trivially weakly closed. Equation (\ref{eq: G_H lower}) shows $G(H)$ to be coercive. Suprema of continuous functionals are lower semi-continuous so the functional is a sum of convex lower semi-continuous functionals and hence itself convex and lower semi-continuous. By Corollary 3.9 of \cite{haimbrezis_2011_functional} $G(H)$ is consequently weakly lower semi-continuous. Thus $G(H)$ has a minimizer by Theorem 1.2 of \cite{struwe_2008_variational}.

We move on to uniqueness. Fix $r > 2$. For the sake of contradiction, suppose that there exists a sequence $\{t_n\}_{n=1}^\infty$ such that $t_n \rightarrow \infty$ and at least one minimizer of $G(H)$ has at least one $q^* = 0$ for each $t_n$. We keep using $t$ for time, but our intention is to consider what happens along $\{t_n\}_{n=1}^\infty$. We will show that along the sequence $G(H)$ and hence the error grows quadratically in time. This creates a contradiction between Corollary \ref{cor: optimality of H^*} and Theorem 1 of \cite{orlitsky_2016_optimal}. Note that this is not circular as Corollary \ref{cor: optimality of H^*} does not require uniqueness.

Evaluating at $q=0$ reveals the second term to be $t(r+H_1^2)$. Thus for all $H$ with $q^*=0$ we have that for all $q\in [0,1]$, \begin{align*}
\frac{e^{-qt}}{q}(1-e^{-rqt}+g_{H^2}(qt)) &\leq t(r+H_1^2), \\
 g_{H^2}(qt)&\leq qe^{qt}t(r+H_1^2)-1+e^{-rqt}.
\end{align*}
Observe that by the Cauchy--Schwarz inequality we therefore have that for all such $H$ and $x \leq t$,
\begin{align*}
    \sum_{i=2}^\infty H_i \frac{x^i}{i!} &\leq  \sqrt{\sum_{i=2}^\infty H_i^2\frac{x^i}{i!}} \sqrt{\sum_{i=2}^\infty \frac{x^i}{i!}}, \\
    &= \sqrt{g_{H^2}(x)-H_1^2x}\sqrt{e^x-x-1}, \\
    &\leq \sqrt{xe^x(r+H_1^2)-1+e^{-rx}-H_1^2x}\sqrt{e^x-x-1}.
\end{align*}
Using this gives,
\begin{align*}
    g_H(x) &= H_1x+\sum_{i=2}^\infty H_i \frac{x^i}{i!} ,\\
    &\geq H_1x-\sqrt{xe^x(r+H_1^2)-1+e^{-rx}-H_1^2x}\sqrt{e^x-x-1}. \\
\end{align*}
We aim to show the existence of an $x^*$ so that \begin{align*}
rx^*-\sqrt{x^*e^{x^*}(r+r^2)-1+e^{-rx^*}-r^2x^*}\sqrt{e^{x^*}-x^*-1} > 1-e^{-rx^*}.
\end{align*} We will show that choosing $x^*$ small enough will work. Taylor expanding gives \begin{align*}
    F(r,x) :&=rx-\sqrt{xe^{x}(r+r^2)-1+e^{-rx}-r^2x}\sqrt{e^{x}-x-1}-1+e^{-rx} \\
    &= \frac{x^2}{2}\left(r^2 - \sqrt{r(3r+2)}\right) + O(x^3).
\end{align*}
To determine the sign of the leading coefficient we factorize $r^4 -r(3r+2) = r(r-2)(r+1)^2$ from which it is clear that the leading coefficient is positive for $r > 2$ and so for such $r$ a small enough $x$ will do. Possibly by making $t_0$ larger we may assume that $t > x^*$. Next let $\tilde{p} = \frac{x^*}{t}$.
\begin{align*}
    &\frac{1}{t}\sup_{p \in [0,1]}\frac{e^{-pt}}{p}|1-e^{-rpt}-g_H(pt)| \geq \frac{e^{-\tilde{p}t}}{t\tilde{p}}|1-e^{-r\tilde{p}t}-g_H(\tilde{p}t)|,\\
    & = \frac{e^{-x^*}}{x^*}|1-e^{-rx^*}-g_H(x^*)|, \\
    & \geq \frac{e^{-x^*}}{x^*}(g_H(x^*)-1+e^{-rx^*}), \\
    & \geq\frac{e^{-x^*}}{x^*} \times\\
    & (H_1x^*-\sqrt{x^*e^{x^*}(r+H_1^2)-1+e^{-rx^*}-H_1^2x^*}\sqrt{e^{x^*}-x^*-1}-(1-e^{-rx^*})), \\
    & \rightarrow \\
    &\frac{e^{-x^*}}{x^*}(rx^*-\sqrt{x^*e^{x^*}(r+r^2)-1+e^{-rx^*}-r^2x^*}\sqrt{e^{x^*}-x^*-1}-(1-e^{-rx^*})),\\
    &> 0,
\end{align*}
where the convergence may be assumed by Lemma \ref{lem: convergence of H_1}. This means that $\sup_{p \in [0,1]}\frac{e^{-pt}}{p}|1-e^{-rpt}-g_H(pt)|$ grows linearly in $t$ and hence $G(H)$ grows (at least) quadratically in $t$ (along $t_n$) for the sought contradiction. Thus no such sequence can exist and there exists some $t_0(r)$ after which there cannot be optimizers with $q^* = 0$. Next suppose that for some $t > t_0$ there are two distinct minimizers (each with $q^* \neq 0$). By convexity all convex combinations of the two are also optimizers. However, all such minimizers either (i) are better than the original supposed minimizers by strict convexity or (ii) have $q^*=0$. Either of these is a contradiction and so there can only be one minimizer. 
\end{proof}

\begin{lem} \label{lem: convergence of H_1}
    Fix $r > 0$. Let $\hat{S}_{t,T}^{(H)}$ be a linear estimator with $H_1 \nrightarrow r$ as $t \rightarrow \infty$. Then $G(H) \gtrsim t^2$ along a sequence of times $t_n \rightarrow \infty$. That is, $\limsup_{t \rightarrow \infty} G(H)/t^2 > 0$.
\end{lem}
\begin{proof}
    At $p=0$ we have $\frac{e^{-pt}}{p}|1-e^{-rpt}-g_H(pt)| = |rt-H_1t|=t|r-H_1|$. Since this term appears with a square in $G(H)$, we have $G(H) \geq t^2(r-H_1)^2$; if $|r-H_1| \nrightarrow 0$ there are $\eta > 0$ and $t_n \rightarrow \infty$ with $|r-H_1(t_n)| \geq \eta$, so $G(H) \geq \eta^2 t_n^2$ along this sequence.
\end{proof}

We will prove Theorem \ref{thm: main rate} by combining, in a straightforward way, a number of (non-trivial) results. We give the synthesis first and only then prove the necessary lemmas to make the structure of the argument clear. First we introduce all notation needed to state the key estimates. Define the Poisson probability $p_k(x):= e^{-x}\frac{x^k}{k!}$. We will use the fact that $kp_k(x)=xp_{k-1}(x)$. Define

\begin{align*}
h(x)    &:= \frac{e^{-x}(1-e^{-rx})}{x} = \int_0^r e^{-(1+s)x}ds \leq r, \\
f_H(x)  &:= h-\sum_{k=1}^\infty \frac{H_k}{k}p_{k-1}(x) = \frac{e^{-x}}{x}(1-e^{-rx}-g_H(x)),\\
B(H) &:= \max_{0 \leq x \leq t}|f_H(x)|,\\
 N(H) &:= \left(\sum_{k \geq 1}H_k^2k^{-3/2}\right)^{1/2}.
\end{align*}
Let $\mathcal{H} = \{H: \sum_{k=1}^\infty k^{-3/2}H_k^2 < \infty\}$ be the Hilbert space of sequences on which $N(H)$ is finite. Let $\epsilon = t^{-1/2}$. Let $Q$ be the Cartesian product of the set of signed measures $\nu$ on $[0,t]$ with $||\nu|| \leq 1$ and the set of sequences with $\sum_{k=1}^\infty k^{3/2}a_k^2 \leq 1$. Let $\mathcal{M}_\epsilon[0,t]$ denote the set of signed measures $\nu$ on $[0,t]$ satisfying $||\nu|| \leq 1$ and $\sum_{k=1}^\infty k^{-1/2} (\int_0^t  p_{k-1}(x)d \nu)^2 \leq \epsilon^2$. Define $\mathcal{M}_\epsilon[0,\infty)$ analogously. 
 Notice that, extending by zero, we have a natural inclusion of $\mathcal{M}_\epsilon[0,t]$ in $\mathcal{M}_\epsilon[0,\infty)$. 
Define  \begin{align*}
    \mathcal{T}: \mathcal{M}_\epsilon[0,\infty)&\to (\mathbb{C} \to \mathbb{C}) \\
    \nu &\mapsto  \int_0^\infty e^{-x(1-z)}d\nu(x),\\
    \end{align*}
where we are free to let $\mathcal{T}$ act on elements of $\mathcal{M}_\epsilon[0,t]$ by the aforementioned inclusion. Define the function class $\mathcal{F} = \mathcal{T}(\mathcal{M}_\epsilon[0,\infty))$ to be the image of $\mathcal{T}$.

\begin{proof}[Proof of Theorem \ref{thm: main rate}] We combine the following estimates in the natural way. Namely, we recall that $G(H)$ bounds the MSE, take the infimum over $\mathcal{H} \subset \mathcal{H}_t$ in the first two lines and then use the next to upper bound the previous. 
\begin{align*}
    G(H) &\leq rt+t^2(B(H)+\epsilon N(H))^2. && \text{(Lemma \ref{lem: GH representation})}\\
    B(H)+\epsilon N(H) &= \sup_{(\nu,a) \in Q} \underbrace{\int_0^t f_H d\nu +\epsilon \sum_{k \geq 1} H_ka_k}_{=: L(H,(\nu,a))}. && \text{(Lemma \ref{lem: BNL})}\\
    \inf_{H \in \mathcal{H}} \sup_{(\nu,a) \in Q} L(H,(\nu,a)) &= \sup_{(\nu,a) \in Q} \inf_{H \in \mathcal{H}} L(H,(\nu,a)). && \text{(Lemma \ref{lem: duality})}\\
    \sup_{(\nu,a) \in Q} \inf_{H \in \mathcal{H}} L(H,(\nu,a)) &= \sup_{\nu \in \mathcal{M}_{\epsilon}[0,t]}\int_0^t h d\nu. && \text{(Lemma \ref{lem: inf L to sup nu})}\\
    \sup_{\nu \in \mathcal{M}_{\epsilon}[0,t]}\int_0^t h d\nu &\leq \sup_{f \in \mathcal{F}} \int_{0}^r |f(-s)| ds. && \text{(Lemma \ref{lem: nu to F})}\\
    \sup_{f \in \mathcal{F}} \int_{0}^r |f(-s)| ds &\lesssim \min\left\{r,\frac{(1+r)^2}{\log(t)}\right\}K^{\frac{1}{1+r}}t^{-\frac{1}{1+r}}. && \text{(Lemma \ref{lem: assembly})}
\end{align*}

Combining the estimates in this way yields
\begin{align*}
    \frac{G(H^*)}{(rt)^2}  \lesssim \frac{1}{rt}+\min\left\{1,\frac{(1+r)^2}{r\log(t)}\right\}^2K^{\frac{2}{1+r}}t^{-\frac{2}{1+r}},
\end{align*}
where we note that (\ref{eq: r hypothesis}) was used to invoke Lemma \ref{lem: assembly}. We now let the second term absorb the first.

\begin{align*}
    \frac{1}{rt} = \frac{1}{r}t^{-\frac{r-1}{1+r}}t^{-\frac{2}{1+r}} \leq \log(t)^{-4}t^{-\frac{2}{1+r}} \leq (\min\{1,\frac{(r+1)^2}{r\log(t)}\})^{2}K^{\frac{2}{1+r}}t^{-\frac{2}{1+r}},
\end{align*}
where we used that

\begin{align*}
K &= \sqrt{\log(t)}(r+1)(r+3)\min\left\{\frac{1}{r-1}, \log(t)\right\} \\
&\geq \sqrt{\log(t)}\min\left\{\frac{(r+1)(r+3)}{r-1}, 8\log(t)\right\} \geq 1
\end{align*}
and that
\begin{align*}
\min\left\{1, \frac{(1+r)^2}{r\log(t)}\right\} &\geq \min\left\{1, \frac{4}{\log(t)}\right\} \geq \frac{1}{\log(t)}.
\end{align*}

Now recalling that $G(H)$ bounds the MSE by Theorem \ref{thm: opt powerhouse} completes the proof.

\end{proof}

\begin{lem} \label{lem: poisson prob estimates}
    \begin{align*}
        \sup_{x \geq 0} p_k(x) \leq (k+1)^{-1/2} 
    \end{align*}
\end{lem}
\begin{proof}
    For $k=0$ the result is trivial. For $k\geq 1$ calculus gives that the maximum of $p_k(x)$, achieved at $x=k$, equals $\frac{k^ke^{-k}}{k!}$. By Stirling's approximation, suitably quantified, (1) of \cite{StrilingFormula}, this is bounded by $(2\pi k)^{-1/2} \leq (k+1)^{-1/2}$.
\end{proof}

\begin{lem} \label{lem: GH representation}

\begin{align*}
    G(H) \leq rt+t^2(B(H)+\epsilon N(H))^2
\end{align*}
\end{lem}
\begin{proof}
Making the substitution $y=pt$, note that
    \begin{align*}
        Y_b(H) &= t^2B(H)^2\\
        Y_v(H) &= t \sup_{0 \leq y \leq t}[h(y)+\sum_{k=1}^\infty \frac{p_k(y)}{y}H_k^2] \\
        &= t \sup_{0 \leq y \leq t}[h(y)+\sum_{k=1}^\infty \frac{p_{k-1}(y)}{k}H_k^2]\\
        &\leq t(r+\sum_{k \geq 1}H_k^2 k^{-3/2}) \\
        &= tr+tN(H)^2
    \end{align*}
where we used Lemma \ref{lem: poisson prob estimates} for the inequality. Now combining these gives the result.
\end{proof}

\begin{lem} \label{lem: BNL}
\begin{align*}
B(H)+\epsilon N(H)  =
        \sup_{(\nu,a) \in Q} L(H,(\nu,a))
\end{align*}
\end{lem}
\begin{proof} $Q$ is a Cartesian product, the first term depends only on $\nu$ and the second only on $a$ so the supremum splits. We will show that the two resulting terms match those on the left.

\begin{align*}
    &\sup_{\nu: ||\nu|| \leq 1}\int_0^t f_H d\nu  \leq ||f_H||_\infty ||\nu|| \leq B(H), \\
    &\sup_{\nu: ||\nu|| \leq 1}\int_0^t f_H d\nu \geq \mathrm{sgn}(f_H(x^*))\int_{0}^t f_H d\delta_{x^*} = B(H),
\end{align*}
where $x^*$ is a maximizer of $|f_H|$ on $[0,t]$, so that $\sup_{\nu: ||\nu|| \leq 1}\int_0^t f_H d\nu = B(H)$. Further, for $a$ such that $\sum_{k=1}^\infty k^{3/2}a_k^2 \leq 1$, applying the Cauchy--Schwarz inequality yields,

\begin{align*}
\sum_{k=1}^\infty H_k a_k &= \sum_{k=1}^\infty (H_kk^{-3/4})(a_kk^{3/4}) \leq N(H)\times 1
\end{align*}
The choice $a_k = \frac{H_kk^{-3/2}}{N(H)}$, which satisfies $\sum_{k=1}^\infty k^{3/2}a_k^2 \leq 1$, makes $\sum_{k=1}^\infty H_k a_k = N(H)$ so indeed

\begin{align*}
    N(H) = \sup \{\sum_{k=1}^\infty H_ka_k: \sum_{k=1}^\infty k^{3/2}a_k^2 \leq 1\}
\end{align*} and the claim follows.
\end{proof}

\begin{lem} \label{lem: duality}
    \begin{align*}
       \inf_{H \in \mathcal{H}} \sup_{(\nu,a) \in Q} L(H,(\nu, a)) &= \sup_{(\nu,a) \in Q} \inf_{H \in \mathcal{H}} L(H,(\nu,a)) 
    \end{align*}
\end{lem}
\begin{proof}
The conclusion will follow by applying Theorem 5 of \cite{minimax_theorems} (therein attributed to \cite{Kneser1952}), after verifying the hypothesis of the theorem. First we consider $Q$. Let $M[0,t]$ be the space of finite signed Borel measures on $[0,t]$ with the total variation norm and let $C[0,t]$ denote the space of continuous functions $[0,t] \to \mathbb{R}$ with the sup norm. Then $M[0,t] = C[0,t]^*$ so that $\{\nu \in M[0,t]:||\nu|| \leq 1\}$ is precisely the unit ball in a dual space. Equipping the space with the weak-$*$ topology, this is a non-empty convex subset of a topological vector space, and it is compact by the Banach–Alaoglu theorem. Recognize $\{a: \sum_{k=1}^\infty k^{3/2} a_k^2 \leq 1\}$ as a unit ball in a weighted $\ell^2$ space. We equip it with the weak topology. Because Hilbert spaces are their own dual we may again apply the Banach–Alaoglu theorem to conclude that it is a non-empty compact and convex subset of a topological vector space. Thus $Q$, under the product topology, is a non-empty compact and convex subset of a topological vector space. Meanwhile, $\mathcal{H}$ is non-empty and a vector space and therefore convex.

Next we study $L$. Expanding out its definition one sees it to be affine in both variables. It can be seen to be finite by repeated application of the Cauchy--Schwarz inequality as well as Lemma \ref{lem: poisson prob estimates}.

\begin{align}\label{eq: finiteness of L}
    |\sum_{k=1}^\infty H_ka_k| &\leq N(H)(\sum_{k=1}^\infty k^{3/2}a_k^2)^{1/2} \leq N(H) < \infty. \nonumber \\
    |\int_0^t f_H d\nu| &\leq r+\sup_{x \in [0,t]}\left|\sum_{k=1}^\infty \frac{H_k}{k}p_{k-1}(x)\right| \leq r+ \sum_{k=1}^\infty |H_k|k^{-3/2} \nonumber\\&\leq r+N(H)(\sum_{k=1}^\infty k^{-3/2})^{1/2} < \infty.
\end{align}

Now $L(H,\cdot)$ is the sum of a function of $\nu$ and a function of $a$, so it suffices to establish continuity of each term as a single-variable function. Since $\sum_{k=1}^\infty H_k^2\frac{t^k}{k!} \leq \sup_k( \frac{t^kk^{3/2}}{k!})\sum_{k=1}^\infty H_k^2k^{-3/2}$ we have $\mathcal{H} \subset \mathcal{H}_t$. Then continuity of $f_H$ follows from (\ref{eq: cont of g_H}) since $f_H(x) = \frac{e^{-x}}{x}(1-e^{-rx}-g_H(x))$ and the singularity at $x=0$ is removable. This in turn implies weak-$*$ continuity of the first term. For the second term, view it as taking the inner product with $k^{-3/2}H_k$, which is an element of the space precisely because $N(H) < \infty$, and then scaling. The second term is therefore continuous by the definition of a weak topology. Thus $L(H,\cdot)$ is continuous for each $H \in \mathcal{H}$. Invoking Theorem 5 of \cite{minimax_theorems} now gives the result.
\end{proof}

\begin{lem} \label{lem: inf L to sup nu}
    \begin{align*}
        \sup_{(\nu,a) \in Q} \inf_{H \in \mathcal{H}} L(H,(\nu,a)) &= \sup_{\nu \in \mathcal{M}_{\epsilon}[0,t]}\int_0^t h d\nu.
    \end{align*}
\end{lem}
\begin{proof}
Expanding the definition of $f_H$ and integrating term by term, which is justified by (\ref{eq: finiteness of L}) we have

\begin{align*}
    L= \int_0^t hd\nu+\sum_{k=1}^\infty (\epsilon a_k-\frac{1}{k}\int p_{k-1} d\nu )H_k
\end{align*}
    For a given $(\nu,a)$ we have $\inf_{H \in \mathcal{H}} L(H,(\nu,a)) = -\infty$ unless $a_k = \frac{1}{k\epsilon}\int p_{k-1} d\nu$ for all $k$. To see this, suppose that some $a_k$ is not its prescribed value and pick an $H$ with large $H_k$, sign chosen appropriately and zeros elsewhere. Assuming for a moment that the resulting $a$ is admissible, we may therefore restrict the supremum in $\sup_{(\nu,a) \in Q} \inf_{H \in \mathcal{H}} L(H,(\nu,a))$ to be only a supremum over $\nu$, with $a$ determined from $\nu$ via $a_k = \frac{1}{k \epsilon}\int p_{k-1} d\nu$. However, we must then make sure that $a_k$ is admissible, that is, we need to enforce $\sum_{k=1}^\infty k^{3/2}a_k^2 = \sum_{k=1}^\infty k^{-1/2}\frac{1}{\epsilon^2}(\int_0^t p_{k-1}d\nu)^2 \leq 1$. Now for such $\nu$ and associated $a$, $L(H,(\nu,a)) = \int_0^t h d\nu$, giving the result.
\end{proof}

\begin{lem} \label{lem: nu to F}
    \begin{align*}
        \sup_{\nu \in \mathcal{M}_{\epsilon}[0,t]}\int_0^t h d\nu &\leq  \sup_{f \in \mathcal{F}} \int_{0}^r |f(-s)| ds
    \end{align*}
\end{lem}
\begin{proof}
\begin{align*}
        \int_0^t h(x)d\nu(x) &= \int_0^t \int_0^r e^{-(1+s)x}ds\, d\nu(x) \\
        &= \int_0^r \int_0^t e^{-(1+s)x}d\nu(x)\,ds \\
        &= \int_0^r \mathcal{T}[\nu](-s)\,ds \\
        &\leq \sup_{f \in \mathcal{F}}\int_0^r |f(-s)|\, ds
    \end{align*}
\end{proof}
The bounds on $f \in \mathcal{F}$ will be obtained via the theory of harmonic functions in the complex plane. For subsets of the complex plane we take closures in the Riemann sphere. Fix $\delta \in (0,1)$ and let $ D_\delta := \{z \in \mathbb{C}: \text{Re}(z) < 1\}\setminus \overline {D(0,\delta)}$ where $D(0,\delta)$ denotes the disc centered at the origin with radius $\delta$. The boundary of $D_\delta$ has two components, $C := \overline {D(0,\delta)}\setminus D(0,\delta)$ and $W := \{z \in \mathbb{C}: \text{Re}(z) = 1\} \cup \{\infty \}$. We will need one quantitative bound on the associated harmonic measure, which we now state.

\begin{lem} \label{lem: harmonic measure}
    There exists a continuous function $\omega: \overline{D_\delta} \to [0,1]$ such that
\begin{itemize}
    \item $\omega$ is harmonic on $D_\delta$.
    \item $\omega$ is $1$ on $C$ and $0$ on $W$.
    \item $\omega(-s) \geq \frac{2\delta}{1+s}$ for $s > \delta$.
\end{itemize}
\end{lem}
\begin{proof}
    We construct $\omega$ explicitly. Set $a=1-\sqrt{1-\delta^2}$ and $b= 1+\sqrt{1-\delta^2}$. \begin{align*}
        \Phi(z) &:= \frac{z-a}{z-b} \\
        \omega(z) &:= \frac{\log(1/|\Phi(z)|)}{\text{artanh}(\sqrt{1-\delta^2})}
    \end{align*}
$\Phi$ has a pole at $z = b$ and a root at $z = a$. Clearly $\text{Re}(1+\sqrt{1-\delta^2}) = 1+\sqrt{1-\delta^2} > 1$, so the pole is outside of $\overline{D_\delta}$. Similarly, $|1-\sqrt{1-\delta^2}| < \delta$ so the root is inside $D(0,\delta)$. Thus $\Phi$ is holomorphic and zero-free on a neighborhood of $\overline{D_\delta}$. Since $\omega$ is the scaled logarithm of $|\Phi(z)|$ it is thus continuous on $\overline{D_\delta}$ and harmonic on $D_\delta$. On $W\setminus\{\infty\}$,

\begin{align*}
|z-a|^2 = 1-\delta^2+\text{Im}(z)^2= |z-b|^2.
\end{align*}
Thus $|\Phi(z)| = 1$ there and $\omega$ vanishes on $W\setminus\{\infty\}$. Also, $\Phi(\infty)=1$ so $\omega(\infty) = 0$.  Using $ab = \delta^2$ and then $a+b=2$ we have that for $|z|=\delta$,
\begin{align*}
    |z-a|^2 &= |z|^2-2a\text{Re}(z)+a^2 \\
    &= \delta^2-2a\text{Re}(z)+a^2\\
    &= a(b+a-2\text{Re}(z)) \\
    &= 2a(1-\text{Re}(z))
\end{align*}
and similarly $|z-b|^2=2b(1-\text{Re}(z))$. 
\begin{align*}
    |\Phi(z)|^2 = \frac{|z-a|^2}{|z-b|^2} = \frac{a}{b}= \frac{1-\sqrt{1-\delta^2}}{1+\sqrt{1-\delta^2}},
\end{align*} independently of $z$. Then $\log(1/|\Phi(z)|) = \frac{1}{2}\log(\frac{1+\sqrt{1-\delta^2}}{1-\sqrt{1-\delta^2}})$. By the identity, $\text{artanh}{(x)}=\frac{1}{2}\log(\frac{1+x}{1-x})$, it follows that $\omega=1$ on $C$. Note that the maximum principle now implies that the co-domain $[0,1]$ is valid. For real $s > \delta$,

\begin{align*}
    |\Phi(-s)| &=  \frac{1+s-\sqrt{1-\delta^2}}{1+s+\sqrt{1-\delta^2}}, \\
    &= \frac{1-\frac{\sqrt{1-\delta^2}}{1+s}}{1+\frac{\sqrt{1-\delta^2}}{1+s}}. \\
\end{align*}
Again applying the identity $\text{artanh}{(x)}=\frac{1}{2}\log(\frac{1+x}{1-x})$ gives
\begin{align*}\omega(-s) = \frac{2\text{artanh}(\frac{\sqrt{1-\delta^2}}{1+s})}{\text{artanh}(\sqrt{1-\delta^2})}
\end{align*}
We note that for $x \in (0,1)$,
\begin{align*}
    1 = \frac{d}{dx}x\leq (1-x^2)^{-1} = \text{artanh}'(x)\leq (1-x^2)^{-3/2} = \frac{d}{dx} \frac{x}{\sqrt{1-x^2}}.
\end{align*}
Along with the non-differentiated functions all vanishing at zero this implies,
\begin{align*}
    x \leq \text{artanh}(x) \leq \frac{x}{\sqrt{1-x^2}}.
\end{align*}
Applying the lower bound to the numerator and the upper bound to the denominator and noticing that $\sqrt{1-(\sqrt{1-\delta^2})^2}=\delta$ gives, for $s > \delta$
\begin{align*}
    \omega(-s) \geq 2\frac{\frac{\sqrt{1-\delta^2}}{1+s}}{\frac{\sqrt{1-\delta^2}}{\delta}} = \frac{2\delta}{1+s}.
\end{align*}
\end{proof}

\begin{lem} \label{lem: function class properties}
For any $f \in \mathcal{F}$, $\eta \in (0,\frac{1}{2})$, $\delta := 1-\eta$, $s > \delta$,
\begin{enumerate}[label=(\roman*)]
    \item $|f(z)| \leq 1 \text{ on }\{z \in \mathbb{C}: \text{Re}(z) \leq 1\}.$
    \item $|f(z)| \leq \epsilon (1-|z|^2)^{-1} \text{ on } D(0,1)$,
    \item $\log|f(-s)| \leq \dfrac{\delta}{1+s}\log\left(\dfrac{1+s}{2\delta}\left(\dfrac{\epsilon}{\eta}\right)^{2}\right)$.
    \item $\log|f(-s)| \leq \dfrac{\delta}{1+s}\log\left(\dfrac{(s+\delta)(1+s)e^{4\eta-1/2}}{2\delta(s-\delta)2\sqrt{\eta}}\epsilon^{2}\right)$.
\end{enumerate}
\end{lem}
\begin{proof}
For $(i)$, notice that the integrand in the definition of $\mathcal{T}$ satisfies $|e^{-x(1-z)}|= e^{-x(1-\text{Re}(z))} \leq 1$ for $x \geq 0$ and $\text{Re}(z) \leq 1$. Since $||\nu|| \leq 1$, the integral is bounded by the maximal value of the integrand and the result follows.

By expanding $e^{-x(1-z)}$ and then integrating term-wise we obtain

\begin{align*}
    f(z) &= \int_0^\infty e^{-x(1-z)} d\nu(x) \\
    &= \int_0^\infty e^{-x}\sum_{k=0}^\infty \frac{(xz)^k}{k!} d\nu(x), \\
    &= \sum_{k=0}^\infty z^k \int_0^\infty p_k(x) d\nu(x), \\
    &= \sum_{k=0}^\infty m_kz^k.
\end{align*}
The interchange of summation and integration is justified when $|z| \leq 1$ because $
\int_0^\infty e^{-x(1-|z|)}d|\nu| \leq ||\nu|| \leq 1$. Therefore, for $|z| < 1$,

\begin{align*}
   | f(z)| &\leq \sum_{k=0}^\infty ((k+1)^{-1/4}|m_k|)((k+1)^{1/4}|z|^k) \\
    &\leq \epsilon ({\sum_{k=0}^\infty (k+1)^{1/2} |z|^{2k}})^{1/2} \\
    &\leq  \epsilon ({\sum_{k=0}^\infty (k+1) |z|^{2k}})^{1/2} \\
    &=  \epsilon (1-|z|^2)^{-1}
\end{align*}
where in the last equality we recognized the sum of an arithmetico-geometric sequence, convergent for $|z| < 1$.

Next let $V(z)$ solve the Dirichlet problem on $D(0,\delta)$ with data $|f|^2$. It is non-negative since the boundary data is non-negative. Define $U(z) := V(\frac{\delta^2}{\overline{z}})$. It is harmonic because (i) $z \mapsto \frac{\delta^2}{z}$ is holomorphic and hence preserves harmonicity and (ii) $z \mapsto \overline{z}$ preserves harmonicity by direct computation. By the comparison principle $U$ is non-negative. We may assume that $f$ is not identically zero on $C$ as if it were, it would be zero everywhere and there would be nothing to prove. We must then have $V>0$ on $D(0,\delta)$ since it vanishing at an interior point would, by the strong maximum principle, imply it to be identically zero on $D(0,\delta)$ and hence on the boundary. Thus we have strict positivity of $U$ on $D_\delta$. For each $s > \delta$, consider 

\begin{align*}
    \log|f(z)|^2-\left[\log\left(\frac{U(-s)}{\omega(-s)}\right)-1\right]\omega(z)-\left[\frac{\omega(-s)}{U(-s)}\right]U(z).
\end{align*}

This function is sub-harmonic as the first term is sub-harmonic and the others harmonic. On $\zeta \in C$,

\begin{align*}
    &\log|f(\zeta)|^2-\left[\log\left(\frac{U(-s)}{\omega(-s)}\right)-1\right]-\left[\frac{\omega(-s)}{U(-s)}\right]|f(\zeta)|^2 \\
    &= \log\left(\frac{\omega(-s)|f(\zeta)|^2}{U(-s)}\right)+1-\frac{\omega(-s)|f(\zeta)|^2}{U(-s)} \\
    &\leq 0,
\end{align*} where in the last line we used that $\log(x) \leq x-1$ for positive $x$ and the $x=0$ case is trivial. On $W$, non-positivity holds term-wise. Since $\log|f| \leq 0, U(z) \geq 0$ and $\omega(z) \in [0,1]$, it is bounded by $|\log \frac{U(-s)}{\omega(-s)}-1|$. By the extended maximum principle, it is therefore non-positive on $D_\delta$. Evaluating at $z=-s$ gives that for $s > \delta$ 

\begin{align*}
    \log|f(-s)|^2 &\leq \left[\log\left(\frac{U(-s)}{\omega(-s)}\right)-1\right]\omega(-s)+\left[\frac{\omega(-s)}{U(-s)}\right]U(-s) \\
    &= \omega(-s)\log\left(\frac{U(-s)}{\omega(-s)}\right)-\omega(-s)+\omega(-s) \\
    &= \omega(-s)\log\left(\frac{U(-s)}{\omega(-s)}\right).
\end{align*}

Then applying the lower bound in Lemma \ref{lem: harmonic measure} gives

\begin{align} \label{eq: central harmonic bound}
    \log|f(-s)|^2 \leq \frac{2\delta}{1+s} \log\left(\frac{(1+s)U(-s)}{2\delta}\right).
\end{align}

where we were free to assume that the log is negative because $|f| \leq 1$ implies that the bound holds trivially otherwise.

The rest of the proof will bound the right hand side of (\ref{eq: central harmonic bound}) in two different ways. By the maximum principle, (ii) and the observation $1-\delta^2 = (1-\delta)(1+\delta) = \eta(2-\eta) \geq \eta$,

\begin{align*}
    U(-s) \leq \sup_C |f|^2 \leq \epsilon^2(1-\delta^2)^{-2} \leq \left(\frac{\epsilon}{\eta}\right)^2.
\end{align*}
This gives (iii).

Next we apply Harnack's inequality to bound $U(-s)$, Theorem 1.3.1 of \cite{Ransford} with, in their notation, $w=0, \rho = \delta, r = \frac{\delta^2}{s}$, $t = \pi$,

\begin{align*}
    U(-s)&=V\left(-\frac{\delta^2}{s}\right) \\
    &\leq \frac{s+\delta}{s-\delta}V(0) \\
    &= \frac{s+\delta}{s-\delta}\frac{1}{2\pi}\int_{-\pi}^\pi |f(\delta e^{i\theta})|^2d\theta
\end{align*}
where we applied the mean value property in the last line.

\begin{align*}
    \frac{1}{2\pi}\int_{-\pi}^\pi |f(\delta e^{i\theta})|^2d\theta &= \frac{1}{2\pi}\int_{-\pi}^\pi\sum_{j,k=0}^\infty m_jm_k \delta^{j+k}e^{i(j-k)\theta}d\theta \\
    &=\sum_{k=0}^\infty m_k^2 \delta^{2k} \\
    &\leq (\sup_{k \in \mathbb{Z}_{\geq 0}}\delta^{2k} (k+1)^{1/2})\sum_{k=0}^\infty (k+1)^{-1/2}m_k^2
\end{align*}
The second factor is bounded by $\epsilon^2$ since $\nu \in \mathcal{M}_\epsilon[0,\infty)$. For the first factor,
\begin{align*}
    (1-\eta)^{2k} (k+1)^{1/2}&= e^{2k \log(1-\eta)}(k+1)^{1/2}, \\
    &\leq e^{-2k\eta}(k+1)^{1/2}, \\
    &= e^{2\eta}e^{-2(k+1)\eta}(k+1)^{1/2}, \\
    &\leq e^{2\eta} (4\eta)^{-1/2}e^{-1/2},
\end{align*}
where in the last line we applied the inequality  \begin{align}\label{eq: elementary ineq}
    x^{1/2}e^{-\lambda x} \leq (2\lambda)^{-1/2}e^{-1/2},
\end{align} valid for $\lambda,x >  0$, with $x=k+1$, $\lambda = 2\eta$. (\ref{eq: elementary ineq}) holds since some calculus shows the right hand side to be the maximum of the left hand side. Now combining the bounds on the left and right factors gives (iv).
\end{proof}
The following simple computation occurs repeatedly in what follows so we state it as a lemma. Fix $\eta = \frac{2}{\log(t)}$.

\begin{lem} \label{lem: power fiddling}
    For $A \geq 1$ and $s > -1$, one has
    \begin{align*}
        (A \epsilon^2)^{\frac{\delta}{1+s}} \leq (e^2A \epsilon^2)^{\frac{1}{1+s}}.
    \end{align*}
\end{lem}
\begin{proof}
\begin{align*}
    (A\epsilon^2)^{\delta} = A\epsilon^2 (A\epsilon^2)^{-\eta} \leq A\epsilon^2  \epsilon^{-2\eta} = A\epsilon^2  t^{\eta} = A\epsilon^2 e^{\eta\log t} = e^2 A\epsilon^2 .
\end{align*}
\end{proof}

We partition the integration range $[0,r]$ and use Lemma \ref{lem: function class properties} to establish suitable bounds on $|f(-s)|$ on the different pieces. Pick $\eta = \frac{2}{\log(t)}, \delta = 1-\frac{2}{\log(t)}$ and recall that $\epsilon = t^{-1/2}$. Define $K:= (r+1)(r+3)\sqrt{\log(t)}\min \{\frac{1}{r-1},\log(t)\}$.
\begin{cor} \label{cor: f estimates}For $f \in \mathcal{F}$, $r > 1$ and $t > e^4$,
\begin{align} 
    |f(-s)| &\lesssim \epsilon \log(t) && 0 \leq s \leq \delta, \label{eq: fact near zero}\\
    |f(-s)| &\lesssim (\log(t)^2\epsilon^2)^{\frac{1}{1+s}} && \delta < s \leq \frac{\delta +r}{2}, \label{eq: circle sup bound}\\
    |f(-s)| &\lesssim \left(K \epsilon^2\right)^{\frac{1}{1+s}} && \tfrac{\delta + r}{2} \leq s \leq r, \label{eq:  harnack bound}
\end{align} all implicit constants universal over $f \in \mathcal{F}$, $r > 1$ and $t > e^4$.
\end{cor}
\begin{proof}
For $s \in [0,\delta]$, note that since $\delta < 1$, by Lemma \ref{lem: function class properties}(ii),
    \begin{align*}
        |f(-s)| \leq \epsilon(1-s^2)^{-1} \leq \epsilon(1-\delta^2)^{-1}\leq \frac{\epsilon}{\eta} = \frac{1}{2}\epsilon \log(t),
    \end{align*} as $1-\delta^2 = (1-\delta)(1+\delta) = \eta(2-\eta) > \eta.$

For $s \in (\delta, \frac{\delta+r}{2}]$ we exponentiate both sides in Lemma \ref{lem: function class properties}(iii), obtaining
\begin{align*}
    |f(-s)| \leq \left(\frac{1+s}{2\delta}\right)^{\frac{\delta}{1+s}} \left(\frac{\epsilon}{\eta}\right)^{\frac{2\delta}{1+s}}.
\end{align*}
Defining $x := \frac{2\delta}{1+s} > 0$, the first factor is $x^{-x/2}$ but differentiating and maximizing gives that for all $x>0$ one has $x^{-x/2} \leq e^{\frac{1}{2e}}$. For the second factor we take the square inside and then apply Lemma \ref{lem: power fiddling} with $A = \frac{1}{\eta^2} > 1$ and then notice that $(\frac{e^2}{4})^{\frac{1}{1+s}} \leq \frac{e^2}{4}$.

For $s \in [\frac{r+\delta}{2},r]$, apply Lemma \ref{lem: function class properties}(iv) 
\begin{align*}
    |f(-s)| \leq (A\epsilon^2)^{\frac{\delta}{1+s}}, \qquad A:= \frac{s+\delta}{s-\delta}\frac{1+s}{2\delta}\frac{e^{4\eta-1/2}}{2\sqrt{\eta}}.
\end{align*}
Each factor in $A$ is larger than one. Thus Lemma \ref{lem: power fiddling} applies and gives
\begin{align*}
    |f(-s)| \leq (e^2A \epsilon^2)^{\frac{1}{1+s}}.
\end{align*}
The first factor in $A$ is decreasing in $s$, so we may bound it by $\frac{\frac{r+\delta}{2}+\delta}{\frac{r+\delta}{2}-\delta} = \frac{r+3\delta}{r-\delta}$. The numerator is bounded from above by $r+3$ and we bound the denominator from below by noticing that $r-\delta= (r-1)+\eta \geq \max{\{r-1,\eta\}}$. For the second factor notice that we may upper bound the numerator by $1+r$ since $s \leq r$. For the numerator in the third factor, note that since $\eta \leq \frac{1}{2}$, $e^{4\eta-1/2} \leq e^{3/2}$. This pre-factor as well as the $e^2$ contributes at most a (universal) constant factor since for $s \geq 0$, $c^{\frac{1}{1+s}} \leq \max\{c,1\}$. Assembling we therefore have

\begin{align*}
    |f(-s)| \lesssim \left(\epsilon^2\frac{(r+3)(r+1)}{\max\{r-1,2/\log(t)\}}\sqrt{\log(t)}\right)^{\frac{1}{1+s}}\lesssim (K\epsilon^2)^{\frac{1}{1+s}}.
\end{align*}

\end{proof}

The following bound will be combined with the estimates on $|f(-s)|$ to bound the integral.

\begin{lem} \label{lem: integration with reciprocal power} For $q \in (0,1), b > a \geq 0$,
    \begin{align*}
        \int_a^b q^{\frac{1}{1+s}}ds \leq q^{\frac{1}{1+b}} \min\left\{b-a,\frac{(1+b)^2}{\log(1/q)}\right\}
    \end{align*}
\end{lem}
\begin{proof}
    Using convexity of $\frac{1}{1+s}$ at $s=b$ we have $\frac{1}{1+s} \geq \frac{1}{1+b}+\frac{b-s}{(1+b)^2}$ so $q^{\frac{1}{1+s}} \leq q^{\frac{1}{1+b}}q^{\frac{b-s}{(1+b)^2}}$. The first factor is constant as a function of $s$. Integrating the second factor we have, substituting $u=b-s$,
\begin{align*}
    \int_a^b q^{\frac{b-s}{(1+b)^2}} ds = \int_0^{b-a}q^{\frac{u}{(1+b)^2}}du \leq \int_0^{\infty}q^{\frac{u}{(1+b)^2}}du = \frac{(1+b)^2}{\log(1/q)},
\end{align*} giving the second branch of the minimum. For the first simply notice that $q^{\frac{b-s}{(1+b)^2}} \leq 1$.
\end{proof}

\begin{lem} \label{lem: assembly}
Let $t \geq t_0$ and $r>1$ satisfy
\begin{align} \label{eq: r hypothesis}
    \frac{r-1}{r+1}\log t \geq 4\log\log t.
\end{align}
Then
\begin{align*}
    \sup_{f \in \mathcal{F}} \int_0^r |f(-s)| ds \lesssim \min\left\{r,\frac{(1+r)^2}{\log(t)}\right\}K^{\frac{1}{1+r}}t^{-\frac{1}{1+r}}.
\end{align*}
\end{lem}

\begin{proof}
    We bound the contribution to the integral on each piece in turn. By Corollary \ref{cor: f estimates},
\begin{align*}
    \int_0^\delta |f(-s)|ds \lesssim \delta \epsilon \log(t) \leq \epsilon \log(t) = t^{-1/2}\log(t) \leq \frac{1}{\log(t)}t^{-\frac{1}{1+r}},
\end{align*} where the last inequality is equivalent to (\ref{eq: r hypothesis}), which can be seen by taking logarithms and rearranging. Let $W := \min\{r,\frac{(1+r)^2}{\log(t)}\}$. Applying (\ref{eq: circle sup bound}) and Lemma \ref{lem: integration with reciprocal power},
\begin{align*}
    \int_{\delta}^{\frac{\delta+r}{2}} |f(-s)|ds \lesssim (\log(t)^2\epsilon^2)^{\frac{2}{2+\delta+r}}W \lesssim (\log(t)^2\epsilon^2)^{\frac{2}{3+r}}W \leq t^{-\frac{1}{1+r}}W,
\end{align*}
where the last inequality can be seen to be equivalent to (\ref{eq: r hypothesis}) as after dividing by $W \neq 0$, taking logarithms and substituting $\epsilon^2 = t^{-1}$, it becomes
\begin{align*}
    \frac{4}{3+r} \log\log(t)- \frac{2}{3+r}\log(t) \leq -\frac{1}{1+r}\log(t),
\end{align*}
which, after multiplying through by $3+r$ and combining like terms is (\ref{eq: r hypothesis}).

We turn our attention to the last piece.

If $r \geq  \log(t)$, then the result holds straightforwardly but lacks content. The right hand side becomes proportional to $r$ and $|f(-s)| \leq 1$ from Lemma \ref{lem: function class properties} suffices. We may therefore assume $1 < r \leq \log(t)$. From the definition of $K$ this gives $K \lesssim \log(t)^{3/2}$ so that Lemma \ref{lem: integration with reciprocal power} applies and yields, when combined with (\ref{eq:  harnack bound}),
\begin{align*}
    \int_{\frac{\delta+r}{2}}^r |f(-s)|ds \lesssim  (K\epsilon^2)^{\frac{1}{1+r}}\min\{r,\frac{(1+r)^2}{\log(t)}\} \lesssim K^{\frac{1}{1+r}}t^{-\frac{1}{1+r}}W.
\end{align*}
Adding up the contributions, noticing that the last dominates, now implies the result.
\end{proof}

\begin{prop}\label{prop: distilling} Let $t_{\text{main}} :=  \left(1-\frac{1}{\log(t)}\right)t$, $t_{\text{rest}} := \frac{t}{\log(t)}$. Let $C_0$, $L := 4C_0 \log^2(t_{\text{main}})$, $b':=C_0 \log(t_{\text{main}})$ and $f: \mathbb{Z}_{\geq 0} \to \mathbb{R}$ be as in the proof of Theorem 11 of \cite{polyanskiy_2019_dualizing}. Set\begin{align*}
    H_{\text{d}}(N_s) := \sum_{k=k_{\text{min}}}^{k_{\text{max}}} f(k)\binom{N_s}{k}\binom{\lceil t_{\text{main}}\rceil+\lceil t_{\text{rest}}\rceil-N_s}{\lceil t_{\text{main}}\rceil-k}\frac{\lceil t_{\text{main}}\rceil!\lceil t_{\text{rest}}\rceil!}{(\lceil t_{\text{main}}\rceil+\lceil t_{\text{rest}}\rceil)!},
\end{align*} where
\begin{align*}
    k_{\text{min}}&:=\max\{0,N_s-\lceil t_{\text{rest}}\rceil,N_s-\lceil b'\rceil+1\}, \\
    k_{\text{max}} &:= \min\{N_s,\lceil L\rceil-1,\lceil t_{\text{main}}\rceil\}.
\end{align*} Then,
\begin{align*}
    \mathcal{R}_B(\hat{S}_{t,T}^{(H_\text{d})},S_{t,T}) \lesssim B^2t^{-\frac{2}{1+r}}\log^{4}(t).
\end{align*}
\end{prop}

\begin{proof}
Let $\IE_{\text{Po}}$ denote expectation under the usual Poissonized sampling and $\IE_{\text{Fix}}$ expectation under binomial sampling with $\tilde{n} = \lceil t_{\text{main}}\rceil$, $\tilde{n}' = \lceil t_{\text{rest}}\rceil$, $n = \lceil t_{\text{main}}\rceil+\lceil t_{\text{rest}}\rceil $ samples. Let $\hat{S}^{(\text{PW})}_{t,T}$ denote the estimator of \cite{polyanskiy_2019_dualizing}. By the triangle inequality, 
\begin{align}
\nonumber|\IE_{\text{Po}}[\hat{S}_{t,T}^{(\text{PW})}]-\IE_{\text{Po}}[\hat{S}_{t,T}^{(H_\text{d})}]| &\leq \\
 &|\IE_{\text{Po}}[\hat{S}_{t,T}^{(\text{PW})}]-\IE_{\text{Fix}}[\hat{S}_{t,T}^{(\text{PW})}]| \label{first term distillation} \\
    &+ |\IE_{\text{Fix}}[\hat{S}_{t,T}^{(\text{PW})}]-\IE_{\text{Fix}}[\hat{S}_{t,T}^{(H_\text{d})}]| \label{second term distillation}\\
    &+ |\IE_{\text{Fix}}[\hat{S}_{t,T}^{(H_\text{d})}]-\IE_{\text{Po}}[\hat{S}_{t,T}^{(H_\text{d})}]| \label{third term distillation}
\end{align}
The estimator is chosen so that (\ref{second term distillation}) vanishes, which we now demonstrate. Using that the data is exchangeable, the expectation does not change if we permute the data. We use $\hat{S}_{t,T}^{(\sigma)}$ for $\hat{S}_{t,T}^{(\text{PW})}$ applied to the data after permuting the observations by $\sigma$. By exchangeability,

\begin{align*}
\IE_{\text{Fix}}[\hat{S}_{t,T}^{(\text{PW})}] &= \frac{1}{n!}\sum_{\sigma \in \mathfrak{S}_n}\IE[\hat{S}_{t,T}^{(\sigma)}] \\
&=\frac{1}{n!}\IE[\sum_{\sigma \in \mathfrak{S}_n}\hat{S}_{t,T}^{(\sigma)}].
\end{align*}
We now compute the sum. Namely, recalling the definition of $\hat{S}_{t,T}^{(\text{PW})}$, we see that we need to count how many permutations put $k$ observations of $s$ among the first $\tilde{n}$ samples, while respecting that there should be fewer than $L$ observations of $s$ among the first $\tilde{n}$ observations and there should be fewer than $b'$ observations of $s$ among the remaining observations. For feasible $k$ we construct such a permutation by choosing $k$ among the $N_s$ observations to go into the block of size $\tilde{n}$. Then we choose $\tilde{n}-k$ of the $n-N_s$ observations of species other than $s$ to fill the first block. Finally we need to account for the permutations of the elements within the blocks. For feasibility we need (i) the first choice to make sense, so $0 \leq k \leq N_s$, (ii) to respect the condition of fewer than $L$ observations of $s$ within the first $\tilde{n}$ we need $k < L$, (iii) we then need to be able to make the choice of the non-observations, so we need $0 \leq \tilde{n}-k \leq n-N_s$. (iv) Finally, we need that there are not too many leftover observations of $s$, that is, we need $N_s-k < b'$. Thus we obtain

\begin{align*}
    &\frac{1}{n!}\sum_{\sigma \in \mathfrak{S}_n}\hat{S}_{t,T}^{(\sigma)} \\ 
    &= \sum_{s \in \mathcal{S}}\sum_{k=k_{\text{min}}}^{k_{\text{max}}} f(k)\binom{N_s}{k}\binom{\lceil t_{\text{main}}\rceil+\lceil t_{\text{rest}}\rceil-N_s}{\lceil t_{\text{main}}\rceil-k}\frac{\lceil t_{\text{main}}\rceil!\lceil t_{\text{rest}}\rceil!}{(\lceil t_{\text{main}}\rceil+\lceil t_{\text{rest}}\rceil)!}
\end{align*}

\begin{align*}
    &\IE_{\text{Fix}}[\hat{S}_{t,T}^{(\text{PW})}] \\
    &= \IE_{\text{Fix}}\left[\sum_{s \in \mathcal{S}}\underbrace{\sum_{k=k_{\text{min}}}^{k_{\text{max}}} f(k)\binom{N_s}{k}\binom{\lceil t_{\text{main}}\rceil+\lceil t_{\text{rest}}\rceil-N_s}{\lceil t_{\text{main}}\rceil-k}\frac{\lceil t_{\text{main}}\rceil!\lceil t_{\text{rest}}\rceil!}{(\lceil t_{\text{main}}\rceil+\lceil t_{\text{rest}}\rceil)!}}_{=:H_{\text{d}}(N_s)}\right]\\
    &= \IE_{\text{Fix}}[\hat{S}_{t,T}^{(H_\text{d})}].
\end{align*} 

We turn our attention to (\ref{first term distillation}). It is handled in a way inspired by the Proof of Lemma 8 of \cite{orlitsky_2016_optimal}. Let $\tilde{N}_s$ denote the number of observations of $s$ among the first $\tilde{n}$ samples. We note that

\begin{align*}
    &|\IE_{\text{Po}}[f(\tilde{N}_s)1_{\tilde{N}_s < L}1_{N_s-\tilde{N}_s < b'}]-\IE_{\text{Fix}}[f(\tilde{N}_s)1_{\tilde{N}_s < L}1_{N_s-\tilde{N}_s < b'}]| \\
    &\leq 2||f||_\infty d_{\text{TV}}(\text{Po}(t_{\text{main}}M_s)\times\text{Po}(t_{\text{rest}}M_s),\text{Bi}(\lceil t_{\text{main}}\rceil,M_s)\times \text{Bi}(\lceil t_{\text{rest}}\rceil,M_s))
\end{align*}
By the triangle inequality for $d_{\text{TV}}$,

\begin{align*}
    d_{\text{TV}}(\text{Po}(M_s t_{\text{main}}),\text{Bi}(\lceil t_{\text{main}}\rceil,M_s)) &\leq d_{\text{TV}}(\text{Po}(M_s t_{\text{main}}),\text{Po}(M_s \lceil t_{\text{main}}\rceil)), \\
    &+ d_{\text{TV}}(\text{Po}(M_s \lceil t_\text{main} \rceil ),\text{Bi}(\lceil t_\text{main} \rceil,M_s)), \\
    &\leq |M_s t_{\text{main}}-M_s\lceil t_{\text{main}}\rceil|+M_s, \\
    &\leq 2M_s.
\end{align*}
where the first total variation bound is immediate from the natural coupling and the second comes from Theorem 1 of \cite{barbour_1984_on}. The above bound holds also with $t_{\text{rest}}$ replacing $t_{\text{main}}$. Combining these bounds with the fact that for probability measures $\mathbb{P}_1,\mathbb{P}_2,\mathbb{P}_1',\mathbb{P}_2'$ $d_{\text{TV}}(\mathbb{P}_1 \times \mathbb{P}_2,\mathbb{P}_1' \times \mathbb{P}_2') \leq d_{\text{TV}}(\mathbb{P}_1,\mathbb{P}_1')+d_{\text{TV}}(\mathbb{P}_2,\mathbb{P}_2')$ we have

\begin{align*}
    |\IE_{\text{Po}}[f(\tilde{N}_s)1_{\tilde{N}_s < L}1_{N_s-\tilde{N}_s < b'}]-\IE_{\text{Fix}}[f(\tilde{N}_s)1_{\tilde{N}_s < L}1_{N_s-\tilde{N}_s < b'}]| \leq 8||f||_\infty M_s.
\end{align*}
Summing over $s$ we get that (\ref{first term distillation}) is bounded by  $8||f||_\infty$. Term (\ref{third term distillation}) is handled similarly. Since $||f||_\infty$ is of lower order than (107) of \cite{polyanskiy_2019_dualizing} we get a bias bound of the correct order for $\hat{S}_{t,T}^{(H_\text{d})}$. For the variance, note that $||H_\text{d}||_\infty \leq ||f||_\infty$ so that the variance is of the right order for the same reasons as in Theorem 11 of \cite{polyanskiy_2019_dualizing}. To move to $B > 1$ go via Theorem \ref{thm: opt powerhouse}.
\end{proof}

\begin{lem}\label{lem: concentration of variance estimator}
    Let $\mu \prec 1$ and let $\hat{S}_{t,T}^{(H)}$ be a linear estimator with $\IE[\hat{S}_{t,T}] \geq 0$. Then,
    $$\mathbb{V}[\hat{V}_t] \leq (1+||H||_\infty)^2\IE[\hat{V}_t].$$
\end{lem}
\begin{proof}
\begin{align*}
    \mathbb{V}[\hat{V}_t] &= \sum_{s \in \mathcal{S}}\mathbb{V}[H(N_s)+H(N_s)^2], \\
    &\leq \sum_{s \in \mathcal{S}}\IE[H(N_s)^2+2H(N_s)^3+H(N_s)^4], \\& \leq\sum_{s \in \mathcal{S}}\IE[H(N_s)^2+2||H||_\infty H(N_s)^2+||H||_{\infty}^2H(N_s)^2], \\
    &\leq (1+||H||_\infty)^2\sum_{s \in \mathcal{S}}\IE[H(N_s)^2], \\
    &\leq (1+||H||_\infty)^2\IE[\hat{V}_t],
\end{align*} where we used the non-negativity hypothesis in the last line.
\end{proof}
\begin{proof}[Proof of Corollary \ref{cor: GT Gaussianity}, Theorem \ref{thm: CLT Intermediate}]
    We begin by showing the convergence
    \begin{align*}
        \frac{S_{t,T}-\hat{S}_{t,T}^{(H)}-\IE[S_{t,T}-\hat{S}_{t,T}^{(H)}]}{\sqrt{\mathbb{V}[S_{t,T}-\hat{S}_{t,T}^{(H)}]}} \rightarrow\mathcal{N}(0,1).
    \end{align*}
Let $X_s = 1_{N_s=0}1_{N_s'>0}-H(N_s)$ be the error we make in species $s$.
By independence of the species we only need to check Lindeberg's condition. That is, we need to show that for any $\epsilon > 0$ we have

\begin{align*}
    \frac{1}{\mathbb{V}[S_{t,T}-\hat{S}_{t,T}]}\sum_{s \in \mathcal{S}}\IE[(X_s-b_s)^21_{|X_s-b_s|^2>\epsilon \mathbb{V}[S_{t,T}-\hat{S}_{t,T}]}] \rightarrow 0.
\end{align*}
Since $|X_s|\leq ||H||_\infty + 1$ and hence $|b_s| \leq ||H||_\infty + 1$ we have $|X_s-b_s|^2 \leq 4(||H||_\infty+1)^2$ but since $\frac{||H||_\infty^2+1}{\mathbb{V}[S_{t,T}-\hat{S}_{t,T}]} \rightarrow 0$ we therefore see that there exists a time, depending on $\epsilon$, such that each indicator is the indicator of the impossible event. Thus the sum vanishes and the condition is satisfied. 

To get the convergence with the variance proxy it suffices, by Slutsky's theorem, if we can also show that $$\frac{\hat{V}_t}{\mathbb{V}[S_{t,T}-\hat{S}_{t,T}^{(H)}]} \rightarrow 1,$$ in probability. Straightforward computations yield
\begin{align*}
    \mathbb{V}[S_{t,T}-\hat{S}_{t,T}^{(H)}] &= \IE[S_{t,T}]+\IE[\sum_{s \in \mathcal{S}}H(N_s)^2]-\sum_{s \in \mathcal{S}}b_s^2, \\
    &=\IE[\hat{V}_t]+\sum_{s \in \mathcal{S}}b_s-\sum_{s \in \mathcal{S}}b_s^2. 
\end{align*}
Under our assumptions we therefore have
\begin{align} \label{assympotic unbiasedness variance estimator}
    \frac{\IE[\hat{V}_t]}{\mathbb{V}[S_{t,T}-\hat{S}_{t,T}^{(H)}]} \rightarrow 1.
\end{align}
Combining (\ref{assympotic unbiasedness variance estimator}), Lemma \ref{lem: concentration of variance estimator} and Chebyshev's inequality,
    \begin{align*}
        P(|\frac{\hat{V}_t}{\IE[\hat{V}_t]}-1| \geq \epsilon) &\leq \frac{\mathbb{V}[\hat{V_t}]}{\epsilon^2 \IE[\hat{V}_t]^2}, \\
        &\leq \frac{(1+||H||_\infty)^2}{\epsilon^2 \IE[\hat{V}_t]}, \\
        &= \frac{(1+||H||_{\infty})^2}{\epsilon^2 \mathbb{V}[S_{t,T}-\hat{S}_{t,T}^{(H)}]}\frac{\mathbb{V}[S_{t,T}-\hat{S}_{t,T}^{(H)}]}{\IE[\hat{V}_t]}\\
        &\rightarrow 0.
    \end{align*}
For Corollary \ref{cor: GT Gaussianity} notice that $b_s \equiv 0$ and that when $r \leq 1$ we have $||H||_{\infty} \leq 1$. 
\end{proof}

\begin{proof}[Proof of Proposition \ref{prop: rate lower}]

Define the map \begin{align*}
\kappa_B: \mathcal{S} &\to 2^{\mathcal{S} \times \{1,\ldots,B\}} \\
s&\mapsto \{(s,1),(s,2),\ldots,(s,B)\} 
\end{align*}
From an estimator $\hat{S}_{t,T}$ defined when $\mu \prec B$ we create an estimator defined on $1 \prec \mu \prec 1$ as follows

$$\hat{S}_{t,T}^{{(\text{CL})}}(X_1,X_2,\ldots,X_N) := \frac{1}{B}\hat{S}_{t,T}(\kappa_B(X_1),\kappa_B(X_2),\ldots,\kappa_B(X_N)).$$

Now, clearly

\begin{align*}
    \sup_{\mu \prec B}\IE[|S_{t,T}-\hat{S}_{t,T}|^2] &\geq \sup_{\mu \prec B: \mu \text{ arises from some } \mu \prec 1 \text{ by speciation}}\IE[|S_{t,T}-\hat{S}_{t,T}|^2] \\
    &= B^{2}\sup_{\mu \prec 1} \IE[|S_{t,T}-\hat{S}_{t,T}^{(\text{CL})}|^2] \\
    &\gtrsim B^{2} r^2t^2 t^{-\frac{2}{1+r}}\log^{-4}(t)
\end{align*} by Theorem 11 of \cite{polyanskiy_2019_dualizing}.
\end{proof}

\subsection{Distant Future}

\begin{proof}[Proof of Lemma \ref{lem: new concentration}] We verify the hypothesis of Theorem 2.1 of \cite{bartroff_2018_bounded}. It is stated only for a finite index set, but this is no obstacle. We have the correspondence $Y = S_t^{(i)}$, occupancy counts $N_s$ (or rather, for notational compatibility $N_x$) and unit weights. Note that for $x \in \mathcal{S}$, $N_x$ is Poisson-distributed and hence log-concave and non-negative.
Let $E_t$ be the multiset of sets seen at time $t$. That is, if set $A$ has been observed $j$ times, then the multiplicity of $A$ in $E_t$ is $j$.
For each $x$, associate a probability distribution over sets that contain $x$ to $x$ by
\begin{align*}
    P(Z^{(x)}=A)=\frac{\mu(A)}{M_x}.
\end{align*}
For each $x$, generate $Z^{(x)}_1,Z^{(x)}_2,\ldots$ by sampling from the distribution associated to $x$. Also for each $A \in 2^\mathcal{S}$ generate a Poisson random variable $N_A$ with mean $t\mu(A)$, independent of everything else. Let $E^{(k)}_{x,t}$ be the multiset consisting of $Z^{(x)}_1,Z^{(x)}_2,\ldots,Z^{(x)}_k$ as well as, for each set $A$ which does not contain $x$, $N_A$ copies of $A$. With some thought one sees that, with $\mathcal{L}$ here denoting the law of a random variable,
\begin{align*}
    \mathcal{L}(E_t|N_x=k)=\mathcal{L}(E^{(k)}_{x,t}).
\end{align*}
This is because on both sides the number of observations of sets which do not contain $x$ are generated in the same way and independently of those which contain $x$. Turning to those which contain $x$ we note that conditioning on the sum of independent Poisson random variables makes the joint distribution of the summands multinomial and we choose the distribution associated to each $x$ such that they would be (the correct) multinomial.  

$E^{(k)}_{x,t}$ has associated counts, $$N_{x,y}^{(k)} = \sum_{j=1}^k 1_{y \in Z^{(x)}_j}+\sum_{A \in 2^\mathcal{S}:x \notin A, y \in A}N_A.$$
Moreover, since when we increment $k$ only one more set is added and it can influence at most $B-1$ non-$x$ species, 
\begin{align*}
    \sum_{y \in \mathcal{S}: y \neq x}1_{N^{(k+1)}_{x,y} \geq i}-\sum_{y \in \mathcal{S}: y \neq x}1_{N^{(k)}_{x,y} \geq i}\leq B-1.
\end{align*}
That is, the counts have property $(B-1,\geq)$ in the sense of definition 2.2 of \cite{bartroff_2018_bounded}. Note that the infimum of the support of $N_x$ is zero and we are interested in a threshold $i$. By Theorem 2.1 of \cite{bartroff_2018_bounded} there exists a coupling of $S_t^{(i)}$ to its size biased version so that the coupling difference is bounded by $(B-1)i+1$. Now equation (5) of \cite{bartroff_2018_bounded} and Corollary 1.1 of \cite{bartroff_2018_bounded} combine to give the result.
\end{proof}

\begin{prop} For $i \in \mathbb{Z}_{\geq 0}$
\begin{align*}
    \IE[S_t^{(i+1)}] = (-1)^{i}\frac{t^{i+1}}{i!}\mathcal{L}[\nu]^{(i)}(t).
\end{align*}
In particular, the following identities hold.

\begin{align*}
    \IE[S_t] &= t \mathcal{L}[\nu](t), \\
    \IE[S_t^{(2)}] &= -t^2 \mathcal{L}[\nu]'(t),\\
    \IE[\phi_1] &= t\mathcal{L}[\nu](t)+t^2 \mathcal{L}[\nu]'(t),\\
    \IE[\phi_2] &= -t^2 \mathcal{L}[\nu]'(t)-\frac{t^3}{2} \mathcal{L}[\nu]''(t).
\end{align*}
\end{prop}

\begin{proof}
\begin{align*}
    (-1)^{i}\frac{t^{i+1}}{i!}\mathcal{L}[\nu]^{(i)}(t) &= \frac{t^{i+1}}{i!}\mathcal{L}[x^i \nu](t) \\
    &= \frac{t^{i+1}}{i!} \sum_{s \in \mathcal{S}} \int_0^{M_s} x^i e^{-xt}dx \\
    &= \sum_{s \in \mathcal{S}} \int_0^{M_s} t\frac{(xt)^i}{i!}e^{-xt}dx \\
    &= \sum_{s\in \mathcal{S}}P(N_s \geq i+1) \\
    &= \IE[S_t^{(i+1)}]
\end{align*} where the evaluation of the integral is (2.6) of \cite{integral_evaluation}.
Now write $\phi_1 = S_t^{(1)}-S_t^{(2)}$ and $\phi_2 = S_t^{(2)}-S_t^{(3)}$.
\end{proof}

 In the following we let $c_\alpha := c\Gamma(1-\alpha)$ and use $*$ to denote an \enquote{ideal} version of a quantity. In particular, let

\begin{align*}
    \IE[S_t^*] &:= c\Gamma(1-\alpha)t^\alpha \\
    \IE[\phi_1^*] &:=\alpha c\Gamma(1-\alpha)t^\alpha \\
    \hat{S}_T^* &:=\IE[S_t^*](1+r)^{\frac{\IE[\phi_1^*]}{\IE[S_t^*]}} \\
    \IE[S_t^{(2)*}] &:= (1-\alpha)c\Gamma(1-\alpha)t^\alpha \\
    \IE[\phi_2^*] &:= \frac{\alpha(1-\alpha)}{2}c\Gamma(1-\alpha)t^\alpha
\end{align*}

 From the above we have the following Corollary.
 
\begin{cor}\label{cor: Laplace neo} If $|\nu(x)-cx^{-\alpha}| \leq f(x)$ then
    \begin{align*}
        |\IE[S_t]-\IE[S_t^*]| &\leq t \mathcal{L}[f](t). \\
        |\IE[S_t^{(2)}]-\IE[S_t^{(2)*}]| &\leq -t^2\mathcal{L}[f]'(t). \\
        |\IE[\phi_1]-\IE[\phi_1^*]| &\leq t \mathcal{L}[f](t)-t^2\mathcal{L}[f]'(t)\\
        |\IE[\phi_2]-\IE[\phi_2^*]| &\leq -t^2\mathcal{L}[f]'(t)+\tfrac{t^3}{2}\mathcal{L}[f]''(t).
    \end{align*}
\end{cor}

\begin{proof}[Proof of Theorem \ref{thm: far future main}]
\begin{align*}
    |\hat{S}_{t,T}-S_{t,T}| &= |\hat{S}_T-S_T| \\
    &\leq |\hat{S}_T-\hat{S}_T^*|+|\hat{S}_T^*-\IE[S_T^*]|+|\IE[S_T^*]-\IE[S_T]|+|\IE[S_T]-S_T|
\end{align*}
From the definitions of the idealized quantities the second term is identically zero. Consider the event $A_1 := |\IE[S_T]-S_T| < qz - T\mathcal{L}[f](T)$. By Corollary \ref{cor: Laplace neo} we have that on $A_1$, the sum of the last two terms is upper bounded by $qz$. 
Let $d=d(z)$ be given by,
\begin{align*}  
&\frac{pz\IE[S_t^*]^2(1+r)^{\frac{\IE[{S_t^{(2)*}}]}{\IE[S_t^*]}}}{(1+r)(1+2\log(1+r))\IE[S_t^*]^2+pz(\IE[S_t^*]+\IE[{S_t^{(2)*}}])(1+r)^{\frac{\IE[{S_t^{(2)*}}]}{\IE[S_t^*]}}\log(1+r)}, \\
&=\frac{pz c\Gamma(1-\alpha)t^\alpha(1+r)^{-\alpha}}{(1+2\log(1+r))c\Gamma(1-\alpha)t^\alpha+pz(2-\alpha)\log(1+r)(1+r)^{-\alpha}}.
\end{align*}
Let $A_2$ be the event $|S_t - \IE[S_t]| < d-t\mathcal{L}[f](t)$. Note that on $A_2$, we have $|S_t - \IE[S_t^*]| < d$. Similarly, let $A_3$ be the event that $|S_t^{(2)} - \IE[S_t^{(2)}]| < d+t^2\mathcal{L}[f]'(t)$ which implies $|S_t^{(2)} - \IE[{S_t^{(2)}}^*]| < d$.

\begin{align*}
    |\hat{S}_T-\hat{S}_T^*| &\leq |S_t(1+r)^{\phi_1/S_t} - \IE[S_t^*](1+r)^{\IE[\phi_1^*]/\IE[S_t^*]}| \\
    &= (1+r)|S_t(1+r)^{-S_t^{(2)}/S_t}-\IE[S_t^*](1+r)^{-\IE[S_t^{(2)*}]/\IE[S_t^*]}|.
\end{align*}
Computing the partial derivatives of $g(x,y) := x(1+r)^{-y/x}$ gives
\begin{align*}
    \partial_x g(x,y) = (1+r)^{-y/x}(1+\frac{y}{x} \log(1+r)) \\
    \partial_y g(x,y) = -\log(1+r)(1+r)^{-y/x}
\end{align*}
On the event $A_2 \cap A_3$ we therefore have
\begin{align*}
    |\hat{S}_T-\hat{S}_T^*| \leq (1+r)(1+r)^{-\frac{\IE[S_t^{(2)*}]-d}{\IE[S_t^*]+d}}(1+2\log(1+r))d,
\end{align*}
where we used the hypothesis that we are on $A_2 \cap A_3$ both to bound the derivatives and to bound the deviation in each variable.

Applying Lemma \ref{lem: trans eq} with $x = \IE[S_t^*],y = \IE[{S_t^{(2)*}}],c = r+1, k = (1+r)(1+2\log(r+1))$ and $z$ of the lemma being $pz$ in the current scope now shows this to be upper bounded by $pz$. Thus on $A_1 \cap A_2 \cap A_3$ we have $|S_T - \hat{S}_T| \leq pz+qz = z$ and consequently,

\begin{align*}
    P(|S_{t,T}-\hat{S}_{t,T}|>z) &\leq P(\overline{A_1 \cap A_2 \cap A_3}), \\
    &\leq P(\overline{A_1})+P(\overline{A_2})+P(\overline{A_3}), \\
    &\leq P(|S_T-\IE[S_T]| \geq qz - T\mathcal{L}[f](T)) \\
    &+ P(|S_t - \IE[S_t]| \geq d-t\mathcal{L}[f](t)) \\
    &+P(|S_t^{(2)} - \IE[S_t^{(2)}]| \geq d+t^2\mathcal{L}[f]'(t)).\\
\end{align*}
The theorem now follows from applying Lemma \ref{lem: new concentration}.
\end{proof}

\begin{lem} \label{lem: trans eq}
Suppose that $c > 1, x,y,z,k > 0$.
Then, with $$d:=\frac{x^2zc^{y/x}}{kx^2+(x+y)zc^{y/x}\log(c)},$$
we have the inequality,
\begin{align*}
    c^{-\frac{y-d}{x+d}}kd < z.
\end{align*}
\end{lem}
\begin{proof}

Rearranging we may equivalently show that

\begin{align*}
    d < \frac{z}{k}c^{\frac{y-d}{x+d}}.
\end{align*}
Define $d_0 = \frac{z}{k}c^{y/x}$. Then,

\begin{align*}
    \frac{z}{k}c^{\frac{y-d}{x+d}} &= d_0c^{\frac{y-d}{x+d}-\frac{y}{x}}\\
    &= d_0 \exp(-\frac{(x+y)\log(c)}{x}\frac{d}{x+d})
\end{align*}
Taking logarithms we thus need to show
\begin{align*}
    \frac{(x+y)\log(c)}{x}\frac{d}{d+x} < \log(\frac{d_0}{d})
\end{align*}
Defining $\lambda := \frac{(x+y)d_0\log(c)}{x^2} = \frac{d_0}{d}-1$ this is equivalent to
\begin{align*}
    \frac{\lambda x}{x(1+\lambda)+d_0} < \log(1+\lambda),
\end{align*}
but this is immediate as $\frac{\lambda}{1+\lambda}$ is intermediate between these two quantities. It is larger than the first since $d_0 > 0$ and smaller than the second since it is a standard property of the logarithm that $\frac{u}{1+u} \leq \log(1+u)$.
\end{proof}
 
\begin{proof}[Proof of Corollary \ref{cor: far future rate}]

Pick $p=1/2$ arbitrarily and choose $z = \sqrt{CB}r^\alpha t^{\alpha/2}\log(r)$ in Theorem \ref{thm: far future main}. Notice that $|\nu(x)-cx^{-\alpha}| \leq Kx^{-\alpha/2}$ for $x \in (0,1)$ implies  $|\nu(x)-cx^{-\alpha}| \leq \max\{{K,c}\}x^{-\alpha/2}$ for all $x > 0$. Here asymptotic notation will hide constants that depend on $\alpha,c,K,r_0$ and an index $i$ but which are independent of $C$ and $B$.

\begin{align*}
d(z) \asymp \frac{\sqrt{CB} t^{3\alpha/2}\log(r)}{t^{\alpha}\log(r)+\sqrt{CB}t^{\alpha/2}\log^2(r)} \asymp \sqrt{CB}t^{\alpha/2},
\end{align*}
where we used that $\log(r) = o(t^{\alpha/2})$.

Since $\mathcal{L}[x^{-\alpha/2}](t) = \Gamma(1-\frac{\alpha}{2})t^{\alpha/2-1}$, we have $|t^{i+1} \mathcal{L}[f]^{(i)}(t)| \lesssim  t^{\alpha/2}$ which implies that (\ref{eq: three hypothesis estimates}) is satisfied for large enough $C$, large enough $t$. By Corollary \ref{cor: Laplace neo}, $\IE[S_t] \asymp \IE[S_t^{(2)}] \asymp t^{\alpha},\IE[S_T] \asymp r^\alpha t^\alpha$. Now the result follows from applying these asymptotics to the arguments of the exponentials, recalling $\log(r) = o(t^{\alpha/2})$.
\end{proof}

\begin{proof}[Proof of Theorem \ref{thm: alpha-rate}]
    For any choice of $u_t,v_t,z_t$ we have
    \begin{align}
        &\mathrel{\phantom{=}} P(\hat{\alpha}-\alpha > z_t) \nonumber \\
        &= P(\hat{\alpha}-\alpha > z_t \wedge |S_t-\IE[S_t]| \geq u_t \wedge |\phi_1-\IE[\phi_1]| \geq v_t) \label{eq:alpha1} \\
        &+P(\hat{\alpha}-\alpha > z_t \wedge |S_t-\IE[S_t]| < u_t \wedge |\phi_1-\IE[\phi_1]| \geq v_t) \label{eq:alpha2} \\
        &+P(\hat{\alpha}-\alpha > z_t \wedge |S_t-\IE[S_t]| \geq u_t \wedge |\phi_1-\IE[\phi_1]| < v_t) \label{eq:alpha3} \\
        &+P(\hat{\alpha}-\alpha > z_t \wedge |S_t-\IE[S_t]| < u_t \wedge |\phi_1-\IE[\phi_1]| < v_t). \label{eq:alpha4}
    \end{align}
\eqref{eq:alpha1} and \eqref{eq:alpha3} are each upper bounded by $P(|S_t-\IE[S_t]| \geq u_t)$ which we can, in turn, upper bound by combining Lemma \ref{lem: new concentration} and Corollary \ref{cor: Laplace neo} as follows: Recall that we have assumed $|\nu(x)-cx^{-\alpha}|\leq Kx^{-\frac{\alpha}{2}}$. Hence, we may apply Corollary \ref{cor: Laplace neo} with $f(x)=\max\{K,c\}x^{-\frac{\alpha}{2}}$, where the second part of the max ensures that the bound holds also for $x \geq 1$. Making $K$ larger if needed we can assume $K>c$. Because of $\mathcal{L}[f](t)=K\Gamma(1-\frac{\alpha}{2})/t^{1-\frac{\alpha}{2}}$ in this case, we get
\begin{align*}
\IE[S_t]\leq c\Gamma(1-\alpha)t^{\alpha}+K\Gamma(1-\frac{\alpha}{2})t^{\frac{\alpha}{2}},
\end{align*}
when recalling that $\IE[S_t^*]=c\Gamma(1-\alpha)t^{\alpha}$. From Lemma \ref{lem: new concentration}, we hence obtain
\begin{align}
&\mathrel{\phantom{=}}P(|S_t-\IE[S_t]|\geq u_t) \nonumber \\
&\leq \exp(-\frac{u_t^2}{2B\IE[S_t]})+\exp(-\frac{u_t^2}{2B\IE[S_t]+\frac{2B}{3}u_t}) \nonumber \\
&\leq \exp\left(-\frac{u_t^2}{2B(c\Gamma(1-\alpha)t^{\alpha}+K\Gamma(1-\frac{\alpha}{2})t^{\frac{\alpha}{2}})}\right) \nonumber \\
&+\exp\left(-\frac{u_t^2}{2B(c\Gamma(1-\alpha)t^{\alpha}+K\Gamma(1-\frac{\alpha}{2})t^{\frac{\alpha}{2}})+\frac{2B}{3}u_t}\right). \label{eq:rate_St}
\end{align}
We conclude that there is a constant $C$ for which the above becomes smaller than $\epsilon/3$ when letting $u_t:=Ct^\frac{\alpha}{2}$.

By Corollary \ref{cor: Laplace neo} we have
\begin{align*}
    \IE[S_t^{(2)}] &\leq  (1-\alpha)c\Gamma(1-\alpha)t^\alpha+K\Gamma\left(1-\frac{\alpha}{2}\right)(1-\frac{\alpha}{2})t^\frac{\alpha}{2}.
\end{align*}
Noticing that $\phi_1 = S_t^{(1)}-S_t^{(2)}$ and applying Lemma \ref{lem: new concentration} gives,
\begin{align*}
& P(|\phi_1-\IE[\phi_1]| \geq v_t) \\
 &\leq \exp\left(-\frac{v_t^2}{8B\IE[S_t]}\right)+ \exp\left(-\frac{v_t^2}{8B\IE[S_t]+\frac{4}{3}Bv_t}\right) \\
    &+\exp\left(-\frac{v_t^2}{8(2B-1)\IE[S_t^{(2)}]}\right)+ \exp\left(-\frac{v_t^2}{8(2B-1)\IE[S_t^{(2)}]+\frac{4(2B-1)}{3}v_t}\right), \\
    &\leq \exp\left(-\frac{v_t^2}{8B(c\Gamma(1-\alpha)t^\alpha+K\Gamma(1-\frac{\alpha}{2})t^\frac{\alpha}{2})}\right) \\
    &+ \exp\left(-\frac{v_t^2}{8B(c\Gamma(1-\alpha)t^\alpha+K\Gamma(1-\frac{\alpha}{2})t^\frac{\alpha}{2})+\frac{4}{3}Bv_t}\right) \\
    &+\exp\left(-\frac{v_t^2}{8(2B-1)((1-\alpha)c\Gamma(1-\alpha)t^\alpha+K\Gamma(1-\frac{\alpha}{2})(1-\frac{\alpha}{2})t^\frac{\alpha}{2})}\right)\\
    &+ \exp\left(-\frac{v_t^2}{8(2B-1)((1-\alpha)c\Gamma(1-\alpha)t^\alpha+K\Gamma(1-\frac{\alpha}{2})(1-\frac{\alpha}{2})t^\frac{\alpha}{2})+\frac{4(2B-1)}{3}v_t}\right).
\end{align*}

The above can be made smaller than $\frac{\epsilon}{3}$ by choosing $v_t := Ct^{\frac{\alpha}{2}}$ after possibly increasing $C$.

We argue next that $\eqref{eq:alpha4}=0$ because the event inside the probability is impossible. Note that on the event inside of the probability in \eqref{eq:alpha4} it  holds that
\begin{align*}
\hat{\alpha}-\alpha&=\frac{\phi_1}{S_t}-\alpha\leq\frac{v_t+\IE[\phi_1]}{-u_t+\IE[S_t]}-\alpha \\
&\leq \frac{Ct^{\frac{\alpha}{2}}+c\Gamma(1-\alpha)\alpha t^{\alpha}+K\Gamma\left(1-\frac{\alpha}{2}\right)\left(2-\frac{\alpha}{2}\right)t^{\frac{\alpha}{2}}}{-Ct^\frac{\alpha}{2}+c\Gamma(1-\alpha)t^{\alpha}-K\Gamma(1-\frac{\alpha}{2})t^{\frac{\alpha}{2}}}-\alpha,
\end{align*}
where we substituted the definitions of $v_t$ and $u_t$ and applied Corollary \ref{cor: Laplace neo}.
 Bringing $\alpha$ into the fraction and dividing numerator and denominator by $t^{\frac{\alpha}{2}}$, we obtain
 \begin{align*}
&\mathrel{\phantom{=}}\hat{\alpha}-\alpha \\
&\leq\frac{(1+\alpha)Ct^{\frac{\alpha}{2}}+K\Gamma\left(1-\frac{\alpha}{2}\right)\left(2+\frac{\alpha}{2}\right)t^{\frac{\alpha}{2}}}{-Ct^\frac{\alpha}{2}+c\Gamma(1-\alpha)t^{\alpha}-K\Gamma\left(1-\frac{\alpha}{2}\right)t^{\frac{\alpha}{2}}} \\
&=\frac{(1+\alpha)C+K\Gamma\left(1-\frac{\alpha}{2}\right)\left(2+\frac{\alpha}{2}\right)}{-C+c\Gamma(1-\alpha)t^{\frac{\alpha}{2}}-K\Gamma\left(1-\frac{\alpha}{2}\right)}.
\end{align*}
One sees that the right hand side above is smaller than $z_t := \frac{D}{t^{\frac{\alpha}{2}}}$ for $t$ sufficiently large if we set $D=D(\alpha,\epsilon,c,K,B)$ sufficiently big. Therefore, $\eqref{eq:alpha4}=0$ and the proof of the upper tail is complete. The other tail is treated similarly.
\end{proof}

We give a proof of Theorem \ref{thm: CLT distant} which is really a reduction to some facts which we will then go on to prove. Let $h(z) := e^z-1-z$ and $U := \beta\phi_1+(1-\alpha\beta)S_t-\rho^{-\alpha}S_T$ and $\Delta := \hat{\alpha}-\alpha$.
\begin{proof}[Proof of Theorem \ref{thm: CLT distant}]
First we provide a convenient way of expressing the prediction error. We note that,
\begin{align}
\begin{split}\label{eq: error expression}
        \rho^\alpha U+\rho^\alpha S_t h(\beta \Delta) &= \rho^\alpha (\beta\phi_1+(1-\alpha\beta)S_t-\rho^{-\alpha}S_T)+\rho^\alpha S_t h(\beta \Delta), \\
        &= \rho^\alpha (\beta(S_t \Delta +\alpha S_t)+(1-\alpha\beta)S_t-\rho^{-\alpha}S_T)+\rho^\alpha S_t h(\beta \Delta), \\
        &= \rho^\alpha (\beta S_t \Delta +S_t -\rho^{-\alpha}S_T)+\rho^\alpha S_th(\beta \Delta), \\
        &= \rho^\alpha \beta S_t \Delta +\rho^\alpha S_t -S_T+\rho^\alpha S_t(\rho^\Delta-\beta \Delta -1), \\
        &= S_t \rho^{\alpha + \Delta} -S_T,\\
        &= \hat{S}_{t,T}-S_{t,T}.
\end{split}
\end{align}
We will show that
    \begin{enumerate}
        \item $\frac{\IE[U]}{\sqrt{\mathbb{V}[U]}} \rightarrow 0.$
        \item $\rho^{2\alpha} \mathbb{V}[U] = V_t(1+o(1))$.
        \item $\frac{U-\IE[U]}{\sqrt{\mathbb{V}[U]}} \rightarrow \mathcal{N}(0,1)$.
        \item $\rho^\alpha S_t h(\beta \Delta) = o_P(\sqrt{V_t})$.
        \item $\hat{V}_t/V_t \overset{P}{\rightarrow} 1$.
    \end{enumerate}
This implies the theorem because in light of (\ref{eq: error expression}) we have the decomposition
\begin{align*}
    \frac{\hat{S}_{t,T}-S_{t,T}}{\sqrt{V_t}} = \underbrace{\frac{U-\IE[U]}{\sqrt{\mathbb{V}[U]}}}_{\rightarrow \mathcal{N}(0,1)} \underbrace{\frac{\rho^\alpha \sqrt{\mathbb{V}[U]}}{\sqrt{V_t}}}_{ \rightarrow 1}+ \underbrace{\frac{\IE[U]}{\sqrt{\mathbb{V}[U]}}}_{\rightarrow 0} \underbrace{\frac{\rho^\alpha \sqrt{\mathbb{V}[U]}}{\sqrt{V_t}}}_{\rightarrow 1}+\underbrace{\frac{\rho^\alpha S_th(\beta \Delta)}{\sqrt{V_t}}}_{\overset{P}{\rightarrow} 0}
\end{align*}
where the convergences are immediate from the first four facts. Slutsky's theorem and the final fact then upgrade this to a CLT with an estimated variance.
\end{proof}

First a construction to transfer our asymptotic estimate on $\nu$ into something global.

\begin{lem}\label{lem: quant nu}
    If $\mu \prec 1$ and $\nu(x)=cx^{-\alpha}+o(x^{-\alpha/2})$ as $x \downarrow 0$, then for any $\epsilon > 0$ there exist $\delta > 0$, $C_\delta < \infty$ and a function $f_\epsilon(x)$ so that $|\nu(x)-cx^{-\alpha}| \leq f_\epsilon(x)$ for all $x > 0$ and
\begin{align*}
    \mathcal{L}[f_\epsilon(x)](t) &= \epsilon\Gamma\left(1-\frac{\alpha}{2}\right)t^{\frac{\alpha}{2}-1}+C_\delta e^{-\delta t}\frac{1}{t}, \\
    \mathcal{L}[f_\epsilon(x)]'(t) &= -\epsilon\Gamma\left(2-\frac{\alpha}{2}\right)t^{\frac{\alpha}{2}-2}-C_\delta e^{-\delta t}\left(\frac{1}{t^2}+\frac{\delta}{t}\right), \\
    \mathcal{L}[f_\epsilon(x)]''(t) &= \epsilon\Gamma\left(3-\frac{\alpha}{2}\right)t^{\frac{\alpha}{2}-3}+C_\delta e^{-\delta t}\left(\frac{2}{t^3}+\frac{2\delta}{t^2}+\frac{\delta^2}{t}\right).
\end{align*}
\end{lem}
\begin{proof}
Pick $\epsilon > 0$. By hypothesis there exists some $\delta$ so that $|\nu(x)-cx^{-\alpha}| \leq \epsilon x^{-\alpha/2}$ for $x \in (0,\delta]$. Note that since the intensities of the species sum to (at most) one, $\nu(x) \leq 1/x$. Then for $x \in (\delta,\infty)$ we have $|\nu(x)-cx^{-\alpha}| \leq 1/\delta+c\delta^{-\alpha} =: C_\delta$. 
\begin{align*}
    |\nu(x)-cx^{-\alpha}| \leq f_{\epsilon}(x) := \epsilon x^{-\alpha/2}+1_{x > \delta}C_\delta
\end{align*}
By standard properties of the Laplace transform, \\ $\mathcal{L}[f_\epsilon(x)](t) = \epsilon\Gamma\left(1-\frac{\alpha}{2}\right)t^{\alpha/2-1}+\frac{C_\delta}{t}e^{-\delta t}$. The final lines of the lemma follow by differentiation.
\end{proof}

\begin{lem} \label{lem: asymptotic moment estimates} If $|\nu(x)-cx^{-\alpha}| = o(x^{-{\alpha/2}})$ as $x \downarrow 0$, then
    \begin{align*}
        \IE[S_t] = \IE[S_t^*]+o(t^{\alpha/2}) \\
        \IE[\phi_1] = \IE[\phi_1^*]+o(t^{\alpha/2}) \\
        \IE[\phi_2] = \IE[\phi_2^*]+o(t^{\alpha/2})
    \end{align*}
\end{lem}
\begin{proof}
Applying Corollary \ref{cor: Laplace neo} we see that the distance to the idealized values may be bounded in terms of the Laplace transform and its derivatives, if we find a global bound on $|\nu(x)-cx^{-\alpha}|$. This bound is supplied by Lemma \ref{lem: quant nu}. Multiplying the $k$th derivative by a monomial of degree $k+1$ we obtain for $k=0,1,2$ that there is a leading term $\epsilon C_\alpha t^{\frac{\alpha}{2}}$, $C_\alpha$ being a constant depending only on $\alpha$ and then an exponentially small second term. 

\begin{align*}
    \limsup_{t \rightarrow \infty}\frac{|\IE[S_t]-\IE[S_t^*]|}{t^{\alpha/2}} \leq \epsilon C_\alpha
\end{align*}
Since $\epsilon$ was arbitrary, the $S_t$ part of the lemma now follows. We treat $\phi_1$ and $\phi_2$ the same way. 
\end{proof}

\begin{cor}\label{cor: S_t convergence}If $\mu \prec 1$ and $\nu(x)=cx^{-\alpha}+o(x^{-\alpha/2})$
    \begin{align*}
    \frac{S_t}{c_\alpha t^\alpha} \overset{P}{\rightarrow} 1
    \end{align*}
\end{cor}
\begin{proof}
Note that by independence of the summands $\mathbb{V}[S_t] \leq \IE[S_t]$. By Chebyshev's inequality,

    \begin{align*}
        P(|\frac{S_t}{\IE[S_t]}-1|>\epsilon) \leq \frac{\mathbb{V}[S_t]}{\epsilon^2 \IE[S_t]^2} \leq \frac{1}{\epsilon^2 \IE[S_t]} \rightarrow 0
    \end{align*}
    By Lemma \ref{lem: asymptotic moment estimates}, $\frac{\IE[S_t]}{c_\alpha t^\alpha} \rightarrow 1$. Combining these facts gives the Corollary.
\end{proof}

The next three lemmas concern themselves with the mean and variance of $U$.

\begin{lem}\label{lem: expectation of U} With $\nu(x) = cx^{-\alpha}+o(x^{-\alpha/2}),$
    $$\IE[U] = o((\beta+1)t^{\alpha/2}).$$
\end{lem}
\begin{proof}
Combining Corollary \ref{cor: Laplace neo}  and Lemma \ref{lem: asymptotic moment estimates} we have
\begin{align*}
    \IE[U] &= \beta \IE[\phi_1]+(1-\alpha\beta)\IE[S_t]-\rho^{-\alpha}\IE[S_T] \\
    &= \beta (\IE[\phi_1^*]+o(t^{\alpha/2}))+(1-\alpha\beta)(\IE[S_t^*]+o(t^{\alpha/2}))-\rho^{-\alpha}(\IE[S_T^*] \\
    &+o(\rho^{\alpha/2}t^{\alpha/2})) \\
    &=  \beta\IE[\phi_1^*]+(1-\alpha\beta)\IE[S_t^*]-\rho^{-\alpha}\IE[S_T^*]+o((\beta+1)  t^{\alpha/2}) \\
    &= c\Gamma(1-\alpha)t^\alpha(\alpha \beta +1-\alpha\beta-1)+o((\beta+1)t^{\alpha/2}) \\
    &= o((\beta+1)  t^{\alpha/2})
\end{align*}
\end{proof}

\begin{lem}\label{lem: variance of U}  With $\nu(x) = cx^{-\alpha}+o(x^{-\alpha/2})$,
    \begin{align*}
        \mathbb{V}[U] = c_\alpha \Xi(r,\alpha) t^\alpha + o((1+\beta^2)t^{\alpha/2})
    \end{align*}
\end{lem}

\begin{proof}
Begin by recalling that when $A \subset B$ are events then $\mathbb{C}[1_A,1_B]=P(A \cap B)-P(A)P(B) = P(A)(1-P(B))$. Notice that $\{N_s=1\} \subset \{N_s > 0\} \subset \{N_s+N_s' > 0\}$. By independence of the histories of distinct species we have,
\begin{align*}
    &\mathbb{V}[U] \\
    &=\sum_{s \in \mathcal{S}}\mathbb{V}[\beta 1_{N_s=1}+(1-\alpha \beta)1_{N_s>0}-\rho^{-\alpha}1_{N_s+N_s' > 0}] \\
    &= \sum_{s \in \mathcal{S}}\beta^2\mathbb{V}[1_{N_s = 1}]+(1-\alpha\beta)^2\mathbb{V}[1_{N_s > 0}]+\rho^{-2\alpha}\mathbb{V}[1_{N_s+N_s' >0}] \\
    &+ 2\beta(1-\alpha \beta)\mathbb{C}[1_{N_s=1},1_{N_s>0}]-2\beta \rho^{-\alpha}\mathbb{C}[1_{N_s=1},1_{N_s+N_s'>0}] \\
    &-2(1-\alpha \beta)\rho^{-\alpha}\mathbb{C}[1_{N_s>0},1_{N_s+N_s'>0}] \\
    &=\sum_{s \in \mathcal{S}}\beta^2 M_ste^{-M_st}(1-M_ste^{-M_st})+(1-\alpha\beta)^2e^{-M_st}(1-e^{-M_st})\\
    &+\rho^{-2\alpha}e^{-M_sT}(1-e^{-M_sT}) +2\beta(1-\alpha\beta)M_ste^{-M_st}e^{-M_st}\\
    &-2\beta\rho^{-\alpha}M_ste^{-M_st}e^{-M_sT}-2(1-\alpha \beta)\rho^{-\alpha}(1-e^{-M_st})e^{-M_sT} \\
    &= \beta^2\IE[\phi_1]-\frac{1}{2}\beta^2 \IE[\phi_2(2t)]+ (1-\alpha \beta)^2(\IE[S_{2t}]-\IE[S_t]) \\
    &+\rho^{-2\alpha}(\IE[S_{2\rho t}]-\IE[S_{\rho t}]) \\
    &+\beta(1-\alpha \beta)\IE[\phi_1(2t)]\\
    &-2\beta \frac{\rho^{-\alpha}}{1+\rho}\IE[\phi_1((1+\rho)t)]-2(1-\alpha \beta)\rho^{-\alpha}(\IE[S_{(\rho+1)t}]-\IE[S_{\rho t}]) \\
    &=c_\alpha t^\alpha(\alpha\beta^2-\beta^2\alpha(1-\alpha)2^{\alpha-2}+(1-\alpha\beta)^2(2^\alpha-1)+\rho^{-\alpha}(2^\alpha-1) \\
    &+\beta(1-\alpha\beta)\alpha 2^\alpha -2\beta \frac{\rho^{-\alpha}}{1+\rho}\alpha (1+\rho)^\alpha-2(1-\alpha\beta)\rho^{-\alpha}((\rho+1)^\alpha-\rho^\alpha)) \\
    &+o((1+\beta^2)t^{\alpha/2})
\end{align*}
\end{proof}

\begin{lem} \label{lem: lower bound on Xi} With $\mu$ as above,
    \begin{align*}
        \Xi(r,\alpha) \geq \alpha(1-\alpha)(1-2^{\alpha-2})\beta^2.
    \end{align*}
\end{lem}
\begin{proof}

Collecting terms by power of $\beta$ gives\begin{align*}
    \Xi(r,\alpha) =&  \alpha(1-\alpha)(1-2^{\alpha-2})\beta^2 \\
    +& 2\alpha(\rho^{1-\alpha}(1+\rho)^{\alpha-1}-2^{\alpha-1})\beta \\
    +& (2^{\alpha}-1)(1+\rho^{-\alpha})-2(\rho^{-\alpha}(1+\rho)^{\alpha}-1)
\end{align*} Since $\beta \geq 0$ it suffices to show that the coefficients of the linear and constant terms are non-negative. First we show \begin{align} \label{eq: non-neg 1}
        2\alpha(\rho^{1-\alpha}(1+\rho)^{\alpha-1}-2^{\alpha-1})\beta \geq 0. \end{align} Set $f(\rho) := \rho^{1-\alpha}(1+\rho)^{\alpha-1}-2^{\alpha-1}$ and notice that $f(1) =0$. Calculation yields $f'(\rho) = \frac{1-\alpha}{(1+\rho)^2}(\frac{\rho}{1+\rho})^{-\alpha} \geq 0$ since $\alpha \in (0,1)$ and $\rho \geq 1$ yielding (\ref{eq: non-neg 1}). Next we show that \begin{align}\label{eq: non-neg 2}
            \tilde{f}(\rho) :=(2^{\alpha}-1)(1+\rho^{-\alpha})-2(\rho^{-\alpha}(1+\rho)^{\alpha}-1) \geq 0.
        \end{align}
        Notice that with $g(x) = x^\alpha$ we have
        \begin{align*}
  \rho^{\alpha}\tilde{f}(\rho) &= \rho^\alpha [(2^\alpha-1)(1+\rho^{-\alpha})-2(\rho^{-\alpha}(1+\rho)^\alpha-1)],\\
  &= (2^\alpha-1)(\rho^\alpha+1)-2(1+\rho)^\alpha+2\rho^\alpha,\\
  &= 2^\alpha \rho^\alpha +2^\alpha-\rho^\alpha-1-2(1+\rho)^\alpha +2\rho^\alpha \\
  &=\underbrace{[g(2\rho)-g(1+\rho)]}_{ \geq 0 \text{ by monotonicity}}
  +\underbrace{[g(\rho)-g(1)]-[g(1+\rho)-g(2)]}_{\geq 0 \text{ by concavity}}
\end{align*} This implies (\ref{eq: non-neg 2}). Now since $\beta \geq 0$ owing to $\rho \geq 1$, the lemma follows from (\ref{eq: non-neg 1}) and (\ref{eq: non-neg 2}).
\end{proof}

Now comparing Lemma \ref{lem: expectation of U} with the lower bound on the variance coming from combining Lemma \ref{lem: variance of U} and Lemma \ref{lem: lower bound on Xi} gives fact one once we recall that $\liminf_{t \rightarrow \infty}r(t) > 0$. 

Let $V_t := \rho^{2\alpha} c_\alpha \Xi(r,\alpha)t^\alpha$. Combining Lemma \ref{lem: variance of U}, \ref{lem: lower bound on Xi} and the hypothesis that $\liminf_t r(t) > 0$, we obtain fact 2.

\begin{lem} \label{lem: CLT for U}
    \begin{align*}
        \frac{U-\IE[U]}{\sqrt{\mathbb{V}[U]}} \rightarrow \mathcal{N}(0,1)
    \end{align*}
\end{lem}
\begin{proof}
    Write $U = \sum_{s \in \mathcal{S}}\beta1_{N_s = 1}+(1-\alpha\beta)1_{N_s > 0}-\rho^{-\alpha}1_{N_s+N_s'>0}$ and note that the sum over $\mathcal{S}$ is a sum over independent random variables. Each term is bounded in magnitude by $\beta+|1-\alpha\beta|+\rho^{-\alpha} \leq 2\log(r+1)+2$. Recall that by hypothesis \begin{align*}
        \frac{\log(r)}{t^{\alpha/2}} \rightarrow 0.
    \end{align*} These facts combine with Lemma \ref{lem: lower bound on Xi} and the hypothesis that $r(t)$ is bounded from below to ensure that Lindeberg's condition is satisfied. Namely, eventually each indicator in the condition is the indicator of an impossible event, causing the condition to be satisfied trivially.
\end{proof}

\begin{lem}
    \begin{align*}
        \rho^\alpha S_t h(\beta \Delta) = o_P(\sqrt{V_t})
    \end{align*}
\end{lem}
\begin{proof}
    Since $0\leq h(z) \leq \frac{1}{2}z^2e^{|z|}$,
    \begin{align*}
        0 \leq \rho^\alpha S_t h(\beta \Delta) \leq \frac{1}{2}\rho^\alpha S_t \beta^2 \Delta^2 e^{\beta |\Delta|} = O_P(\rho^\alpha \beta^2)
    \end{align*} To finish notice that, $V_t = \rho^{2\alpha} c_\alpha\Xi(r,\alpha) t^{\alpha} \geq \rho^{2\alpha}c_\alpha C_\alpha \beta^2 t^\alpha$ for some constant $C_\alpha > 0$ depending only on $\alpha$, by Lemma \ref{lem: lower bound on Xi}, and recall that by hypothesis $\frac{\beta^2}{t^\alpha} \rightarrow 0$.
\end{proof}

\begin{lem} \label{lem: Xi sensitivity} There exists a constant $C$, independent of $r$ and $\alpha$, so that
    \begin{align*}
        |\Xi(r,\alpha)-\Xi(r,\hat{\alpha})| \leq C(1+\beta^2)|\alpha-\hat{\alpha}|
    \end{align*}
\end{lem}
\begin{proof}
    Let $||f|| := \sup_{\alpha \in [0,1]}|f(\alpha)|+\sup_{\alpha \neq \hat{\alpha}}\frac{|f(\alpha)-f(\hat{\alpha})|}{|\hat{\alpha}-\alpha|}$ be the Lipschitz algebra norm and recall that, $||1||=1$, $||cf||=|c|\,||f||$, $||f+g|| \leq ||f||+||g||$, $||fg|| \leq ||f||\,||g||$. We bound the norms of the expressions that make up $\Xi(r,\alpha)$. First note that for any $c > 1$ independent of $\alpha$ we have $\frac{d}{d\alpha}c^\alpha = \log(c)c^\alpha$ and recall that for differentiable functions the maximum of the derivative is the Lipschitz constant. Recall also that $|\log(x)|$ is increasing for $x \geq 1$.
\begin{align*}
    ||\alpha|| &\leq 1+1=2 \\
    ||\alpha(1-\alpha)|| &\leq \frac{1}{4}+1 \leq 2 \\
    ||2^\alpha|| &\leq 2+2\log(2) \leq 4 \\
    ||2^{\alpha-2}|| &= \frac{1}{4}||2^\alpha|| \leq 1 \\
    ||1-\alpha \beta || &\leq (1+\beta) +  \beta = 1+2\beta\\
    ||(1+\rho^{-1})^\alpha|| &\leq (1+\rho^{-1})+(1+\rho^{-1})\log(1+\rho^{-1}) \leq 2+2=4 \\
    ||\rho^{-\alpha}(2^\alpha-1)|| &\leq 1+  \sup_{\alpha \in [0,1]} |\frac{d}{d\alpha}\rho^{-\alpha}(2^\alpha-1)|\\
    &= 1+\sup_{\alpha \in [0,1]} \rho^{-\alpha}|2^\alpha \log(2)-(2^\alpha-1)\log{\rho}| \\
    &\leq 1+ 2+\sup_{\alpha \in [0,1]}\alpha \rho^{-\alpha}\log(\rho) \\
    &\leq 3+e^{-1} \leq 4
\end{align*}
where we used that $2^\alpha -1 \leq \alpha$ for $\alpha \in (0,1)$ and observed that $\alpha \rho^{-\alpha}\log(\rho) = xe^{-x}|_{x=\alpha \log(\rho)}$.

\begin{align*}
    &||\Xi(r,\alpha)|| \\
    &=||\alpha\beta^2-\beta^2\alpha(1-\alpha)2^{\alpha-2}+(1-\alpha\beta)^2(2^\alpha-1)+\rho^{-\alpha}(2^\alpha-1) \\
    +&\beta(1-\alpha\beta)\alpha 2^\alpha -2\beta \alpha \rho^{-\alpha}(1+\rho)^{\alpha-1}-2(1-\alpha\beta)\rho^{-\alpha}((\rho+1)^\alpha-\rho^\alpha)|| \\
    \leq & 2\beta^2+2\beta^2+(1+2\beta)^2(4+1)+4+\beta(1+2\beta)8+16\beta+10(1+2\beta) \\
    \leq& C(1+\beta^2).
\end{align*}
\end{proof}
Define $\hat{V}_t := \rho^{2\hat{\alpha}}  \Xi(r,\hat{\alpha}) S_t$.
\begin{lem}
    \begin{align*}
        \frac{\hat{V}_t}{V_t} \overset{P}{\rightarrow} 1.
    \end{align*}
\end{lem}
\begin{proof}
    \begin{align*}
        \frac{\hat{V}_t}{V_t} = \rho^{2\Delta} \frac{S_t}{c_\alpha t^\alpha} \frac{\Xi(r,\hat{\alpha})}{\Xi(r,\alpha)}
    \end{align*}
    $\rho^{2\Delta} \overset{P}{\rightarrow} 1$ by Theorem \ref{thm: alpha-rate}.  $\frac{S_t}{c_\alpha t^\alpha} \rightarrow 1$ by Corollary \ref{cor: S_t convergence}. By Lemma \ref{lem: Xi sensitivity}, Lemma \ref{lem: lower bound on Xi} and Theorem \ref{thm: alpha-rate}
    \begin{align*}
        |\frac{\Xi(r,\hat{\alpha})}{\Xi(r,\alpha)}-1| \leq \frac{C(1+\beta^2)}{\Xi(r,\alpha)}|\Delta| \leq C'|\Delta| \rightarrow 0.
    \end{align*} Combining these three results proves the lemma.
\end{proof}

\section{Data and Experiments} \label{sec: Appendix 2: Data}
\subsection{Real World Data} \label{app: data}
In this appendix we describe the real world data used and how it was processed.

\subsubsection{Surnames}
We use the surnames frequency data from the 2010 census \cite{comenetz_2016_frequently}. 

\subsubsection{Butterflies}
We start with the large database of butterfly observation data in  \cite{yu_2025_spatial}. As the model does not account for time evolution of the distribution and geographically distinct sources we sought a subset of the data with many observations over a short time span in a small geographical area. Among those we picked one with a large number of observed species as we consider this the more interesting regime. In this way we settled on using data collected near coordinates $(-5.25,145.27)$, Papua New Guinea between 2008-05-30 and 2008-07-29 by L. Sam and F. Kimbeng and staff of the Binatang Research Center. We note again that the high discovery regime is favorable for $\hat{S}_{t,T}^{(H^*)}$, this has some cherry-picking effect and one should not expect it to perform as well when applied to data collected in less ecologically diverse places.
\subsubsection{Messages}
We use the CollegeMsg temporal network data in \cite{panzarasa_2009_patterns}. Since most people sent many messages in the 193 days over which the data was collected, we sub-sample the data by a factor $100$. Our hope was also that this would help with exchangeability as it limits call-response dynamics.
\subsubsection{Magic the Gathering}
We simulated packs using \cite{joshbirnholz_2025_github}. The associated README claims that \enquote{packs are meant to have cards that match what you'd really find in booster packs from that set}. Despite this we find it likely that the simulator does not reproduce all aspects of pack generation. For convenience and reproducibility our GitHub contains a text file with the generated packs.

\subsubsection{Genes}
We choose the same data set as the first example in \cite{daley_2014_modeling}, namely we test performance using sequence read archive accession number SRX151616, which uses an MDA haploid library for single sperm sequencing. To limit computational cost the authors of \cite{daley_2014_modeling} use a binning approach. Because this interferes with the generalized unseen species structure we instead limit computational cost by restricting ourselves to the first $10^5$ locations on the genome. That is, for each fragment, we keep it only if it hits this initial segment. Because the experiment, by design, covers most locations many times we keep only every 10th fragment to put us in a higher discovery regime. We use GRCh38 as the reference genome.

\subsection{Experiment Results Table Form} \label{app: data tables}
\begin{table}[ht]
  \centering
  \scriptsize
  \setlength{\tabcolsep}{3pt}
  \caption{Absolute percentage error (\%) for \textbf{Hamlet (True Order)} (single run, no std).}
  \label{tab:appendix:hamletTO1K1S1runjson}
  \resizebox{\textwidth}{!}{%
  \begin{tabular}{rcccccccc}
    \toprule
    \% seen & $H^{*}$ & SGT & PW & Trivial & Ratio-$\alpha$ & MLE-$\alpha$ & Pad\'{e}--GT [2,3] & Chebyshev \\
    \midrule
    5.1 & $22.9$ & $44.3$ & $99.5$ & $86.5$ & $5.0$ & $10.8$ & $52.4$ & $48.5$ \\
    10.2 & $2.6$ & $31.0$ & $194.0$ & $78.1$ & $4.2$ & $1.9$ & $24.3$ & $31.3$ \\
    15.3 & $13.5$ & $19.3$ & $91.1$ & $71.1$ & $3.9$ & $0.8$ & $4.8$ & $15.6$ \\
    20.4 & $4.8$ & $4.6$ & $4.2$ & $63.3$ & $4.7$ & $7.5$ & $13.1$ & $5.3$ \\
    25.5 & $0.8$ & $6.3$ & $4.6$ & $58.8$ & $0.9$ & $1.8$ & $2.6$ & $9.3$ \\
    29.6 & $0.2$ & $6.1$ & $31.8$ & $55.0$ & $3.0$ & $0.3$ & $1.9$ & $14.1$ \\
    34.7 & $7.8$ & $6.4$ & $3.4$ & $50.4$ & $3.7$ & $1.3$ & $3.8$ & $19.2$ \\
    39.8 & $1.2$ & $2.7$ & $4.0$ & $44.4$ & $0.8$ & $1.4$ & $1.4$ & $29.5$ \\
    44.9 & $2.2$ & $2.3$ & $16.9$ & $39.8$ & $1.0$ & $1.0$ & $2.2$ & $34.5$ \\
    49.9 & $1.0$ & $0.9$ & $11.8$ & $35.1$ & $0.2$ & $1.6$ & $0.9$ & $41.8$ \\
    \bottomrule
  \end{tabular}
  }
\end{table}

\begin{table}[ht]
  \centering
  \scriptsize
  \setlength{\tabcolsep}{3pt}
  \caption{Absolute percentage error (\%) for \textbf{Butterflies (True Order)} (single run, no std).}
  \label{tab:appendix:butterfliesTO1K1S1runjson}
  \resizebox{\textwidth}{!}{%
  \begin{tabular}{rcccccccc}
    \toprule
    \% seen & $H^{*}$ & SGT & PW & Trivial & Ratio-$\alpha$ & MLE-$\alpha$ & Pad\'{e}--GT [2,3] & Chebyshev \\
    \midrule
    3.3 & $57.5$ & $92.1$ & --- & $97.3$ & $17.8$ & $97.3$ & --- & $94.1$ \\
    10.0 & $4.5$ & $67.5$ & --- & $87.7$ & $23.3$ & $87.7$ & --- & $74.0$ \\
    13.3 & $1.9$ & $58.4$ & $5.4$ & $83.6$ & $23.3$ & $83.6$ & --- & $65.6$ \\
    20.0 & $14.8$ & $37.9$ & $16.5$ & $72.6$ & $26.4$ & $26.4$ & $9.6$ & $45.2$ \\
    23.3 & $6.7$ & $34.4$ & $24.2$ & $68.5$ & $11.7$ & $13.4$ & $22.8$ & $40.2$ \\
    30.0 & $20.1$ & $19.2$ & $21.2$ & $57.5$ & $16.6$ & $18.5$ & $31.5$ & $21.9$ \\
    33.3 & $8.8$ & $23.3$ & $24.0$ & $57.5$ & $3.0$ & $5.5$ & $34.7$ & $23.9$ \\
    40.0 & $4.7$ & $21.7$ & $8.6$ & $53.4$ & $3.6$ & $1.4$ & $15.2$ & $18.2$ \\
    43.3 & $1.9$ & $21.4$ & $7.7$ & $50.7$ & $7.7$ & $4.6$ & $19.7$ & $16.0$ \\
    46.7 & $0.0$ & $18.5$ & $0.1$ & $47.9$ & $10.5$ & $6.7$ & $22.4$ & $13.8$ \\
    \bottomrule
  \end{tabular}
  }
\end{table}

\begin{table}[ht]
  \centering
  \scriptsize
  \setlength{\tabcolsep}{3pt}
  \caption{Absolute percentage error (\%) for \textbf{Messages (True Order)} (single run, no std).}
  \label{tab:appendix:collegemsgTO1K100S1runjson}
  \resizebox{\textwidth}{!}{%
  \begin{tabular}{rcccccccc}
    \toprule
    \% seen & $H^{*}$ & SGT & PW & Trivial & Ratio-$\alpha$ & MLE-$\alpha$ & Pad\'{e}--GT [2,3] & Chebyshev \\
    \midrule
    5.2 & $38.8$ & $67.6$ & $33.9$ & $90.0$ & $7.3$ & $16.0$ & $86.3$ & $73.6$ \\
    10.2 & $77.9$ & $58.2$ & $60.1$ & $82.1$ & $14.5$ & $0.4$ & $71.0$ & $61.7$ \\
    15.4 & $67.2$ & $51.2$ & $74.1$ & $75.7$ & $22.4$ & $8.5$ & $63.6$ & $52.9$ \\
    20.4 & $42.8$ & $41.6$ & $58.3$ & $69.3$ & $19.2$ & $6.9$ & $48.1$ & $42.0$ \\
    25.4 & $31.8$ & $40.4$ & $36.5$ & $64.7$ & $24.2$ & $12.1$ & $38.1$ & $38.6$ \\
    29.5 & $27.8$ & $36.3$ & $31.1$ & $60.0$ & $23.0$ & $11.4$ & $35.7$ & $32.9$ \\
    34.6 & $27.0$ & $30.8$ & $18.4$ & $54.6$ & $20.2$ & $9.8$ & $27.8$ & $25.2$ \\
    39.6 & $22.8$ & $27.9$ & $36.7$ & $50.2$ & $19.9$ & $10.4$ & $25.9$ & $20.5$ \\
    44.7 & $21.1$ & $23.7$ & $37.5$ & $45.4$ & $17.4$ & $9.2$ & $29.2$ & $13.4$ \\
    49.7 & $16.8$ & $20.4$ & $30.7$ & $40.4$ & $14.9$ & $7.5$ & $22.2$ & $7.5$ \\
    \bottomrule
  \end{tabular}
  }
\end{table}

\begin{table}[ht]
  \centering
  \scriptsize
  \setlength{\tabcolsep}{3pt}
  \label{tab:appendix:hamletTO0K1S100runjson}
  \resizebox{\textwidth}{!}{%
  \begin{tabular}{rcccccccc}
    \toprule
    \% seen & $H^{*}$ & SGT & PW & Trivial & Ratio-$\alpha$ & MLE-$\alpha$ & Pad\'{e}--GT [2,3] & Chebyshev \\
    \midrule
    5.1 & $27.4 \pm 2.1$ & $40.2 \pm 0.6$ & $108.2 \pm 9.5$ & $85.4 \pm 0.0$ & $17.5 \pm 0.8$ & $24.8 \pm 0.7$ & $98.3 \pm 10.6$ & $44.5 \pm 0.4$ \\
    10.2 & $16.4 \pm 1.4$ & $20.7 \pm 0.5$ & $82.0 \pm 5.5$ & $76.5 \pm 0.0$ & $10.8 \pm 0.4$ & $16.0 \pm 0.4$ & $344.2 \pm 112.2$ & $22.0 \pm 0.3$ \\
    15.3 & $11.7 \pm 0.9$ & $10.8 \pm 0.5$ & $76.0 \pm 5.2$ & $69.1 \pm 0.0$ & $8.0 \pm 0.3$ & $12.3 \pm 0.3$ & $46.7 \pm 7.4$ & $6.9 \pm 0.4$ \\
    20.4 & $5.8 \pm 0.5$ & $5.8 \pm 0.3$ & $52.6 \pm 3.7$ & $62.6 \pm 0.1$ & $6.0 \pm 0.3$ & $9.8 \pm 0.2$ & $13.8 \pm 1.3$ & $5.0 \pm 0.3$ \\
    25.5 & $4.4 \pm 0.3$ & $3.9 \pm 0.3$ & $40.8 \pm 3.1$ & $56.8 \pm 0.1$ & $4.6 \pm 0.2$ & $8.0 \pm 0.2$ & $7.2 \pm 0.5$ & $12.9 \pm 0.4$ \\
    29.6 & $3.3 \pm 0.3$ & $2.7 \pm 0.2$ & $33.0 \pm 2.4$ & $52.5 \pm 0.1$ & $3.7 \pm 0.2$ & $6.9 \pm 0.2$ & $3.7 \pm 0.3$ & $19.0 \pm 0.4$ \\
    34.7 & $2.1 \pm 0.2$ & $1.9 \pm 0.1$ & $21.9 \pm 1.8$ & $47.4 \pm 0.1$ & $2.8 \pm 0.2$ & $5.6 \pm 0.2$ & $2.0 \pm 0.2$ & $25.6 \pm 0.3$ \\
    39.8 & $2.0 \pm 0.2$ & $1.5 \pm 0.1$ & $22.3 \pm 1.5$ & $42.7 \pm 0.1$ & $2.2 \pm 0.1$ & $4.7 \pm 0.1$ & $1.9 \pm 0.1$ & $31.3 \pm 0.3$ \\
    44.9 & $1.4 \pm 0.1$ & $1.3 \pm 0.1$ & $17.9 \pm 1.4$ & $38.2 \pm 0.1$ & $1.8 \pm 0.1$ & $4.0 \pm 0.1$ & $1.3 \pm 0.1$ & $37.0 \pm 0.3$ \\
    49.9 & $1.5 \pm 0.1$ & $1.1 \pm 0.1$ & $18.1 \pm 1.2$ & $33.9 \pm 0.1$ & $1.5 \pm 0.1$ & $3.4 \pm 0.1$ & $1.1 \pm 0.1$ & $42.0 \pm 0.3$ \\
    \bottomrule
  \end{tabular}
  }
    \caption{Mean absolute percentage error (\%) $\pm$ SEM over 100 permutations for \textbf{Hamlet (Permutation Average)}.}
\end{table}

\begin{table}[ht]
  \centering
  \scriptsize
  \setlength{\tabcolsep}{3pt}
  \label{tab:appendix:butterfliesTO0K1S100runjson}
  \resizebox{\textwidth}{!}{%
  \begin{tabular}{rcccccccc}
    \toprule
    \% seen & $H^{*}$ & SGT & PW & Trivial & Ratio-$\alpha$ & MLE-$\alpha$ & Pad\'{e}--GT [2,3] & Chebyshev \\
    \midrule
    3.3 & $49.9 \pm 2.6$ & $85.6 \pm 1.0$ & --- & $95.0 \pm 0.3$ & $76.9 \pm 8.2$ & --- & --- & $89.2 \pm 0.7$ \\
    10.0 & $32.5 \pm 2.3$ & $60.7 \pm 1.3$ & --- & $84.8 \pm 0.5$ & $50.7 \pm 3.8$ & $71.4 \pm 2.9$ & --- & $68.5 \pm 1.0$ \\
    13.3 & $28.9 \pm 2.1$ & $52.1 \pm 1.3$ & $60.4 \pm 4.8$ & $80.5 \pm 0.5$ & $42.0 \pm 3.0$ & $57.3 \pm 3.0$ & --- & $60.3 \pm 1.1$ \\
    20.0 & $23.9 \pm 1.7$ & $35.6 \pm 1.3$ & $38.9 \pm 2.9$ & $70.9 \pm 0.6$ & $31.2 \pm 2.3$ & $36.8 \pm 2.5$ & --- & $43.0 \pm 1.1$ \\
    23.3 & $21.8 \pm 1.6$ & $29.2 \pm 1.2$ & $36.2 \pm 2.6$ & $66.4 \pm 0.6$ & $27.6 \pm 1.9$ & $30.9 \pm 2.1$ & --- & $35.4 \pm 1.1$ \\
    30.0 & $17.7 \pm 1.3$ & $20.5 \pm 1.1$ & $27.9 \pm 2.2$ & $58.5 \pm 0.6$ & $21.3 \pm 1.5$ & $22.6 \pm 1.5$ & --- & $23.0 \pm 1.1$ \\
    33.3 & $16.9 \pm 1.1$ & $17.3 \pm 1.1$ & $27.8 \pm 2.1$ & $54.9 \pm 0.6$ & $18.3 \pm 1.2$ & $19.5 \pm 1.3$ & $20.4 \pm 1.5$ & $17.9 \pm 1.1$ \\
    40.0 & $13.9 \pm 1.0$ & $13.0 \pm 0.9$ & $25.6 \pm 1.8$ & $47.8 \pm 0.6$ & $14.4 \pm 1.0$ & $15.6 \pm 1.0$ & $13.3 \pm 1.0$ & $10.9 \pm 0.8$ \\
    43.3 & $13.7 \pm 1.0$ & $10.1 \pm 0.9$ & $23.2 \pm 1.6$ & $43.6 \pm 0.6$ & $14.0 \pm 0.9$ & $15.4 \pm 1.0$ & --- & $9.2 \pm 0.7$ \\
    46.7 & $13.0 \pm 1.0$ & $9.6 \pm 0.7$ & $22.9 \pm 1.5$ & $40.2 \pm 0.7$ & $12.7 \pm 0.9$ & $14.3 \pm 1.0$ & --- & $10.0 \pm 0.7$ \\
    \bottomrule
  \end{tabular}
  }
  \caption{Mean absolute percentage error (\%) $\pm$ SEM over 100 permutations for \textbf{Butterflies (Permutation Average)}.}
\end{table}

\begin{table}[ht]
  \centering
  \scriptsize
  \setlength{\tabcolsep}{3pt}
  \label{tab:appendix:collegemsgTO0K100S100runjson}
  \resizebox{\textwidth}{!}{%
  \begin{tabular}{rcccccccc}
    \toprule
    \% seen & $H^{*}$ & SGT & PW & Trivial & Ratio-$\alpha$ & MLE-$\alpha$ & Pad\'{e}--GT [2,3] & Chebyshev \\
    \midrule
    5.2 & $34.9 \pm 2.1$ & $54.2 \pm 0.4$ & $35.7 \pm 1.8$ & $88.4 \pm 0.0$ & $81.5 \pm 2.6$ & $82.6 \pm 2.3$ & --- & $64.7 \pm 0.3$ \\
    10.2 & $18.0 \pm 1.2$ & $31.6 \pm 0.5$ & $26.3 \pm 1.7$ & $78.9 \pm 0.1$ & $51.9 \pm 1.5$ & $56.8 \pm 1.4$ & $114.2 \pm 11.0$ & $42.7 \pm 0.4$ \\
    15.4 & $12.8 \pm 1.1$ & $18.0 \pm 0.6$ & $17.6 \pm 1.2$ & $70.2 \pm 0.1$ & $36.7 \pm 1.2$ & $43.0 \pm 1.1$ & $139.3 \pm 32.6$ & $26.3 \pm 0.5$ \\
    20.4 & $9.6 \pm 0.7$ & $11.1 \pm 0.6$ & $13.7 \pm 1.1$ & $62.8 \pm 0.1$ & $26.7 \pm 0.9$ & $34.0 \pm 0.8$ & $63.6 \pm 17.1$ & $15.0 \pm 0.5$ \\
    25.4 & $8.1 \pm 0.6$ & $7.4 \pm 0.5$ & $10.6 \pm 0.7$ & $56.1 \pm 0.1$ & $19.6 \pm 0.7$ & $27.2 \pm 0.6$ & $51.2 \pm 15.3$ & $6.9 \pm 0.4$ \\
    29.5 & $7.4 \pm 0.6$ & $5.8 \pm 0.3$ & $9.4 \pm 0.7$ & $51.1 \pm 0.1$ & $15.3 \pm 0.6$ & $23.0 \pm 0.5$ & $25.6 \pm 5.3$ & $3.7 \pm 0.3$ \\
    34.6 & $5.9 \pm 0.4$ & $4.5 \pm 0.3$ & $8.1 \pm 0.7$ & $45.5 \pm 0.1$ & $11.5 \pm 0.5$ & $19.1 \pm 0.4$ & $50.5 \pm 34.0$ & $5.6 \pm 0.4$ \\
    39.6 & $4.8 \pm 0.4$ & $4.0 \pm 0.3$ & $8.3 \pm 0.7$ & $40.6 \pm 0.1$ & $8.0 \pm 0.4$ & $15.3 \pm 0.4$ & $11.4 \pm 2.0$ & $9.1 \pm 0.5$ \\
    44.7 & $4.8 \pm 0.3$ & $3.6 \pm 0.3$ & $8.2 \pm 0.6$ & $35.8 \pm 0.1$ & $5.7 \pm 0.4$ & $12.6 \pm 0.3$ & $7.6 \pm 1.1$ & $13.1 \pm 0.5$ \\
    49.7 & $4.4 \pm 0.3$ & $3.3 \pm 0.2$ & $6.2 \pm 0.5$ & $31.3 \pm 0.2$ & $4.7 \pm 0.3$ & $10.8 \pm 0.3$ & $4.7 \pm 0.7$ & $17.3 \pm 0.5$ \\
    \bottomrule
  \end{tabular}
  }
  \caption{Mean absolute percentage error (\%) $\pm$ SEM over 100 permutations for \textbf{Messages (Permutation Average)}.}
\end{table}

\begin{table}[ht]
  \centering
  \scriptsize
  \setlength{\tabcolsep}{3pt}
  \label{tab:appendix:mtgN12TO0K1S100runjson}
  \resizebox{\textwidth}{!}{%
  \begin{tabular}{rcccccccc}
    \toprule
    \% seen & $H^{*}$ & SGT & PW & Trivial & Ratio-$\alpha$ & MLE-$\alpha$ & Pad\'{e}--GT [2,3] & Chebyshev \\
    \midrule
    8.3 & $13.7 \pm 0.0$ & $64.3$ & --- & $86.7 \pm 0.0$ & $59.3 \pm 0.0$ & $86.7 \pm 0.0$ & --- & $66.7$ \\
    16.7 & $7.6 \pm 0.4$ & $40.2 \pm 0.3$ & --- & $74.7 \pm 0.1$ & $39.8 \pm 1.2$ & $42.2 \pm 1.6$ & --- & $40.4 \pm 0.3$ \\
    25.0 & $14.1 \pm 0.6$ & $24.1 \pm 0.4$ & --- & $63.7 \pm 0.1$ & $28.1 \pm 1.1$ & $29.2 \pm 1.0$ & $21.5 \pm 1.8$ & $19.5 \pm 0.5$ \\
    33.3 & $15.8 \pm 0.6$ & $13.9 \pm 0.5$ & $67.8 \pm 1.8$ & $53.8 \pm 0.2$ & $19.5 \pm 0.9$ & $21.3 \pm 0.8$ & $13.6 \pm 1.0$ & $5.4 \pm 0.4$ \\
    41.7 & $15.0 \pm 0.6$ & $7.7 \pm 0.5$ & $52.3 \pm 1.5$ & $44.7 \pm 0.2$ & $13.5 \pm 0.7$ & $15.8 \pm 0.7$ & $9.7 \pm 0.7$ & $10.3 \pm 0.6$ \\
    \bottomrule
  \end{tabular}
  }
  \caption{Mean absolute percentage error (\%) $\pm$ SEM over 100 permutations for \textbf{MTG ($N=12$) (Permutation Average)}.}
\end{table}

\begin{table}[ht]
  \centering
  \scriptsize
  \setlength{\tabcolsep}{3pt}
  \label{tab:appendix:mtgN24TO0K1S100runjson}
  \resizebox{\textwidth}{!}{%
  \begin{tabular}{rcccccccc}
    \toprule
    \% seen & $H^{*}$ & SGT & PW & Trivial & Ratio-$\alpha$ & MLE-$\alpha$ & Pad\'{e}--GT [2,3] & Chebyshev \\
    \midrule
    4.2 & $11.6 \pm 0.0$ & $74.6 \pm 0.0$ & --- & $91.1 \pm 0.0$ & $114.3 \pm 0.0$ & $91.1 \pm 0.0$ & --- & $80.1 \pm 0.0$ \\
    8.3 & $38.0 \pm 0.9$ & $55.2 \pm 0.3$ & --- & $82.9 \pm 0.1$ & $84.9 \pm 2.3$ & $77.5 \pm 1.5$ & --- & $63.3 \pm 0.2$ \\
    16.7 & $35.5 \pm 1.4$ & $30.3 \pm 0.5$ & $89.5 \pm 2.7$ & $68.8 \pm 0.1$ & $47.3 \pm 1.8$ & $51.0 \pm 1.6$ & $138.1 \pm 48.6$ & $38.1 \pm 0.4$ \\
    20.8 & $30.2 \pm 1.4$ & $21.9 \pm 0.5$ & $67.1 \pm 2.2$ & $62.5 \pm 0.2$ & $37.5 \pm 1.4$ & $42.3 \pm 1.3$ & $93.6 \pm 21.5$ & $28.0 \pm 0.5$ \\
    25.0 & $23.4 \pm 1.2$ & $15.7 \pm 0.5$ & $49.4 \pm 2.4$ & $56.8 \pm 0.2$ & $29.8 \pm 1.1$ & $35.5 \pm 1.0$ & $59.3 \pm 5.1$ & $19.5 \pm 0.5$ \\
    29.2 & $17.7 \pm 1.0$ & $10.9 \pm 0.5$ & $35.6 \pm 2.0$ & $51.4 \pm 0.2$ & $24.1 \pm 1.0$ & $30.4 \pm 0.9$ & $40.8 \pm 5.0$ & $12.1 \pm 0.5$ \\
    33.3 & $14.3 \pm 0.9$ & $7.3 \pm 0.4$ & $26.0 \pm 1.8$ & $46.3 \pm 0.2$ & $20.4 \pm 0.9$ & $26.9 \pm 0.8$ & $31.6 \pm 4.4$ & $6.2 \pm 0.4$ \\
    41.7 & $10.0 \pm 0.6$ & $4.5 \pm 0.3$ & $19.3 \pm 1.4$ & $37.3 \pm 0.2$ & $13.4 \pm 0.7$ & $20.2 \pm 0.6$ & $13.4 \pm 1.3$ & $6.0 \pm 0.4$ \\
    45.8 & $8.9 \pm 0.6$ & $5.0 \pm 0.3$ & $19.4 \pm 1.3$ & $33.4 \pm 0.2$ & $10.7 \pm 0.6$ & $17.4 \pm 0.5$ & $8.9 \pm 0.8$ & $9.0 \pm 0.5$ \\
    \bottomrule
  \end{tabular}
  }
  \caption{Mean absolute percentage error (\%) $\pm$ SEM over 100 permutations for \textbf{MTG ($N=24$) (Permutation Average)}.}
\end{table}

\begin{table}[ht]
  \centering
  \scriptsize
  \setlength{\tabcolsep}{3pt}
  \label{tab:appendix:mtgN48TO0K1S100runjson}
  \resizebox{\textwidth}{!}{%
  \begin{tabular}{rcccccccc}
    \toprule
    \% seen & $H^{*}$ & SGT & PW & Trivial & Ratio-$\alpha$ & MLE-$\alpha$ & Pad\'{e}--GT [2,3] & Chebyshev \\
    \midrule
    4.2 & $103.1 \pm 1.9$ & $63.4 \pm 0.2$ & --- & $86.7 \pm 0.0$ & $178.5 \pm 4.1$ & $145.9 \pm 4.2$ & --- & $73.9 \pm 0.1$ \\
    10.4 & $54.5 \pm 2.8$ & $29.1 \pm 0.4$ & $115.6 \pm 3.1$ & $70.2 \pm 0.1$ & $104.5 \pm 2.5$ & $112.4 \pm 2.2$ & $1645.9 \pm 1476.8$ & $45.4 \pm 0.3$ \\
    14.6 & $33.0 \pm 2.0$ & $15.2 \pm 0.6$ & $78.3 \pm 3.1$ & $61.0 \pm 0.2$ & $78.2 \pm 2.1$ & $89.6 \pm 1.8$ & $278.1 \pm 108.8$ & $31.5 \pm 0.4$ \\
    20.8 & $18.9 \pm 1.4$ & $5.8 \pm 0.4$ & $44.2 \pm 2.6$ & $49.9 \pm 0.2$ & $49.9 \pm 1.5$ & $64.7 \pm 1.2$ & $1804.6 \pm 1694.5$ & $17.3 \pm 0.5$ \\
    25.0 & $13.6 \pm 1.0$ & $4.3 \pm 0.3$ & $27.0 \pm 1.9$ & $43.6 \pm 0.2$ & $38.6 \pm 1.0$ & $54.3 \pm 0.9$ & $93.3 \pm 33.3$ & $10.2 \pm 0.4$ \\
    29.2 & $12.4 \pm 1.0$ & $4.7 \pm 0.3$ & $22.0 \pm 1.5$ & $38.2 \pm 0.2$ & $29.4 \pm 0.9$ & $45.7 \pm 0.8$ & $58.0 \pm 24.4$ & $5.6 \pm 0.4$ \\
    35.4 & $11.7 \pm 1.0$ & $5.3 \pm 0.3$ & $18.8 \pm 1.4$ & $31.4 \pm 0.2$ & $19.9 \pm 0.7$ & $35.8 \pm 0.6$ & $95.6 \pm 52.6$ & $3.8 \pm 0.3$ \\
    39.6 & $13.0 \pm 1.0$ & $5.0 \pm 0.3$ & $17.2 \pm 1.3$ & $27.5 \pm 0.2$ & $15.7 \pm 0.6$ & $30.9 \pm 0.5$ & $22.5 \pm 4.1$ & $4.9 \pm 0.3$ \\
    45.8 & $12.2 \pm 0.9$ & $4.1 \pm 0.3$ & $17.9 \pm 1.2$ & $22.4 \pm 0.2$ & $10.7 \pm 0.4$ & $24.7 \pm 0.4$ & $21.4 \pm 6.8$ & $7.3 \pm 0.4$ \\
    47.9 & $10.8 \pm 0.9$ & $3.6 \pm 0.3$ & $16.4 \pm 1.1$ & $21.0 \pm 0.2$ & $9.1 \pm 0.4$ & $22.7 \pm 0.4$ & $19.9 \pm 6.2$ & $7.7 \pm 0.4$ \\
    \bottomrule
  \end{tabular}
  }
  \caption{Mean absolute percentage error (\%) $\pm$ SEM over 100 permutations for \textbf{MTG ($N=48$) (Permutation Average)}.}
\end{table}

\begin{table}[ht]
  \centering
  \scriptsize
  \setlength{\tabcolsep}{3pt}
  \label{tab:appendix:genesTO0K10S100runjson}
  \resizebox{\textwidth}{!}{%
  \begin{tabular}{rcccccccc}
    \toprule
    \% seen & $H^{*}$ & SGT & PW & Trivial & Ratio-$\alpha$ & MLE-$\alpha$ & Pad\'{e}--GT [2,3] & Chebyshev \\
    \midrule
    5.1 & $122.6 \pm 8.1$ & $27.0 \pm 1.5$ & $3328.5 \pm 238.4$ & $64.6 \pm 0.1$ & $33.0 \pm 0.7$ & $24.1 \pm 0.4$ & $71.8 \pm 3.4$ & $29.7 \pm 1.0$ \\
    10.2 & $41.4 \pm 3.2$ & $13.8 \pm 1.0$ & $1936.9 \pm 143.1$ & $57.8 \pm 0.1$ & $19.9 \pm 0.7$ & $23.7 \pm 0.4$ & $91.2 \pm 11.7$ & $12.0 \pm 0.8$ \\
    15.3 & $21.0 \pm 1.6$ & $10.5 \pm 0.8$ & $932.5 \pm 68.8$ & $52.5 \pm 0.2$ & $13.0 \pm 0.7$ & $21.7 \pm 0.4$ & $65.0 \pm 13.5$ & $9.1 \pm 0.7$ \\
    20.4 & $13.8 \pm 1.0$ & $7.8 \pm 0.6$ & $518.2 \pm 40.1$ & $47.7 \pm 0.2$ & $8.6 \pm 0.5$ & $19.0 \pm 0.4$ & $15.4 \pm 1.1$ & $12.1 \pm 0.8$ \\
    25.5 & $7.9 \pm 0.6$ & $6.2 \pm 0.5$ & $319.3 \pm 24.3$ & $43.3 \pm 0.2$ & $6.1 \pm 0.4$ & $16.6 \pm 0.4$ & $8.9 \pm 0.6$ & $16.3 \pm 0.9$ \\
    29.6 & $6.2 \pm 0.5$ & $5.0 \pm 0.4$ & $168.4 \pm 12.2$ & $40.0 \pm 0.2$ & $5.0 \pm 0.3$ & $14.8 \pm 0.3$ & $6.0 \pm 0.4$ & $20.2 \pm 0.9$ \\
    34.7 & $4.6 \pm 0.4$ & $4.3 \pm 0.3$ & $97.7 \pm 7.3$ & $36.2 \pm 0.2$ & $4.1 \pm 0.3$ & $12.9 \pm 0.3$ & $4.4 \pm 0.3$ & $24.1 \pm 0.9$ \\
    39.8 & $4.0 \pm 0.3$ & $3.9 \pm 0.3$ & $46.8 \pm 3.5$ & $32.6 \pm 0.2$ & $3.5 \pm 0.3$ & $11.0 \pm 0.3$ & $4.0 \pm 0.3$ & $29.0 \pm 0.9$ \\
    44.9 & $3.2 \pm 0.2$ & $3.2 \pm 0.2$ & $31.3 \pm 2.2$ & $29.1 \pm 0.2$ & $2.9 \pm 0.2$ & $9.4 \pm 0.3$ & $3.2 \pm 0.2$ & $32.3 \pm 0.8$ \\
    49.9 & $2.8 \pm 0.2$ & $2.8 \pm 0.2$ & $3.6 \pm 0.2$ & $25.7 \pm 0.2$ & $2.5 \pm 0.2$ & $7.9 \pm 0.3$ & $2.8 \pm 0.2$ & $36.1 \pm 0.8$ \\
    \bottomrule
  \end{tabular}
  }
  \caption{Mean absolute percentage error (\%) $\pm$ SEM over 100 permutations for \textbf{Genome Coverage}.}
\end{table}

\end{document}